\documentclass{amsart}
\usepackage{latexsym,amsxtra,amscd,ifthen}
\usepackage{amsfonts}
\usepackage{verbatim}
\usepackage{amsmath}
\usepackage{amsthm}
\usepackage{amssymb}
\usepackage{url}
\usepackage[arrow,matrix]{xy}
\usepackage{graphicx}
\usepackage{tikz}
\usetikzlibrary{cd}

\theoremstyle{plain}

\newtheorem{theorem}{Theorem}
\newtheorem{lemma}[theorem]{Lemma}
\newtheorem{proposition}[theorem]{Proposition}
\newtheorem{corollary}[theorem]{Corollary}

\numberwithin{theorem}{section}
\numberwithin{equation}{theorem}

\theoremstyle{definition}
\newtheorem{definition}[theorem]{Definition}

\newtheorem{example}[theorem]{Example}
\newtheorem{remark}[theorem]{Remark}
\newtheorem{question}[theorem]{Question}
\newtheorem*{question*}{Question}

\newcommand{\fm}{\mathfrak{m}}
\newcommand{\fn}{\mathfrak{n}}

\DeclareMathOperator{\rk}{rk}

\DeclareMathOperator{\Proj}{Proj}

\DeclareMathOperator{\fpdim}{fpd}
\DeclareMathOperator{\fpc}{fpcx}
\DeclareMathOperator{\fpv}{fpv}

\DeclareMathOperator{\fpgldim}{fpgldim}

\newcommand{\fpk}{\textup{fp}\kappa}

\DeclareMathOperator{\fpg}{fpg}

\DeclareMathOperator{\FP}{FP}

\DeclareMathOperator{\End}{End}
\DeclareMathOperator{\Ext}{Ext}
\DeclareMathOperator{\Tor}{Tor}
\DeclareMathOperator{\Hom}{Hom}
\DeclareMathOperator{\pd}{projdim}

\DeclareMathOperator{\GKdim}{GKdim}

\DeclareMathOperator{\gldim}{gldim}

\DeclareMathOperator{\gr}{gr}
\DeclareMathOperator{\Mod}{Mod}

\def\GKgr{\operatorname{gr}}

\xyoption{all}

\begin{document}

\title{Frobenius-Perron theory of endofunctors}

\author{J.M.Chen,  Z.B. Gao,  E. Wicks, J. J. Zhang, X-.H. Zhang and H. Zhu}

\address{Chen: School of Mathematical Sciences,
Xiamen University, Xiamen 361005, Fujian, China}

\email{chenjianmin@xmu.edu.cn}

\address{Gao: Department of Communication Engineering,
Xiamen University, Xiamen 361005, Fujian, China}
\email{gaozhibin@xmu.edu.cn}

\address{Wicks: Department of Mathematics, Box 354350,
University of Washington, Seattle, Washington 98195, USA}

\email{lizwicks@uw.edu}

\address{J.J. Zhang: Department of Mathematics, Box 354350,
University of Washington, Seattle, Washington 98195, USA}

\email{zhang@math.washington.edu}

\address{X.-H. Zhang: College of Sciences, Ningbo University of Technology, Ningbo, 315211, 
Zhejiang, China}

\email{zhang-xiaohong@t.shu.edu.cn}

\address{Zhu: Department of Information Sciences, the School of Mathematics and Physics,
Changzhou University,
Changzhou 213164, China}

\email{zhuhongazhu@aliyun.com}

\begin{abstract}
We introduce the Frobenius-Perron dimension of an endofunctor of a 
$\Bbbk$-linear category and provide some applications. 
\end{abstract}






\maketitle


\setcounter{section}{-1}
\section{Introduction}
\label{xxsec0}



\bigskip

The spectral radius (also called the Frobenius-Perron dimension) of a 
matrix is an elementary and extremely useful invariant in linear algebra, 
combinatorics, topology, probability and statistics. 
The Frobenius-Perron dimension has become a 
crucial concept in the study of fusion categories and representations 
of semismiple weak Hopf algebras since it was introduced by 
Etingof-Nikshych-Ostrik \cite{ENO} in early 2000 (also see \cite{EG, EGO, Ni}). 
In this paper several Frobenius-Perron type invariants are proposed to study derived 
categories, representations of finite dimensional algebras, and complexity of 
algebras and categories. 

Throughout let $\Bbbk$ be an algebraically closed field, and let everything 
be over $\Bbbk$.

\subsection{Definitions}
\label{xxsec0.1}
The first goal is to introduce the Frobenius-Perron dimension of an 
endofunctor of a category. If an object $V$ in a fusion category 
${\mathcal C}$ is considered as the associated tensor endofunctor 
$V\otimes_{\mathcal C} -$, then our definition of the Frobenius-Perron 
dimension agrees with the definition given in \cite{ENO}, see 
[Example \ref{xxex2.11}] for details. Our definition applies to 
the derived category of projective schemes and finite dimensional 
algebras, as well as other abelian and additive categories 
[Definitions \ref{xxdef2.3} and \ref{xxdef2.4}]. 
We refer the reader to 
Section \ref{xxsec2} for the following invariants of an endofunctor: 
\begin{enumerate}
\item[]
Frobenius-Perron dimension (denoted by $\fpdim$, and $\fpdim^n$ for $n\geq 1$), 
\item[]
Frobenius-Perron growth (denoted by $\fpg$), 
\item[]
Frobenius-Perron curvature (denoted by $\fpv$), and 
\item[]
Frobenius-Perron series (denoted by $\FP$). 
\end{enumerate}
One can further define the above invariants for an abelian or a triangulated 
category. Note that the Frobenius-Perron dimension/growth/curvature of a 
category can be a non-integer, see Proposition \ref{xxpro5.12}(1), Example 
\ref{xxex8.7}, and Remark \ref{xxrem5.13}(5) for non-integral values of 
$\fpdim$, $\fpg$, and $\fpv$ respectively.  

If ${\mathfrak A}$ is an abelian category, let $D^b({\mathfrak A})$
denote the bounded derived category of ${\mathfrak A}$. 
On the one hand it is 
reasonable to call $\fpdim$ a dimension function since 
$$\fpdim (D^{b}(\Mod-\Bbbk[x_1,\cdots,x_n]))=n$$ 
[Proposition \ref{xxpro4.3}(1)], but on the other hand, one might
argue that $\fpdim$ should not be called a dimension function since
$$\fpdim (D^b(coh({\mathbb P}^n)))=
\begin{cases} 1, & \quad n=1\\
\infty, &\quad  n\geq 2
\end{cases}$$ 
[Propositions \ref{xxpro6.5} and \ref{xxpro6.7}]. 
In the latter case, $\fpdim$ is an 
indicator of representation type of the category of $coh({\mathbb P}^n)$, 
namely, $coh({\mathbb P}^n)$ is tame if $n=1$, and is of wild representation 
type for all $n\geq 2$. A similar 
statement holds for projective curves in terms of 
genus [Proposition \ref{xxpro6.5}]. 

We can define the Frobenius-Perron (``fp'') version of several other classical 
invariants
\begin{enumerate}
\item[]
fp global dimension (denoted by $\fpgldim$) [Definition \ref{xxdef2.7}(1)],
\item[]
fp Kodaira dimension (denoted by $\fpk$) \cite{CG}.
\end{enumerate}
The first one is defined for all triangulated categories and the 
second one is defined for triangulated categories with Serre functor. 
In general, the $\fpgldim A$ does not agree with the classical global 
dimension of $A$ [Theorem \ref{xxthm7.8}]. The fp version of the 
Kodaira dimension agrees with the classical definition for smooth 
projective schemes \cite{CG}.

Our second goal is to provide several applications. 

\subsection{Embeddings}
\label{xxsec0.2}
In addition to the fact that the Frobenius-Perron dimension 
is an effective and sensible invariant of many categories, 
this invariant  increases when the ``size'' of the 
endofunctors and categories increase. 

\begin{theorem}
\label{xxthm0.1}
Suppose ${\mathcal C}$ and ${\mathcal D}$ are $\Bbbk$-linear 
categories.
Let $F: {\mathcal C}\to {\mathcal D}$ be a fully faithful functor. 
Let $\sigma_{\mathcal C}$ and $\sigma_{\mathcal D}$ be endofunctors 
of ${\mathcal C}$ and ${\mathcal D}$ respectively. Suppose that
$F\circ \sigma_{\mathcal C}$ is naturally isomorphic to 
$\sigma_{\mathcal D} \circ F$. Then $\FP(u,t,\sigma_{\mathcal C})\leq 
\FP(u,t,\sigma_{\mathcal D})$.
\end{theorem}

By taking $\sigma$ to be the suspension functor of a pre-triangulated 
category, we have the following immediate consequence. (Note that the 
fp-dimension of a triangulated category ${\mathcal T}$ is defined 
to be $\fpdim (\Sigma)$, where $\Sigma$ is the suspension of 
${\mathcal T}$.)

\begin{corollary}
\label{xxcor0.2} 
Let ${\mathcal T}_2$ be a pre-triangulated category and
${\mathcal T}_1$  a full pre-triangulated subcategory of 
${\mathcal T}_2$. Then the following hold.
\begin{enumerate}
\item[(1)]
$\fpdim {\mathcal T}_1\leq \fpdim {\mathcal T}_2$.
\item[(2)]
$\fpg {\mathcal T}_1\leq \fpg {\mathcal T}_2$.
\item[(3)]
$\fpv {\mathcal T}_1\leq \fpv {\mathcal T}_2$.
\item[(4)]
If ${\mathcal T}_2$ has fp-subexponential growth, so does 
${\mathcal T}_1$.
\end{enumerate}
\end{corollary}

Fully faithful embeddings of derived categories of projective 
schemes have been investigated in the study of Fourier-Mukai 
transforms, birational geometry, and noncommutative crepant 
resolutions (NCCRs) by Bondal-Orlov \cite{BO1, BO2},
Van den Bergh \cite{VdB}, Bridgeland \cite{Bri}, 
Bridgeland-King-Reid \cite{BKR} and more.

Note that if $\fpgldim ({\mathcal T})<\infty$, then 
$\fpg({\mathcal T})=0$. If $\fpg({\mathcal T})<\infty$, 
then $\fpv({\mathcal T})\leq 1$. Hence, $\fpdim$, $\fpgldim$, 
$\fpg$ and $\fpv$ measure the ``size'', ``representation 
type'', or ``complexity'' of a triangulated category 
${\mathcal T}$ at different levels. Corollary \ref{xxcor0.2} 
has many consequences concerning non-existence of fully 
faithful embeddings provided that we compute the 
invariants $\fpdim$, $\fpg$ and $\fpv$ of various categories 
efficiently. 

\subsection{Tame vs wild}
\label{xxsec0.3}
Here we mention a couple of more 
applications. First we  extend the classical trichotomy 
on the representation types of quivers to the $\fpdim$.

\begin{theorem}
\label{xxthm0.3} 
Let $Q$ be a finite quiver and let ${\mathcal Q}$ be 
the bounded derived category of finite dimensional
left $\Bbbk Q$-modules. 
\begin{enumerate}
\item[(1)]
$\Bbbk Q$ is of finite representation type 
if and only if $\fpdim {\mathcal Q}=0$.
\item[(2)]
$\Bbbk Q$ is of tame representation type
if and only if $\fpdim {\mathcal Q}=1$.
\item[(3)]
$\Bbbk Q$ is of wild representation type 
if and only if $\fpdim {\mathcal Q}=\infty$.
\end{enumerate}
\end{theorem}

By the classical theorems of Gabriel \cite{Ga1} and Nazarova 
\cite{Na}, the quivers of finite and tame representation types 
correspond to the $ADE$ and $\widetilde{A}\widetilde{D}\widetilde{E}$
diagrams respectively. 

The above theorem fails for quiver algebras with relations 
[Proposition \ref{xxpro5.12}].
As we have already seen, $\fpdim$ is related to the ``size''
of a triangulated category, as well as, the representation types. We 
will see soon that $\fpg$ is also closely connected with the complexity 
of representations. When we focus on the representation 
types, we make some tentative definitions. 

Let ${\mathcal T}$ be a triangulated category (such as $D^b(\Mod_{f.d.}-A)$).
\begin{enumerate}
\item[(i)]
We call ${\mathcal T}$ {\it fp-trivial}, if 
$\fpdim {\mathcal T}=0$.
\item[(ii)]
We call ${\mathcal T}$ {\it fp-tame}, if $\fpdim {\mathcal T}=1$.
\item[(iii)]
We call ${\mathcal T}$ {\it fp-potentially-wild}, if $\fpdim {\mathcal T}>1$.
Further,
\begin{enumerate}
\item[(iiia)]
${\mathcal T}$ is {\it fp-finitely-wild}, if 
$1<\fpdim {\mathcal T}<\infty$.
\item[(iiib)]
${\mathcal T}$ is {\it fp-locally-finitely-wild}, 
if $\fpdim {\mathcal T}=\infty$
and $\fpdim^n({\mathcal T})<\infty$ for all $n$.
\item[(iiic)]
${\mathcal T}$ is {\it fp-wild}, if $\fpdim^1 {\mathcal T}=\infty$. 
\end{enumerate}
\end{enumerate}
There are other notions of tame/wildness in representation theory, 
see for example, \cite{GKr, Dr2}. Following the above definition, 
$\fpdim$ provides a quantification of the tame-wild dichotomy.
By Theorem \ref{xxthm0.3}, finite/tame/wild representation types 
of the path algebra $\Bbbk Q$ are equivalent to the fp-version of 
these properties of ${\mathcal Q}$. Let $A$ be a quiver algebra 
with relations and let ${\mathcal A}$ be the derived category 
$D^b(\Mod_{f.d.}-A)$. Then, in general, finite/tame/wild 
representation types of $A$ are NOT equivalent to the fp-version of 
these properties of ${\mathcal A}$ [Example \ref{xxex5.5}]. 
It is natural to ask

\begin{question}
\label{xxque0.4}
For which classes of algebras $A$, is the fp-wildness of 
${\mathcal A}$ equivalent to the classical and other wildness 
of $A$ in representation theory literature?
\end{question}
 
\subsection{Complexity}
\label{xxsec0.4}
The complexity of a module or of an algebra is an 
important invariant in studying representations 
of finite dimensional algebras 
\cite{AE, Ca, CDW, GLW}. 
Let $A$ be the quiver algebra $\Bbbk Q/(R)$ with relations $R$. 
The {\it complexity} 
of $A$ is defined to be the complexity of the 
$A$-module $T:=A/Jac(A)$, namely, 
$$cx(A)=cx(T):=\limsup_{n\to\infty} \log_{n} (\dim \Ext^n_A(T,T))+1.$$
Let $\GKdim$ denote the 
Gelfand-Kirillov dimension of an algebra 
(see \cite{KL} and \cite[Ch. 13]{MR}).
Under some reasonable hypotheses, one can show  
$$cx(A)=\GKdim \left(\bigoplus_{n=0}^{\infty}\Ext^n_A(T,T)\right).$$
It is easy to see that $cx(A)$ is an derived invariant. We 
extend the definition of the complexity to any triangulated 
category [Definition \ref{xxdef8.2}(4)].

\begin{theorem} 
\label{xxthm0.5}
Let $A$ be the algebra $\Bbbk Q/(R)$ with relations $R$ and let 
${\mathcal A}$ be the bounded derived category of finite dimensional
left $A$-modules. Then 
$$\fpg({\mathcal A})\leq cx(A)-1.$$ 
\end{theorem}

The equality $\fpg({\mathcal A})=cx(A)-1$ holds under some 
hypotheses [Theorem \ref{xxthm8.4}(2)].

\subsection{Frobenius-Perron function}
\label{xxsec0.5}
If ${\mathcal T}$ is a triangulated category with Serre functor $S$, 
we have a fp-function 
$$fp: {\mathbb Z}^2 \longrightarrow  {\mathbb R}_{\geq 0}$$
which is defined by
$$fp(a,b):=\fpdim (\Sigma^a \circ S^b)\in {\mathbb R}_{\geq 0}.$$
Then $\fpdim({\mathcal T})$ is the value of the fp-function at $(1,0)$. 

The fp-function for the projective line ${\mathbb P}^1$ and 
the quiver $A_2$ are given in the Examples \ref{xxex5.1} and 
\ref{xxex5.4} respectively.

The statements in Theorem \ref{xxthm0.3}, Questions 
\ref{xxque0.4} and \ref{xxque7.11} indicate that $fp(1,0)$ 
predicts the representation type of ${\mathcal T}$ for
certain triangulated categories. It is expected that 
values of the fp-function at other points in ${\mathbb Z}^2$ 
are sensitive to other properties of ${\mathcal T}$.

\subsection{Other properties}
\label{xxsec0.6}
The paper contains some basic properties of $\fpdim$.
Let us mention one of them.

\begin{proposition}[Serre duality]
\label{xxpro0.6} 
Let ${\mathcal C}$ be a $\Hom$-finite category with Serre functor $S$.
Let $\sigma$ be an endofunctor of ${\mathcal C}$. 
\begin{enumerate}
\item[(1)]
If $\sigma$ has a right adjoint $\sigma^!$, then 
$$\fpdim (\sigma)=\fpdim(\sigma^! \circ S).$$
\item[(2)]
If $\sigma$ is an equivalence with quasi-inverse 
$\sigma^{-1}$, then 
$$\fpdim (\sigma)=\fpdim(\sigma^{-1} \circ S).$$
\item[(3)]
If ${\mathcal C}$ is $n$-Calabi-Yau, then we have a duality,
for all $i$,
$$\fpdim (\Sigma^i)=\fpdim (\Sigma^{n-i}).$$
\end{enumerate}
\end{proposition}

\subsection{Computations}
\label{xxsec0.7}
Our third goal is to develop methods for computation. 
To use fp-invariants, we need to compute as many examples 
as possible. In general it is extremely difficult to 
calculate useful invariants for derived categories, as
the definitions of  these invariants are quite sophisticated.
We develop some techniques for computing fp-invariants. 
In Sections 4 to 8, we compute the
fp-dimension for some non-trivial examples. 

\subsection{Conventions}
\label{xxsec0.8}
\begin{enumerate}
\item[(1)]
Usually $Q$ means a quiver.
\item[(2)]
${\mathcal T}$ is a (pre)-triangulated category with suspension functor
$\Sigma=[1]$.
\item[(3)]
If $A$ is an algebra over the base field $\Bbbk$, then $\Mod_{f.d.}-A$
denote the category of finite dimensional left $A$-modules.
\item[(4)]
If $A$ is an algebra, then we use ${\mathfrak A}$ for the 
abelian category $\Mod_{f.d.}-A$.
\item[(5)]
When ${\mathfrak A}$ is an abelian category, we use 
${\mathcal A}$ for the bounded derived category $D^b({\mathfrak A})$.
\end{enumerate}

\bigskip

This paper is organized as follows. We provide background material
in Section 1. The basic definitions are introduced in 
Section 2. Some basic properties are given in Section 3. 
Section 4 deals with some derived categories of modules 
over commutative rings. In Section 5, we work out the fp-theories
of the projective line and quiver $A_2$, as well as an example
of non-integral $\fpdim$. In Section 6, we develop some 
techniques to handle the $\fpdim$ of projective curves 
and prove the tame-wild dichotomy of projective curves
in terms of $\fpdim$. Theorem \ref{xxthm0.3} is proved 
in Section 7 where representation types are discussed.
Section 8 focuses on the complexity 
of algebras and categories. We continue to develop the 
fp-theory in our companion paper \cite{CG}.

\section{Preliminaries}
\label{xxsec1}

\subsection{Classical Definitions}
\label{xxsec1.1}
Let $A$ be an $n\times n$-matrix over complex numbers ${\mathbb C}$.
The {\it spectral radius} of $A$ is defined to be 
$$\rho(A):=\max\{ |r_1|, |r_2|, \cdots, |r_n|\}\quad \in {\mathbb R}$$
where $\{r_1,r_2,\cdots, r_n\}$ is the complete set of eigenvalues of $A$.
When each entry of $A$ is a positive real number,
$\rho(A)$ is also called the {\it Perron root} or the 
{\it Perron-Frobenius eigenvalue} of $A$.
When applying $\rho$ to the adjacency matrix of a graph 
(or a quiver), the spectral radius of the adjacency matrix is sometimes 
called the {\it Frobenius-Perron dimension} of the graph (or the quiver).

Let us mention a classical result concerning the spectral radius 
of simple graphs. A finite graph $G$ is called {\it simple} if 
it has no loops and no multiple edges. Smith \cite{Sm} 
formulated the following result:

\begin{theorem}\cite[Theorem 1.3]{DoG}
\label{xxthm1.1}
Let $G$ be a finite, simple, and connected graph with adjacency matrix $A$.
\begin{enumerate}
\item[(1)]
$\rho(A)=2$ if and only if G is one of the extended Dynkin diagrams 
of type  $\widetilde{A}\widetilde{D}\widetilde{E}$.
\item[(2)]
$\rho(A)<2$ if and only if G is one of the Dynkin diagrams of type 
$ADE$.
\end{enumerate}
\end{theorem}

To save space we refer to \cite{DoG} and \cite{HPR} 
for the diagrams of the $ADE$ and 
$\widetilde{A}\widetilde{D}\widetilde{E}$ quivers/graphs.

In order to include some infinite-dimensional cases, we extend 
the definition of the spectral radius in the following way.

Let $A:=(a_{ij})_{n\times n}$ be an $n\times n$-matrix with 
entries $a_{ij}$ in $\overline{\mathbb R}:=
{\mathbb R}\cup \{\pm \infty\}$. 
Define $A'=(a'_{ij})_{n\times n}$ where 
$$a'_{ij}=\begin{cases} a_{ij} & a_{ij}\neq \pm \infty,\\
x_{ij} & a_{ij}=\infty,\\
-x_{ij} & a_{ij}=-\infty.
\end{cases}
$$
In other words, we are replacing $\infty$  in the
$(i,j)$-entry by a finite real number, called $x_{ij}$, in the
$(i,j)$-entry. And every $x_{ij}$ is considered as a variable or 
a function mapping ${\mathbb R}\to {\mathbb R}$. 

\begin{definition}
\label{xxdef1.2} 
Let $A$ be an $n\times n$-matrix with entries in $\overline{\mathbb R}$.
The {\it spectral radius} of $A$ is defined to be
\begin{equation}
\label{E1.2.1}\tag{E1.2.1}
\rho(A):=\liminf_{{\text{all}}\; x_{ij}\to \infty} \; \rho(A')
\quad \in \overline{\mathbb R}.
\end{equation}
\end{definition}

\begin{remark}
\label{xxrem1.3} 
It also makes sense to use $\limsup$ instead of $\liminf$ in \eqref{E1.2.1}. 
We choose to take $\liminf$ in this paper.
\end{remark}

Here is an easy example. 

\begin{example}
\label{xxex1.4}
Let $A=\begin{pmatrix} 1 & -\infty\\ 0 & 2\end{pmatrix}$. Then 
$A'=\begin{pmatrix} 1 & -x_{12}\\ 0 & 2\end{pmatrix}$. It is obvious that
$$\rho(A)=\lim_{x_{12}\to \infty} \rho(A')=\lim_{x_{12}\to \infty} 2=2.$$
\end{example}

\subsection{$\Bbbk$-linear categories}
\label{xxsec1.2} 
If ${\mathcal C}$ is a $\Bbbk$-linear category, then 
$\Hom_{\mathcal C}(M,N)$ is a $\Bbbk$-module for all
objects $M,N$ in ${\mathcal C}$. If ${\mathcal C}$ is 
also abelian, then $\Ext^i_{\mathcal C}(M,N)$ are 
$\Bbbk$-modules for all $i\geq 0$. Let $\dim$ be the 
$\Bbbk$-vector space dimension. 

\begin{remark}
\label{xxrem1.5}
One can use a dimension function other than $\dim$. Even 
when a category ${\mathcal C}$ is not $\Bbbk$-linear, it might 
still make sense to define some dimension function $\dim$ on 
the Hom-sets of the category ${\mathcal C}$. The definition of 
Frobenius-Perron dimension given in the next section can be 
modified to fit this kind of $\dim$.
\end{remark}

\subsection{Frobenius-Perron dimension of a quiver}
\label{xxsec1.3}
In this subsection we recall some known elementary definitions and 
facts.

\begin{definition}
\label{xxdef1.6} Let $Q$ be a quiver.
\begin{enumerate}
\item[(1)]
If $Q$ has finitely many vertices, then the {\it Frobenius-Perron 
dimension} of $Q$ is defined to be 
$$\fpdim Q:=\rho(A(Q))$$ 
where $A(Q)$ is the adjacency matrix of $Q$.
\item[(2)]
Let $Q$ be any quiver. 
The {\it Frobenius-Perron dimension} of $Q$ is defined to 
be 
$$\fpdim Q:=\sup\{ \fpdim Q'\}$$
where $Q'$ runs over all finite subquivers of $Q$.
\end{enumerate}
\end{definition}

See \cite[Propositions 2.1 and 3.2]{ES} for connections between 
$\fpdim$ of a quiver and its representation types, as well as its 
complexity. We need the following well-known facts in linear algebra.

\begin{lemma}
\label{xxlem1.7}
\begin{enumerate}
\item[(1)]
Let $B$ be a square matrix with nonnegative entries and let
$A$ be a principal minor of $B$. Then $\rho(A)\leq \rho(B)$.
\item[(2)]
Let $A:=(a_{ij})_{n\times n}$ and $B:=(b_{ij})_{n\times n}$ 
be two square matrices such that $0\leq a_{ij}\leq b_{ij}$
for all $i,j$. Then $\rho(A)\leq \rho(B)$.
\end{enumerate}
\end{lemma}

Let $Q$ be a quiver with vertices $\{v_1,\cdots, v_n\}$. An 
oriented cycle based at a vertex $v_i$ is called 
{\it indecomposable} if it is not a product of two oriented 
cycles based at $v_i$. For each vertex $v_i$ let $\theta_i$ be 
the number of indecomposable oriented cycles based at $v_i$. 
Define the {\it cycle number} of a quiver $Q$ to be 
$$\Theta(Q):=\max\{ \theta_i\mid \forall \; i\}.$$
The following result should be well-known.
 
\begin{theorem}
\label{xxthm1.8}
Let $Q$ be a quiver and let $\Theta(Q)$ be the cycle number of $Q$.
\begin{enumerate}
\item[(1)]
$\fpdim (Q)=0$ if and only if $\Theta(Q)=0$, namely, $Q$ is acyclic.
\item[(2)]
$\fpdim (Q)=1$ if and only if $\Theta(Q)=1$.
\item[(3)]
$\fpdim (Q)>1$ if and only if $\Theta(Q)\geq 2$.
\end{enumerate}
\end{theorem}

The proof is not hard, and to save space, it is omitted.

\section{Definitions}
\label{xxsec2}
Throughout the rest of the paper, let ${\mathcal C}$ denote a 
$\Bbbk$-linear category. A functor between two $\Bbbk$-linear 
categories is assumed to preserve the $\Bbbk$-linear structure.
For simplicity, 
$\dim(A,B)$ stands for $\dim \Hom_{\mathcal C}(A,B)$ for any 
objects $A$ and $B$ in ${\mathcal C}$.

The set of finite subsets of nonzero objects in ${\mathcal C}$ 
is denoted by $\Phi$ and the set of subsets of $n$ nonzero objects 
in ${\mathcal C}$ is denoted by $\Phi_n$ for each $n\geq 1$.
It is clear that $\Phi=\bigcup_{n\geq 1} \Phi_n$. We do not
consider the empty set as an element of $\Phi$. 

\begin{definition}
\label{xxdef2.1} 
Let $\phi:=\{X_1, X_2, \cdots,X_n\}$ be a finite subset of nonzero
objects in ${\mathcal C}$, namely, $\phi\in \Phi_n$. Let $\sigma$ 
be an endofunctor of ${\mathcal C}$.
\begin{enumerate}
\item[(1)]
The {\it adjacency matrix} of $(\phi, \sigma)$ is defined to be
$$A(\phi, \sigma):=(a_{ij})_{n\times n}, \quad 
{\text{where}}\;\; a_{ij}:=\dim(X_i, \sigma(X_j)) \;\;\forall i,j.$$
\item[(2)]
An object $M$ in ${\mathcal C}$ is called a {\it brick} 
\cite[Definition 2.4, Ch. VII]{AS} if 
\begin{equation}
\notag
\Hom_{\mathcal C}(M,M)=\Bbbk. 
\end{equation}
If ${\mathcal C}$ is a pre-triangulated category,  an object $M$ in 
${\mathcal C}$ is called an {\it atomic} object if it is a brick 
and satisfies
\begin{equation}
\label{E2.1.1}\tag{E2.1.1}
\Hom_{\mathcal C}(M, \Sigma^{-i}(M))=0, \quad \forall \; i>0.
\end{equation}
\item[(3)]
$\phi\in \Phi$ is called a {\it brick set} (respectively, an 
{\it atomic set})  if each $X_i$ is a brick (respectively, atomic) 
and 
$$\dim(X_i, X_j)=\delta_{ij}$$
for all $1\leq i,j\leq n$. The set of brick (respectively, atomic) 
$n$-object subsets is denoted by $\Phi_{n,b}$ (respectively, 
$\Phi_{n,a}$). We write $\Phi_{b}=\bigcup_{n\geq 1} 
\Phi_{n,b}$ (respectively, $\Phi_{a}=\bigcup_{n\geq 1} 
\Phi_{n,a}$). Define the {\it b-height} of ${\mathcal C}$ to be
$$h_b({\mathcal  C})=\sup\{n\mid \Phi_{n,b} \; 
{\rm{ is\; nonempty}}\}$$
and the {\it a-height} of ${\mathcal C}$ (when ${\mathcal C}$ is 
pre-triangulated) to be
$$h_a({\mathcal  C})=\sup\{n\mid \Phi_{n,a} \; 
{\rm{ is\; nonempty}}\}.$$
\end{enumerate}
\end{definition}

\begin{remark}
\label{xxrem2.2}
\begin{enumerate}
\item[(1)]
A brick may not be atomic. Let $A$ be the algebra
$$\Bbbk \langle x,y\rangle /(x^2, y^2-1, xy+yx).$$ 
This is a 4-dimensional Frobenius algebra (of injective 
dimension zero). There are two simple left $A$-modules 
$$S_0:=A/(x, y-1), \quad {\rm{and}} \quad S_1:=A/(x, y+1).$$ 
Let $M_i$ be the injective hull of $S_i$
for $i=0,1$. (Since $A$ is Frobenius, $M_i$ is projective.) There are two 
short exact sequences
$$0\longrightarrow S_0 \longrightarrow M_0 
\xrightarrow{\;\;f\;\;} S_1 \longrightarrow 0$$
and 
$$0 \longrightarrow S_1 \xrightarrow{\;\;g\;\;} M_1 
\longrightarrow S_0 \longrightarrow 0.$$
It is easy to check that $\Hom_A(M_i,M_j)=\Bbbk$ for all
$0\leq i,j\leq 1$. Let ${\mathcal A}$ be the 
derived category $D^b(\Mod_{f.d.}-A)$ and let $X$ be 
the complex
$$\cdots\longrightarrow 0 \longrightarrow M_0 \xrightarrow{g\circ f} 
M_1 \longrightarrow 0 \longrightarrow \cdots $$
An easy computation shows that $\Hom_{\mathcal A}(X,X)=\Bbbk=
\Hom_{\mathcal A}(X, X[-1])$. So $X$ is a brick, but not 
atomic. 
\item[(2)]
A brick object is called a {\it schur} object by several authors, 
see \cite{CC} and \cite{CKW}. It is also called 
{\it endo-simple} by others, see \cite{vR1, vR2}.
\item[(3)]
An atomic object in a triangulated category is close to being  
a point-object defined by Bondal-Orlov \cite[Definition 2.1]{BO1}. 
A point-object was defined on a triangulated category with 
Serre functor. In this paper we do not automatically assume 
the existence of a Serre functor in general. When 
${\mathcal C}$ is not a pre-triangulated category, we can not 
even ask for \eqref{E2.1.1}. In that case we can only talk 
about bricks.
\end{enumerate}
\end{remark}

\begin{definition}
\label{xxdef2.3}
Retain the notation as in Definition \ref{xxdef2.1}, and 
we use $\Phi_{b}$ as the testing objects. When ${\mathcal C}$ 
is a pre-triangulated category, $\Phi_{b}$ is automatically 
replaced by $\Phi_{a}$ unless otherwise stated.
\begin{enumerate}
\item[(1)]
The {\it $n$th Frobenius-Perron dimension} of $\sigma$ is defined to be
$$\fpdim^n (\sigma):=\sup_{\phi\in \Phi_{n,b}}\{\rho(A(\phi,\sigma))\}.$$
If $\Phi_{n,b}$ is empty, then by convention, $\fpdim^n(\sigma)=0$.
\item[(2)]
The {\it Frobenius-Perron dimension} of $\sigma$ is defined to be
$$\fpdim (\sigma):=\sup_n \{\fpdim^n(\sigma)\}
=\sup_{\phi\in \Phi_{b}} \{\rho(A(\phi,\sigma)) \}.$$
\item[(3)]
The {\it Frobenius-Perron growth} of $\sigma$ is defined to be
$$\fpg (\sigma):=\sup_{\phi\in \Phi_{b}} 
\{\limsup_{n\to\infty} \; \log_{n}(\rho(A(\phi,\sigma^n))) \}.$$
By convention, $\log_n 0 =-\infty$.
\item[(4)]
The {\it Frobenius-Perron curvature} of $\sigma$ is defined to be
$$\fpv (\sigma):=\sup_{\phi\in \Phi_{b}} \{\limsup_{n\to\infty} \;  
(\rho(A(\phi,\sigma^n)))^{1/n} \}.$$
This is motivated by the concept of the {\it curvature} of a 
module over an algebra due to Avramov \cite{Av}.
\item[(5)]
We say $\sigma$ has {\it fp-exponential growth} (respectively,
{\it fp-subexponential growth})
if $\fpv(\sigma)>1$ (respectively, $\fpv(\sigma)\leq 1$).
\end{enumerate}
\end{definition}

Sometimes we prefer to have all information from the Frobenius-Perron 
dimension. We make the following definition.

\begin{definition}
\label{xxdef2.4}
Let ${\mathcal C}$ be a category and $\sigma$ be an endofunctor
of ${\mathcal C}$.
\begin{enumerate}
\item[(1)]
The {\it Frobenius-Perron theory} (or fp-theory) of $\sigma$ is defined 
to be the set
$$\{\fpdim^n (\sigma^m)\}_{n\geq 1,m\geq 0}.$$
\item[(2)]
The {\it Frobenius-Perron series} (or fp-series) of $\sigma$ is defined to be
$$\FP(u, t,\sigma):=\sum_{m=0}^{\infty} \sum_{n=1}^{\infty} 
\fpdim^n (\sigma^m) t^m u^n.$$
\end{enumerate}
\end{definition}

\begin{remark}
\label{xxrem2.5} 
To define Frobenius-Perron dimension, one only needs have 
an assignment $\tau: \Phi_{n}\to M_{n\times n}(\Mod-\Bbbk)$,
for every $n\geq 1$,
satisfying the property that 

\bigskip

{\it if $\phi_1$ is a subset of $\phi_2$,
then $\tau(\phi_1)$ is a principal submatrix of $\tau(\phi_2)$.}

\bigskip

\noindent
Then we define the adjacency matrix of $\phi\in \Phi_n$ to be
$$A(\phi,\tau)=(a_{ij})_{n\times n}$$
where 
$$a_{ij}=\dim \; (\tau(\phi))_{ij}\;\; \forall i,j.$$
Then the Frobenius-Perron dimension of $\tau$ is defined 
in the same way as in Definition \ref{xxdef2.3}. If there is a 
sequence of $\tau_m$,
the Frobenius-Perron series of $\{\tau_m\}$ is defined 
in the same way as in Definition \ref{xxdef2.4} by replacing $\sigma^m$ by
$\tau_m$. See Example \ref{xxex2.6} next.
\end{remark}

\begin{example}
\label{xxex2.6}
\begin{enumerate}
\item[(1)]
Let ${\mathfrak A}$ be a $\Bbbk$-linear abelian category.
For each $m\geq 1$ and
$\phi=\{X_1,\cdots,X_n\}$, define 
$$E^{m}: \phi\longrightarrow \left(\Ext^m_{\mathfrak A}(X_i,X_j)\right)_{n\times n}.$$
By convention, let $\Ext^0_{\mathfrak A}(X_i,X_j)$ 
denote $\Hom^0_{\mathfrak A}(X_i,X_j)$.
Then, for each $m\geq 0$, one can define the Frobenius-Perron 
dimension of $E^{m}$ as mentioned in Remark \ref{xxrem2.5}.
\item[(2)]
Let ${\mathfrak A}$ be the $\Bbbk$-linear abelian category
$\Mod_{f.d.}-A$ where $A$ is a finite dimensional 
commutative algebra over a base field $\Bbbk$.
For each $m\geq 1$ and $\phi=\{X_1,\cdots,X_n\}$, define 
$$T_{m}: \phi\longrightarrow \left(\Tor_m^A(X_i,X_j)\right)_{n\times n}.$$
By convention, let $\Tor_0^A(X_i,X_j)$ denote $X_i\otimes_A X_j$.
Then, for each $m\geq 0$, one can define the Frobenius-Perron 
dimension of $T_{m}$ as mentioned in Remark \ref{xxrem2.5}.
\end{enumerate}
\end{example}

\begin{definition}
\label{xxdef2.7} 
\begin{enumerate}
\item[(1)]
Let ${\mathfrak A}$ be an abelian category. The 
{\it Frobenius-Perron dimension} of ${\mathfrak A}$
is defined to be
$$\fpdim {\mathfrak A}:=\fpdim (E^1)$$
where $E^1:=\Ext^1_{\mathfrak A}(-,-)$
is defined as in Example \ref{xxex2.6}(1). The 
{\it Frobenius-Perron theory} of ${\mathfrak A}$ is the collection
$$\{\fpdim^m (E^n)\}_{m\geq 1, n\geq 0}$$
where $E^n:=\Ext^n_{\mathcal A}(-,-)$ is defined as in Example 
\ref{xxex2.6}(1). 
\item[(2)]
Let ${\mathcal T}$ be a pre-triangulated category with suspension
$\Sigma$. The 
{\it Frobenius-Perron dimension} of ${\mathcal T}$
is defined to be
$$\fpdim {\mathcal T}:=\fpdim (\Sigma).$$
The {\it Frobenius-Perron theory} of ${\mathcal T}$ is the collection
$$\{\fpdim^m (\Sigma^n)\}_{m\geq 1, n\in {\mathbb Z}}.$$
The {\it fp-global dimension} of ${\mathcal T}$ is defined to be
$$\fpgldim {\mathcal T}:=\sup \{n \mid \fpdim(\Sigma^n)\neq 0\}.$$
If ${\mathcal T}$ possesses a Serre functor $S$, the 
{\it Frobenius-Perron $S$-theory} of ${\mathcal T}$ is the collection
$$\{\fpdim^m (\Sigma^n \circ S^ w)\}_{m\geq 1, n,w\in {\mathbb Z}}.$$
\end{enumerate}
\end{definition}

\begin{remark}
\label{xxrem2.8}
\begin{enumerate}
\item[(1)]
The Frobenius-Perron dimension (respectively, Frobenius-Perron theory, 
fp-global dimension) can be defined for 
suspended categories \cite{KV} and 
pre-$n$-angulated categories \cite{GKO} in the same
way as Definition \ref{xxdef2.7}(2) since there is a suspension 
functor $\Sigma$.
\item[(2)]
When ${\mathfrak A}$ is an abelian category, another way of defining 
the Frobenius-Perron dimension $\fpdim {\mathfrak A}$ is as follows.
We first embed ${\mathfrak A}$ into the derived category 
$D^b({\mathfrak A})$. The suspension functor $\Sigma$ of 
$D^b({\mathfrak A})$ maps ${\mathfrak A}$ to ${\mathfrak A}[1]$ 
(so it is not a functor of ${\mathfrak A}$). The adjacency matrix 
$A(\phi, \Sigma)$ is still defined as in Definition \ref{xxdef2.1}(1) 
for brick sets $\phi$ in ${\mathfrak A}$. Then we can define 
$$\fpdim(\Sigma\mid_{\mathfrak A}):
=\sup_{\phi\in \Phi_{b}, \phi\subset {\mathfrak A}} \{\rho(A(\phi,\Sigma)) \}$$
as in Definition \ref{xxdef2.3}(2) by considering only the brick sets 
in ${\mathfrak A}$. Now $\fpdim({\mathfrak A})$ agrees with 
$\fpdim(\Sigma\mid_{\mathfrak A})$.
\end{enumerate}
\end{remark}

The following lemma is clear.

\begin{lemma}
\label{xxlem2.9}
Let ${\mathfrak A}$ be an abelian category and $n\geq 1$.
Then $\fpdim^n (D^b({\mathfrak A}))\geq \fpdim^n ({\mathfrak A})$. 
A similar statement holds for $\fpdim$, $\fpg$ and $\fpv$.
\end{lemma}

\begin{proof}
This follows from the fact that there is a fully faithful 
embedding ${\mathfrak A}\to D^b({\mathfrak A})$ and that
$E^1$ on ${\mathfrak A}$ agrees with $\Sigma$ on $D^b({\mathfrak A})$.
\end{proof}

For any category ${\mathcal C}$ with an endofunctor $\sigma$, we define
the {\it $\sigma$-quiver} of ${\mathcal C}$, denoted by 
$Q^{\sigma}_{\mathcal C}$, as follows:
\begin{enumerate}
\item[(1)]
the vertex set of $Q^{\sigma}_{\mathcal C}$ consists of 
bricks in $\Phi_{1,b}$ in ${\mathcal C}$ (respectively, 
atomic objects in $\Phi_{1,a}$ when ${\mathcal C}$ is pre-triangulated), and 
\item[(2)]
the arrow set of $Q^{\sigma}_{\mathcal C}$ consists of $n_{X,Y}$-arrows
from $X$ to $Y$, for all $X,Y\in \Phi_{1,b}$ (respectively, 
in $\Phi_{1,a}$), 
where $n_{X,Y}=\dim (X,\sigma(Y))$. 
\end{enumerate}

If $\sigma$ is $E^1$, this  quiver is denoted by 
$Q_{\mathcal C}^{E^1}$, which will be used in later sections.

The following lemma follows from the definition.

\begin{lemma}
\label{xxlem2.10} Retain the above notation. Then
$\fpdim \sigma\leq \fpdim Q_{\mathcal C}^{\sigma}$.
\end{lemma}

The fp-theory was motivated by the Frobenius-Perron 
dimension of objects in tensor or fusion categories 
\cite{EG}, see the following example. 

\begin{example}
\label{xxex2.11} 
First we recall the definition of the Frobenius-Perron 
dimension given in \cite[Definitions 3.3.3 and 6.1.6]{EG}. Let 
${\mathcal C}$ be a finite semisimple $\Bbbk$-linear tensor category. 
Suppose that $\{X_1,\cdots, X_n\}$ is the complete list of 
non-isomorphic simple objects in ${\mathcal C}$. Since 
${\mathcal C}$ is semisimple, 
every object $X$ in ${\mathcal C}$ is a direct sum 
$$X=\bigoplus_{i=1}^n X_i^{\oplus a_i}$$
for some integers $a_i\geq 0$. The tensor product on ${\mathcal C}$
makes its Grothendieck ring ${\bf Gr}({\mathcal C})$ a 
${\mathbb Z}_{+}$-ring \cite[Definition 3.1.1]{EG}.
For every object $V$ in ${\mathcal C}$ and every $j$, write 
\begin{equation}
\label{E2.11.1}\tag{E2.11.1}
V\otimes_{\mathcal C} 
X_j\cong \bigoplus_{i=1}^n X_i^{\oplus a_{ij}}
\end{equation}
for some integers $a_{ij}\geq 0$.
In the Grothendieck ring ${\bf Gr}({\mathcal C})$, the left
multiplication by $V$ sends $X_j$ to $\sum_{i=1}^n a_{ij}
X_i$. Then, by \cite[Definition 3.3.3]{EG},
the {\bf Frobenius-Perron dimension} of $V$ is defined to be 
\begin{equation}
\label{E2.11.2}\tag{E2.11.2}
{\bf FPdim}(V):=\rho((a_{ij})_{n\times n}).
\end{equation}
In fact the Frobenius-Perron dimension  is defined for any object
in a ${\mathbb Z}_{+}$-ring. 

Next we use Definition \ref{xxdef2.3}(2) to calculate the 
Frobenius-Perron dimension. Let $\sigma$ be the tensor
functor $V\otimes_{\mathcal C} -$ that is a $\Bbbk$-linear endofunctor 
of ${\mathcal C}$. If $\phi$ is a brick 
subset of ${\mathcal C}$, then $\phi$ is a subset of
$\phi_n:=\{X_1,\cdots,X_n\}$. For simplicity, assume that
$\phi$ is $\{X_1,\cdots,X_s\}$ for some $s\leq n$. It 
follows from \eqref{E2.11.1} that
\begin{equation}
\notag
\Hom_{\mathcal C}(X_i, \sigma(X_j))= \Bbbk^{\oplus a_{ij}},\;\;
\forall\; i,j.
\end{equation}
Hence the adjacency matrix of $(\phi_n,\sigma)$ is 
$$A(\phi_n,\sigma)=(a_{ij})_{n\times n}$$
and the adjacency matrix of $(\phi,\sigma)$ is a 
principal minor of $A(\phi_n,\sigma)$. By Lemma 
\ref{xxlem1.7}(1), $\rho(A(\phi,\sigma))\leq
\rho(A(\phi_n,\sigma))$. By Definition \ref{xxdef2.3}(2),
the {\it Frobenius-Perron dimension} of the functor 
$\sigma=V\otimes_{\mathcal C} -$ is 
$$\fpdim(V\otimes_{\mathcal C} -)=\sup_{\phi\in \Phi_b} \{\rho(A(\phi,\sigma))\}
=\rho(A(\phi_n,\sigma))=\rho((a_{ij})_{n\times n}),$$
which agrees with \eqref{E2.11.2}. This justifies 
calling $\fpdim(V\otimes_{\mathcal C} -)$ the Frobenius-Perron 
dimension of $V$. 
\end{example}

\section{Basic properties}
\label{xxsec3}

For simplicity, ``Frobenius-Perron'' is abbreviated to ``fp''.

\subsection{Embeddings}
\label{xxsec3.1}
It is clear that the fp-series and the fp-dimensions are 
invariant under equivalences of categories. We record this 
fact below. Recall that the Frobenius-Perron series
$\FP(u,t,\sigma)$ of an endofunctor $\sigma$ is defined in 
Definition \ref{xxdef2.4}(2). 

\begin{lemma} 
\label{xxlem3.1}
Let $F: {\mathcal C}\to {\mathcal D}$ be an equivalence of categories.
Let $\sigma_{\mathcal C}$ and $\sigma_{\mathcal D}$ be endofunctors 
of ${\mathcal C}$ and ${\mathcal D}$ respectively. Suppose that
$F\circ \sigma_{\mathcal C}$ is naturally isomorphic to 
$\sigma_{\mathcal D} \circ F$. Then $\FP(u,t,\sigma_{\mathcal C})=
\FP(u,t,\sigma_{\mathcal D})$.
\end{lemma}

Let ${\mathbb R}_{+}$ denote the set of non-negative real numbers
union with $\{\infty\}$. Let 
$$f(u,t):=\sum_{m,n=0}^{\infty} f_{m,n} t^m u^n \quad {\text{and}} \quad
g(u,t):= \sum_{m,n=0}^{\infty} g_{m,n} t^m u^n$$ 
be two elements in 
${\mathbb R}_{+}[[u,t]]$. We write $f\leq g$ if $f_{m,n}\leq g_{m,n}$ 
for all $m,n$. 

\begin{theorem}
\label{xxthm3.2}
Let $F: {\mathcal C}\to {\mathcal D}$ be a faithful functor that 
preserves brick subsets. 
\begin{enumerate}
\item[(1)]
Let $\sigma_{\mathcal C}$ and 
$\sigma_{\mathcal D}$ be endofunctors of ${\mathcal C}$ and 
${\mathcal D}$ respectively. Suppose that $F\circ \sigma_{\mathcal C}$ is 
naturally isomorphic to $\sigma_{\mathcal D} \circ F$. Then 
$\FP(u,t,\sigma_{\mathcal C})\leq \FP(u,t,\sigma_{\mathcal D})$.
\item[(2)]
Let $\tau_{\mathcal C}$ and 
$\tau_{\mathcal D}$ be assignments of ${\mathcal C}$ and 
${\mathcal D}$ respectively satisfying the property in Remark 
\ref{xxrem2.5}. Suppose that $\rho(A(\phi, \tau_{\mathcal C})) \leq 
\rho(A(F(\phi), \tau_{\mathcal D}))$ for all $\phi\in 
\Phi_{n,b}({\mathcal C})$ and all $n$. Then 
$\FP(u,t,\tau_{\mathcal C})\leq \FP(u,t,\tau_{\mathcal D})$.
\end{enumerate}
\end{theorem}

\begin{proof}
(1) For every $\phi=\{X_1,\cdots,X_n\}\in \Phi_{n}({\mathcal C})$, 
let $F(\phi)$ be $\{F(X_1),\cdots, F(X_n)\}$ in
$\Phi_{n}({\mathcal D})$. By hypothesis, if $\phi\in 
\Phi_{n,b}({\mathcal C})$, then 
$F(\phi)$ is in $\Phi_{n,b}({\mathcal D})$.
Let $A=(a_{ij})$ (respectively, 
$B=(b_{ij})$) be the adjacency matrix of $(\phi,\sigma_{\mathcal C})$ 
(respectively, of $(F(\phi), \sigma_{\mathcal D})$).
Then, by the faithfulness of $F$,
$$\begin{aligned}
a_{ij}&=\dim (X_i, \sigma_{\mathcal C}(X_j))
\leq \dim (F(X_i), F(\sigma_{\mathcal C}(X_j)))\\
&=\dim (F(X_i), \sigma_{\mathcal D}(F(X_j)))
=b_{ij}.
\end{aligned}
$$
By Lemma \ref{xxlem1.7}(2), 
\begin{equation}
\label{E3.2.1}\tag{E3.2.1}
\rho(A(\phi,\sigma_{\mathcal C})) =:\rho(A)\leq 
\rho(B):=\rho(A(F(\phi), \sigma_{\mathcal D})).
\end{equation}
By definition,
\begin{equation}
\label{E3.2.2}\tag{E3.2.2}
\fpdim^n(\sigma_{\mathcal C})\leq \fpdim^n(\sigma_{\mathcal D}).
\end{equation}
Similarly, for all $n,m$, 
$\fpdim^n(\sigma^m_{\mathcal C})\leq \fpdim^n(\sigma^m_{\mathcal D})$.
The assertion follows.

(2) The proof of part (2) is similar.
\end{proof}

Theorem \ref{xxthm0.1} follows from Theorem \ref{xxthm3.2}.

\subsection{(a-)Hereditary algebras and categories}
\label{xxsec3.2}
Recall that the global dimension of an abelian category 
${\mathfrak A}$ is defined to be
$$\gldim {\mathfrak A}:=\sup\{ n\mid \Ext^n_{\mathfrak A}(X,Y)\neq 0,
\; {\rm{for\; some\; }} X,Y\in {\mathfrak A}\}.$$
The global dimension of an algebra $A$ is defined to be the 
global dimension of the category of left $A$-modules. An algebra 
(or an abelian category) is called {\it hereditary} if it has 
global dimension at most one.

There is a nice property concerning the indecomposable objects 
in the derived category of a hereditary abelian  category
(see \cite[Section 2.5]{Ke1}).

\begin{lemma}
\label{xxlem3.3} 
Let ${\mathfrak A}$ be a hereditary abelian category. Then
every indecomposable object in the derived category 
$D({\mathfrak A})$ is isomorphic to a shift of an object in
${\mathfrak A}$.
\end{lemma}

Note that every brick (or atomic) object in an additive category 
is indecomposable. Based on the property in the above lemma, 
we make a definition.

\begin{definition}
\label{xxdef3.4}
An abelian category ${\mathfrak A}$ is called {\it a-hereditary} 
(respectively,  {\it b-hereditary}) if every atomic (respectively, 
brick) object $X$ in the bounded derived category $D^b({\mathfrak A})$ 
is of the form $M[i]$ for some object $M$ in ${\mathfrak A}$ and 
$i\in {\mathbb Z}$. The object $M$ is automatically a brick object 
in ${\mathfrak A}$.
\end{definition}

If $\alpha$ is an auto-equivalence of an abelian category
${\mathfrak A}$, then it extends naturally to an auto-equivalence,
denoted by $\overline{\alpha}$, of the derived category 
${\mathcal A}:=D^b({\mathfrak A})$. The main result in this 
subsection is the following. Recall that the $b$-height of 
${\mathfrak A}$, denoted by $h_b({\mathfrak A})$, is defined
in Definition \ref{xxdef2.1}(3) and that the Frobenius-Perron
global dimension of ${\mathcal A}$, denoted by $\fpgldim 
{\mathcal A}$, is defined in Definition \ref{xxdef2.7}(2). 

\begin{theorem}
\label{xxthm3.5}
Let ${\mathfrak A}$ be an a-hereditary abelian category with 
an auto-equivalence $\alpha$. For each $n$, define 
$n'=\min\{n,h_b({\mathfrak A})\}$. Let ${\mathcal A}$ be 
$D^b({\mathfrak A})$. 
\begin{enumerate}
\item[(1)]
If $m<0$ or $m>\gldim {\mathfrak A}$, then
$$\fpdim(\Sigma^m \circ \overline{\alpha})=0.$$
As a consequence, $\fpgldim {\mathcal A}\leq \gldim 
{\mathfrak A}$.
\item[(2)]
For each $n$, 
\begin{equation}
\label{E3.5.1}\tag{E3.5.1}
\fpdim^n(\alpha) \leq \fpdim^n(\overline{\alpha})\leq 
\max_{1\leq i\leq n'}\{\fpdim^{i}(\alpha)\}.
\end{equation}
If $\gldim {\mathfrak A}<\infty$, then 
\begin{equation}
\label{E3.5.2}\tag{E3.5.2}
\fpdim^n(\overline{\alpha})=
\max_{1\leq i\leq n'}\{\fpdim^{i}(\alpha)\}.
\end{equation}
\item[(3)]
Let $g:=\gldim {\mathfrak A}<\infty$. Let $\beta$ be the 
assignment $(X,Y)\to (\Ext^g_{\mathfrak A}(X,\alpha(Y)))$.
Then
\begin{equation}
\label{E3.5.3}\tag{E3.5.3}
\fpdim^n(\Sigma^{g}\circ \overline{\alpha})
=\max_{1\leq i\leq n'}\{\fpdim^{i}(\beta)\}.
\end{equation}
\item[(4)]
For every hereditary abelian category ${\mathfrak A}$, we have 
$\fpdim ({\mathcal A})=\fpdim ({\mathfrak A})$.
\end{enumerate}
\end{theorem}

\begin{proof} (1) Since ${\mathcal A}$ is a-hereditary, every
atomic object in ${\mathcal A}$ is of the form $M[i]$.

Case 1: $m<0$. Write $\phi$ as $\{M_1[d_1], \cdots, M_n[d_n]\}$ 
where $d_i$ is decreasing and $M_i$ is in ${\mathfrak A}$. Then, 
for $i\leq j$,
$$a_{ij}=\Hom_{\mathcal A}
(M_i[d_i],(\Sigma^m \circ \overline{\alpha} )M_j[d_j])
=\Hom_{\mathcal A}(M_i,\alpha(M_j)[d_j-d_i+m])=0$$
since $d_j-d_i+m<0$. Thus the adjacency matrix $A:=(a_{ij})_{n\times n}$ 
is strictly lower triangular. As a consequence, $\rho(A)=0$. By 
definition, $\fpdim (\Sigma^m \circ \overline{\alpha})=0$.

Case 2: $m>\gldim {\mathfrak A}$. Write $\phi$ as $\{M_1[d_1], 
\cdots M_n[d_n]\}$ where $d_i$ is increasing and $M_i$ is in 
${\mathfrak A}$. Then, for $i\geq j$,
$$a_{ij}=\Hom_{\mathcal A}
(M_i[d_i],(\Sigma^m \circ \overline{\alpha} )M_j[d_j])
=\Hom_{\mathcal A}(M_i,\alpha(M_j)[d_j-d_i+m])=0$$
since $d_j-d_i+m>\gldim {\mathfrak A}$. Thus the adjacency 
matrix $A:=(a_{ij})_{n\times n}$ is strictly upper triangular. 
As a consequence, $\rho(A)=0$. By definition, 
$\fpdim (\Sigma^m \circ \overline{\alpha})=0$.

(2) Let $F$ be the canonical fully faithful embedding 
${\mathfrak A}\to {\mathcal A}$. By Theorem 
\ref{xxthm3.2} and \eqref{E3.2.2}, 
$$\fpdim^{n}(\alpha)\leq \fpdim^n(\overline{\alpha}).$$
For the other assertion, write $\phi$ as a disjoint union 
$\phi_{d_1}\cup \cdots \cup \phi_{d_s}$ where $d_i$ is strictly 
decreasing and the subset $\phi_{d_i}$ consists of objects of 
the form $M[d_i]$ for $M\in {\mathfrak A}$. For any objects 
$X\in \phi_{d_i}$ and $Y\in \phi_{d_j}$ for $i<j$, 
$\Hom_{\mathcal A}(X,Y)=0$. Thus the adjacency matrix of 
$(\phi, \overline{\alpha})$ is of the form
\begin{equation}
\label{E3.5.4}\tag{E3.5.4}
A(\phi, \overline{\alpha})=\begin{pmatrix} A_{11} & 0      &0 & \cdots &0\\
                                 \ast   & A_{22} &0 &\cdots & 0\\
			       \ast &\ast &A_{33} &\cdots &0\\
			     \ldots &\ldots &\ldots &\ldots &0\\
			   \ast& \ast&\ast &\cdots & A_{ss}
		       \end{pmatrix}
\end{equation}
where each $A_{ii}$ is the adjacency matrix 
$A(\phi_{d_i}, \overline{\alpha})$. For each $\phi_{d_i}$, 
we have
$$A(\phi_{d_i}, \overline{\alpha})=A(\phi_{d_i}[-d_i],\overline{\alpha})
=A(\phi_{d_i}[-d_i],\alpha)$$
which implies that 
$$\rho(A_{ii})\leq \fpdim^{s_i}(\alpha)\leq \max_{1\leq j\leq n'} \fpdim^{j}(\alpha) $$
where $s_i$ is the size of $A_{ii}$ and $n'=\min\{n, h_b({\mathfrak A})\}$.
By using the matrix \eqref{E3.5.4}, 
$$\rho(A(\phi, \overline{\alpha}))=\max_i\{\rho(A_{ii})\}\leq 
\max_{1\leq j\leq n'} \fpdim^{j}(\alpha).$$ 
Then \eqref{E3.5.1} follows. 

Suppose now that $g:=\gldim {\mathfrak A}<\infty$. Let $\phi\in 
\Phi_{n,a}({\mathcal A})$. Pick any $M\in \Phi_{1,b}({\mathfrak A})$. 
Then, for $g'\gg g$, $\phi':=\phi\cup\{ M[g']\}\in 
\Phi_{n+1,a}({\mathcal A})$. By Lemma \ref{xxlem1.7}(1), 
$\rho(A(\phi', \overline{\alpha}))\geq  \rho(A(\phi, \overline{\alpha}))$.
Hence $\fpdim^n (\overline{\alpha})$ is increasing as $n$ increases.
Therefore \eqref{E3.5.2} follows from \eqref{E3.5.1}.

(3) The proof is similar to the proof of part (2). Let $F$ be the 
canonical fully faithful embedding ${\mathfrak A}\to {\mathcal A}$. By 
Theorem \ref{xxthm3.2}(2) and \eqref{E3.2.2}, 
$$\fpdim^n(\beta)\leq \fpdim^{n}(\Sigma^{g}\circ \overline{\alpha}).$$
By the argument at the end of proof of part (2),
$\fpdim^{n}(\Sigma^{g}\circ \overline{\alpha})$ increases when $n$ 
increases. Then 
$$\max_{1\leq j\leq n'} \fpdim^{j}(\beta)
\leq \fpdim^{n}(\Sigma^{g}\circ \overline{\alpha}).$$

For the other direction, write $\phi$ as a disjoint union 
$\phi_{d_1}\cup \cdots \cup \phi_{d_s}$ where $d_i$ is strictly 
increasing and $\phi_{d_i}$ consists of objects of the form 
$M[d_i]$ for $M\in {\mathfrak A}$. For objects $X\in \phi_{d_i}$ 
and $Y\in \phi_{d_j}$ for $i<j$, 
$\Hom_{\mathcal A}(X,\Sigma^{g}(\alpha(Y)))=0$. Let 
$\gamma=\Sigma^g \circ \overline{\alpha}$.
Then the adjacency matrix of $(\phi, \gamma)$ is of the form
\eqref{E3.5.4}, namely, 
\begin{equation}
\notag
A(\phi, \gamma)=\begin{pmatrix} A_{11} & 0      &0 & \cdots &0\\
                                 \ast   & A_{22} &0 &\cdots & 0\\
			       \ast &\ast &A_{33} &\cdots &0\\
			     \ldots &\ldots &\ldots &\ldots &0\\
			   \ast& \ast&\ast &\cdots & A_{ss}
		       \end{pmatrix}
\end{equation}
where each $A_{ii}$ is the adjacency matrix 
$A(\phi_{d_i}, \gamma)$. For each $\phi_{d_i}$, we have
$$A(\phi_{d_i}, \gamma)=A(\phi_{d_i}[-d_i],\gamma)
=A(\phi_{d_i}[-d_i],\beta)$$
which implies that 
$$\rho(A_{ii})\leq \fpdim^{s_i}(\beta)\leq 
\max_{1\leq j\leq n'} \fpdim^{j}(\beta) $$
where $s_i$ is the size of $A_{ii}$. By using matrix \eqref{E3.5.4}, 
$$\rho(A(\phi, \gamma))=\max_i\{\rho(A_{ii})\}\leq 
\max_{1\leq j\leq n'} \fpdim^{j}(\beta).$$ 
The assertion follows.

(4) Take $\alpha$ to be the identity functor of ${\mathfrak A}$ 
and $g=1$ (since ${\mathfrak A}$ is hereditary). By \eqref{E3.5.3}, 
we have
$$\fpdim^n (\Sigma)=\max_{1\leq i\leq n'}\{\fpdim^{i}(E^1)\}.$$
By taking $\sup_n$, we obtain that $\fpdim(E^1)=\fpdim(\Sigma)$.
The assertion follows.
\end{proof}

\subsection{Categories with Serre functor}
\label{xxsec3.3}

Recall from \cite[Section 2.6]{Ke2} that if a $\Hom$-finite category 
${\mathcal C}$ has a Serre functor $S$, then there is a natural isomorphism 
$$\Hom_{\mathcal C}(X,Y)^*\cong \Hom_{\mathcal C}(Y, S(X))$$
for all $X,Y\in {\mathcal C}$. A (pre-)triangulated $\Hom$-finite 
category ${\mathcal C}$ with Serre functor $S$ is called 
{\it $n$-Calabi-Yau} if there is a natural isomorphism
$$S\cong \Sigma^n.$$
(In \cite[Section 2.6]{Ke2} it is called {\it weakly $n$-Calabi-Yau}.)
We now prove Proposition \ref{xxpro0.6}.

\begin{proposition}[Serre duality]
\label{xxpro3.6} 
Let ${\mathcal C}$ be a $\Hom$-finite category with Serre functor $S$.
Let $\sigma$ be an endofunctor of ${\mathcal C}$. 
\begin{enumerate}
\item[(1)]
If $\sigma$ has a right adjoint $\sigma^!$, then 
$$\fpdim (\sigma)=\fpdim(\sigma^! \circ S).$$
\item[(2)]
If $\sigma$ is an equivalence with quasi-inverse 
$\sigma^{-1}$, then 
$$\fpdim (\sigma)=\fpdim(\sigma^{-1} \circ S).$$
\item[(3)]
If ${\mathcal C}$ is {\rm{(}}pre-{\rm{)}}triangulated and $n$-Calabi-Yau, 
then we have a duality
$$\fpdim (\Sigma^i)=\fpdim (\Sigma^{n-i})$$
for all $i$.
\end{enumerate}
\end{proposition}

\begin{proof} (1) Let $\phi=\{X_1,\cdots,X_n\}\in \Phi_{n,b}$ and let 
$A(\phi,\sigma)$ be the adjacency matrix
with $(i,j)$-entry $a_{ij}=\dim(X_i, \sigma(X_j))$. By Serre duality,
$$a_{ij}=\dim(X_i,\sigma(X_j))=\dim (\sigma(X_j), S(X_i))
=\dim (X_j, (\sigma^{!} \circ S)(X_i)),$$
which is the $(j,i)$-entry of the adjacency matrix $A(\phi, \sigma^{!} \circ S)$.
Then $\rho(A(\phi, \sigma))=\rho(A(\phi, \sigma^{!} \circ S))$. It follows
from the definition that $\fpdim^n(\sigma)=\fpdim^n(\sigma^{!} \circ S)$
for all $n\geq 1$. The assertion follows from the definition.

(2,3) These are consequences of part (1).
\end{proof}

\subsection{Opposite categories}
\label{xxsec3.4}

\begin{lemma} 
\label{xxlem3.7}
Let $\sigma$ be an endofunctor of ${\mathcal C}$ and suppose that $\sigma$ 
has a left adjoint $\sigma^*$. Consider $\sigma^*$ as an endofunctor 
of the opposite category  ${\mathcal C}^{op}$ of ${\mathcal C}$. Then 
$$\fpdim^n (\sigma\mid_{\mathcal C})=\fpdim^n (\sigma^*\mid_{{\mathcal C}^{op}}).$$
for all $n$.
\end{lemma}

\begin{proof} Let $\phi:=\{X_1,\cdots,X_n\}$ be a brick subset of ${\mathcal C}$
(which is also a brick subset of ${\mathcal C}^{op}$). Then 
$$
\dim_{\mathcal C}(X_i, \sigma(X_j)) =\dim_{\mathcal C}(\sigma^*(X_i), X_j)
=\dim_{{\mathcal C}^{op}} (X_j, \sigma^*(X_i)),
$$
which implies that the adjacency matrix of $\sigma^*$ as an endofunctor of 
${\mathcal C}^{op}$ is the transpose of the adjacency matrix of $\sigma$. 
The assertion follows.
\end{proof}

\begin{definition}
\label{xxdef3.8} 
\begin{enumerate}
\item[(1)]
Two pre-triangulated categories $({\mathcal T}_i, \Sigma_i)$, 
for $i=1,2$, are called 
{\it fp-equivalent} if 
\begin{equation}
\notag
\fpdim^n(\Sigma_1^m)=\fpdim^n(\Sigma_2^m)
\end{equation}
for all $n\geq 1,m\in {\mathbb Z}$.
\item[(2)]
Two algebras are {\it fp-equivalent} if their bounded derived categories
of finitely generated modules are fp-equivalent.
\item[(3)]
Two pre-triangulated categories with Serre functors 
$({\mathcal T}_i, \Sigma_i, S_i)$, for $i=1,2$, are called 
{\it fp-$S$-equivalent} if 
\begin{equation}
\notag
\fpdim^n(\Sigma_1^m\circ S_1^k )=\fpdim^n(\Sigma_2^m\circ S_2^k)
\end{equation}
for all $n\geq 1,m,k\in {\mathbb Z}$.
\end{enumerate}
\end{definition}

\begin{proposition}
\label{xxpro3.9}
Let ${\mathcal T}$ be a pre-triangulated category.
\begin{enumerate}
\item[(1)]
${\mathcal T}$ and ${\mathcal T}^{op}$ are fp-equivalent.
\item[(2)]
Suppose $S$ is a Serre functor of ${\mathcal T}$. Then $({\mathcal T}, S)$ 
and  $({\mathcal T}^{op}, S^{op})$ are fp-$S$-equivalent.
\end{enumerate}
\end{proposition}

\begin{proof} (1) Let $\Sigma$ be the suspension of ${\mathcal T}$. Then 
${\mathcal T}^{op}$ is also pre-triangulated with suspension functor 
being $\Sigma^{-1}=\Sigma^*$ (restricted to ${\mathcal T}^{op}$).
The assertion follows from Lemma \ref{xxlem3.7}.

(2) Note that the Serre functor of ${\mathcal T}^{op}$ is equal to $S^{-1}=S^{*}$
(restricted to ${\mathcal T}^{op}$). The assertion follows by Lemma \ref{xxlem3.7}.
\end{proof}

\begin{corollary}
\label{xxcor3.10}
Let $A$ be a finite dimensional algebra. 
\begin{enumerate}
\item[(1)]
$A$ and $A^{op}$ are fp-equivalent. 
\item[(2)]
Suppose $A$ has finite global dimension. In this case, the bounded derived
category of finite dimensional $A$-modules has a Serre functor.
Then $A$ and $A^{op}$ are fp-$S$-equivalent.
\end{enumerate}
\end{corollary}

\begin{proof} (1) Since $A$ is finite dimensional, the
$\Bbbk$-linear dual induces an equivalence of triangulated categories 
between
$D^b(\Mod_{f.d.}-A)^{op}$ and $D^b(\Mod_{f.d.}-A^{op})$. The assertion
follows from Proposition \ref{xxpro3.9}(1).

The proof of (2) is similar by using Proposition \ref{xxpro3.9}(2) instead.
\end{proof}

There are examples where ${\mathcal T}$ and ${\mathcal T}^{op}$ 
are not triangulated equivalent, see Example \ref{xxex3.12}.
In this paper, a $\Bbbk$-algebra $A$ is called {\it local} if $A$ 
has a unique maximal ideal $\fm$ and $A/\fm\cong \Bbbk$. The 
following lemma is easy and well-known.

\begin{lemma}
\label{xxlem3.11}
Let $A$ be a finite dimensional local algebra over $\Bbbk$. Let 
${\mathfrak A}$ be the category $\Mod_{f.d.}-A$ and ${\mathcal A}$
be $D^b({\mathfrak A})$. 
\begin{enumerate}
\item[(1)]
Let $X$ be an object in ${\mathcal A}$ such that
$\Hom_{\mathcal A}(X,X[-i])=0$ for all $i>0$. Then 
$X$ is of the form $M[n]$ where $M$ is an object 
in ${\mathfrak A}$ and $n\in {\mathbb Z}$.
\item[(2)]
Every atomic object in ${\mathcal A}$ is of the form $M[n]$ 
where $M$ is a brick object in ${\mathfrak A}$ and $n\in 
{\mathbb Z}$. Namely, ${\mathfrak A}$ is a-hereditary.
\end{enumerate}
\end{lemma}

\begin{proof} (2) is an immediate consequence of 
part (1). We only prove part (1). 

On the contrary we suppose that $H^m(X)\neq 0$
and $H^n(X)\neq 0$ for some $m<n$. Since $X$ is a bounded
complex, we can take $m$ to be minimum of such integers and 
$n$ to be the maximum of such integers. Since $A$ is local,
there is a nonzero map from $H^n(X)\to H^m(X)$, which 
induces a nonzero morphism in $\Hom_{\mathcal A}(X, X[m-n])$.
This contradicts the hypothesis.
\end{proof}

\begin{example}
\label{xxex3.12}
Let $m,n$ be integers $\geq 2$. Define $A_{m,n}$ to be the algebra
$$\Bbbk\langle x_1, x_2 \rangle/(x_1^{m}, x_2^{n}, x_1x_2).$$
It is easy to see that $A_{m,n}$ is a finite dimensional local
connected graded algebra generated in degree 1 
(with $\deg x_1=\deg x_2=1$). If $A_{m,n}$ is isomorphic to
$A_{m',n'}$ as algebras, by \cite[Theorem 1]{BZ}, these 
are isomorphic as graded algebras. Suppose $f: A_{m,n}\to A_{m',n'}$ 
is an isomorphism of graded algebras and write 
$$ f(x_1)=a x_1+bx_2, \quad f(x_2)= c x_1+ dx_2.$$
Then the relation $f(x_1)f(x_2)=0$ forces $b=c=0$. As a consequence,
$m=m'$ and $n=n'$. So we have proven that
\begin{enumerate}
\item[(1)]
$A_{m,n}$ is isomorphic to $A_{m',n'}$ if and only if $m=m'$
and $n=n'$.
\end{enumerate}

Next we claim that 

\begin{enumerate}
\item[(2)]
if $m\neq n$, then the derived category $D^b(\Mod_{f.d.}-A_{m,n})$ 
is not triangulated equivalent to $D^b(\Mod_{f.d.}-A_{m,n}^{op})$.
\end{enumerate}

Let $m,n,m',n'$ be integers $\geq 2$.
Suppose that $D^b(\Mod_{f.d.}-A_{m,n})$ is triangulated equivalent 
to $D^b(\Mod_{f.d.}-A_{m',n'})$. Since $A_{m,n}$ is local, by 
\cite[Theorem 2.3]{Ye}, every 
tilting complex over $A_{m,n}$ is of the form $P[n]$ where $P$ is a 
progenerator over $A_{m,n}$. As a consequence, $A_{m,n}$ is 
Morita equivalent to $A_{m',n'}$. Since both 
$A_{m,n}$ and  $A_{m',n'}$ are local, Morita equivalence implies
that $A_{m,n}$ is isomorphic to $A_{m',n'}$. By part (1), $m=m'$
and $n=n'$. In other words, if $(m,n)\neq (m',n')$, then 
$D^b(\Mod_{f.d.}-A_{m,n})$ is not triangulated equivalent 
to $D^b(\Mod_{f.d.}-A_{m',n'})$. As a consequence, 
if $m\neq n$, then $D^b(\Mod_{f.d.}-A_{m,n})$ is not triangulated 
equivalent to $D^b(\Mod_{f.d.}-A_{n,m})$. By definition, 
$A_{m,n}^{op}\cong A_{n,m}$. Therefore the claim (2) follows.

We can show that $D^b(\Mod_{f.d.}-A)$ is dual to 
$D^b(\Mod_{f.d.}-A^{op})$ by using the $\Bbbk$-linear dual.  
In other words, $D^b(\Mod_{f.d.}-A)^{op}$
is triangulated equivalent to $D^b(\Mod_{f.d.}-A^{op})$.
Therefore the following is a consequence of part (2).

\begin{enumerate}
\item[(3)]
Suppose $m\neq n$ and let ${\mathcal A}$ be $D^b(\Mod_{f.d.}-A_{m,n})$. 
Then ${\mathcal A}$ is not triangulated equivalent to
${\mathcal A}^{op}$. But by Proposition \ref{xxpro3.9}(1), ${\mathcal A}$
and ${\mathcal A}^{op}$ are fp-equivalent.
\end{enumerate}
\end{example}

\section{Derived category over a commutative ring}
\label{xxsec4}

Throughout this section $A$ is a commutative algebra 
and ${\mathcal A}=D^b(\Mod-A)$. (In other sections 
${\mathcal A}$ usually denotes $D^b(\Mod_{f.d.}-A)$.)

\begin{lemma}
\label{xxlem4.1} 
Let $A$ be a commutative algebra. Let $X$ be an atomic
object in ${\mathcal A}$. Then $X$ is of the form 
$M[i]$ for some simple $A$-module $M$ and some $i\in {\mathbb Z}$.
As a consequence, $\Mod-A$ is a-hereditary. 
\end{lemma}

\begin{proof} Consider $X$ as a bounded above complex of projective $A$-modules. 
Since $A$ is commutative, every $f\in A$ induces naturally 
a morphism of $X$ by multiplication. For each $i$, $H^i(X)$ is an 
$A$-module. We have natural morphisms of $A$-algebras
$$A\to \Hom_{{\mathcal A}}(X,X)\to \End_A(H^i(X)).$$
By definition, $\Hom_{{\mathcal A}}(X,X)=\Bbbk$. 
Thus $\Hom_{{\mathcal A}}(X,X)=A/\fm$ for some ideal $\fm$ of 
$A$ that has codimension 1. Hence the $A$-action on $H^i(X)$ 
factors through the map $A\to A/\fm$. This means that $H^i(X)$ 
is a direct sum of $A/\fm$. 

Let $n=\sup X$ and $m=\inf X$. Then $H^m(X)=(A/\fm)^{\oplus s}$
and $H^n(X)=(A/\fm)^{\oplus t}$ for some $s,t>0$. If $m<n$, 
then
$$\Hom_{{\mathcal A}}(X, X[m-n])\cong \Hom_{{\mathcal A}}(X[n], X[m])
\cong \Hom_{A}(H^n(X), H^m(X))\neq 0$$
which contradicts \eqref{E2.1.1}. Therefore $m=n$ and $X=M[n]$
for $M:=H^n(X)$. Since $X$ is atomic, $M$ has only one copy of $A/\fm$.
\end{proof}

%

\begin{lemma}
\label{xxlem4.2}
Let $A$ be a commutative algebra. Let $X$ and $Y$ be two 
atomic objects in ${\mathcal A}$.
Then $\Hom_{\mathcal A}(X,Y)\neq 0$ if and only if there is an ideal 
$\fm$ of $A$ of codimension 1 such that $X\cong A/\fm[m]$
and $Y\cong A/\fm[n]$ for some $0\leq n-m\leq \pd A/\fm$.
\end{lemma}

\begin{proof} By Lemma \ref{xxlem4.1}, $X\cong A/\fm[m]$ for some 
ideal $\fm$ of codimension 1 and some integer $m$. Similarly, 
$Y\cong A/\fn[n]$ for ideal $\fn$ of codimension 1 and integer $n$. 

Suppose $\Hom_{\mathcal A}(X,Y)\neq 0$. 
If $\fm\neq \fn$, then clearly $\Hom_{\mathcal A}(X,Y)=0$. Hence $\fm=\fn$.
Further, $\Ext^{n-m}_{A}(A/\fm,A/\fm)\cong \Hom_{\mathcal A}(X,Y)\neq 0$
implies that $0\leq n-m\leq \pd A/\fm$. The converse can be proved in a similar
way.
\end{proof}

If $A$ is an affine commutative ring over $\Bbbk$, then 
every simple $A$-module is $1$-dimensional. Hence 
$(A/\fm)[i]$ is a brick (and atomic) object in ${\mathcal A}$
for every $i\in {\mathbb Z}$ and every maximal ideal $\fm$ of $A$.
The fp-global dimension $\fpgldim ({\mathcal A})$ is defined 
in Definition \ref{xxdef2.7}(2).

\begin{proposition}
\label{xxpro4.3}
Let $A$ be an affine commutative domain of global dimension
$g<\infty$. 
\begin{enumerate}
\item[(1)]
$\fpdim ({\mathcal A})=g$.
\item[(2)]
$\fpdim (\Sigma^i)={g \choose i}$ for all $i$.
\item[(3)]
$\fpgldim ({\mathcal A})=g$.
\end{enumerate}
\end{proposition}

\begin{proof}
(1) By Lemma \ref{xxlem4.1}, every atomic object is of the
form $M[i]$ for some $M\cong A/\fm$ where $\fm$ is an ideal of
codimension 1, and $i\in {\mathbb Z}$. It is well-known that
\begin{equation}
\label{E4.3.1}\tag{E4.3.1}
\dim \Ext^i_A(A/\fm, A/\fm)={g \choose i}, 
\;\; \forall \;\; i.
\end{equation}
If $\fm_1$ and $\fm_2$ are two different maximal ideals, then 
\begin{equation}
\label{E4.3.2}\tag{E4.3.2}
\Ext^i_A(A/\fm_1, A/\fm_2)=0
\end{equation}
for all $i$. 
Let $\phi$ be an atomic $n$-object subset. We can decompose 
$\phi$ into a disjoint union 
$\phi_{A/\fm_1}\cup \cdots \cup \phi_{A/\fm_s}$ 
where $\phi_{A/\fm}$ consists of objects of the form
$A/\fm [i]$ for $i\in {\mathbb Z}$. It follows from 
\eqref{E4.3.2} that the adjacency matrix
is a block-diagonal matrix. Hence, we only need
to consider the case when $\phi=\phi_{A/\fm}$
after we use the reduction similar to 
the one used in the proof of Theorem \ref{xxthm3.5}.
Let $\phi=\phi_{A/\fm}=\{A/\fm [d_1], \cdots, A/\fm [d_m]\}$ 
where $d_i$ is increasing. By Lemma \ref{xxlem4.2}, we 
have $d_{i+1}-d_{i}>g$, or 
$d_i+g< d_{i+1}$, for all $i=1,\cdots, m-1$.
Under these conditions, the adjacency matrix is lower 
triangular with each diagonal being $g$. Thus 
$\fpdim (\Sigma)=g$.

The proof of (2) is similar. (3) is a consequence of (2).
\end{proof}

Suggested by Theorem \ref{xxthm3.5}, we could introduce some 
secondary invariants as follows. The {\it stabilization index}
of a triangulated category ${\mathcal T}$ is defined to be
$$SI({\mathcal T})=\min \{ n\mid \fpdim^{n'}{\mathcal T}=\fpdim 
{\mathcal T}, \; \forall \; n'\geq n\}.$$
The {\it global stabilization index} of ${\mathcal T}$ is defined to
be 
$$GSI({\mathcal T})=\max\{ SI({\mathcal T}')\mid {\text{for}}\;\;
{\text{all}}\;\; {\text{thick}}\;\; {\text{triangulated}}\;\;
{\text{full}}\;\; {\text{subcategories}}\;\; {\mathcal T}'\subseteq
{\mathcal T}\}.$$
It is clear that both stabilization index and global stabilization 
index can be defined for an abelian category.

Similar to Proposition \ref{xxpro4.3}, one can show the following.
Suppose that $A$ is affine.
For every $i$, let
$$d_i:= \sup\{ \dim \Ext^i_{A}(A/\fm,A/\fm)\mid \; \;
{\rm{maximal\; ideals}}\;\; \fm\subseteq A\}.$$

\begin{proposition}
\label{xxpro4.4}
Let $A$ be an affine commutative algebra. 
Then, for each $i$, $\fpdim (\Sigma^i)=d_i<\infty$ 
and $\rho(A(\phi, \Sigma^i))\leq d_i$ for all $\phi\in \Phi_{n,a}$.
As a consequence, for each integer $i$, the following hold.
\begin{enumerate}
\item[(1)]
$\fpdim(\Sigma^i)=\fpdim^1(\Sigma^i)$. Hence the stabilization 
index of ${\mathcal A}$ is 1.
\item[(2)]
$\fpdim(\Sigma^i)$ is a finite integer.
\end{enumerate}
\end{proposition}

\begin{theorem}
\label{xxthm4.5}
Let $A$ be an affine commutative algebra and ${\mathcal A}$ be $D^b(\Mod−A)$.
Let ${\mathcal T}$ be a triangulated full subcategory of ${\mathcal A}$
with suspension $\Sigma_{\mathcal T}$. Let $i$ be an integer.
\begin{enumerate}
\item[(1)]
$\fpdim(\Sigma^i_{\mathcal T})=\fpdim^1(\Sigma^i_{\mathcal T})$. 
As a consequence, the global stabilization index
of ${\mathcal A}$ is 1.
\item[(2)]
$\fpdim(\Sigma^i_{\mathcal T})$ is a finite integer.
\item[(3)]
If ${\mathcal T}$ is isomorphic to $D^b(\Mod_{f.d.}-B)$ for some
finite dimensional algebra $B$, then $B$ is Morita equivalent to 
a commutative algebra.
\end{enumerate}
\end{theorem}

\begin{proof} (1,2) These are similar to Proposition \ref{xxpro4.4}.

(3) Since $B$ is finite dimensional, it is Morita 
equivalent to  a basic algebra. So we can assume $B$ is basic and
show that $B$ is commutative. Write $B$ as a $\Bbbk Q/(R)$ where $Q$
is a finite quiver with admissible ideal $R\subseteq (\Bbbk Q)_{\geq 2}$.
We will show that $B$ is commutative.

First we claim that each connected component of $Q$ consists 
of only one vertex. Suppose not. Then $Q$ contains distinct vertices $v_1$ and $v_2$ with 
an arrow $\alpha: v_1\to v_2$. Let $S_1$ and $S_2$ be the simple modules 
corresponding to $v_1$ and $v_2$ respectively. Then $\{S_1, S_2\}$ is an
atomic set in ${\mathcal T}$. The arrow represents a nonzero 
element in $\Ext^1_{B}(S_1,S_2)$. Hence
$$\Hom_{\mathcal T}(S_1, S_2[1])\cong \Ext^1_{B}(S_1,S_2)\neq 0.$$
By Lemma \ref{xxlem4.2}, $S_1$ is isomorphic to a complex shift of 
$S_2$. But this is impossible. Therefore, the claim holds.

It follows from the claim in the last paragraph that $B=B_1\oplus 
\cdots \oplus B_n$ where each $B_i$ is a finite dimensional local ring
corresponding to a vertex, say $v_i$.
Next we claim that each $B_i$ is commutative. Without loss of generality,
we can assume $B_i=B$.

Now let $\iota$ be the fully faithful embedding from 
$$\iota: {\mathcal T}:=D^b(\Mod_{f.d.}-B)
\longrightarrow {\mathcal A}:=D^b(\Mod-A).$$ 
Let $S$ be the unique simple left $B$-module. Then, by Lemma \ref{xxlem4.1}, 
there is a maximal ideal $\fm$ of $A$ such that $\iota(S)=A/\fm[w]$ for some 
$w\in {\mathbb Z}$. 
After a shift, we might assume that $\iota(S)=A/\fm$. The left $B$-module 
$B$ has a composition series such that each simple subquotient is isomorphic 
to $S$, which implies that, as a left $A$-module, $\iota(B)$ is generated by
$A/\fm$ in ${\mathcal A}$. By induction on the length of
$B$, one sees that, for every $n\in {\mathbb Z}$, $H^n(\iota(B))$ is  
a left $A/\fm^d$-module for some $d\gg 0$ (we can take $d={\rm{length}} (_BB)$). 
Since $\Hom_{\mathcal A}(\iota(B), \iota(B)[-i])=
\Hom_{\mathcal T}(B, B[-i])=0$ for all $i>0$,
the proof of Lemma \ref{xxlem3.11}(2) shows that $\iota(B)
\cong M[m]$ for some left $A/\fm^d$-module $M$ and $m\in {\mathbb Z}$.
Since there are nonzero maps from $S$ to $B$ and from $B$ to $S$, 
we have nonzero maps from $A/\fm$ to $\iota(B)$ and from 
$\iota(B)$ to $A/\fm$. This implies that $m=0$. Since $B$ is local
(and then $B/\fm_B$ is 1-dimensional for the maximal 
ideal $\fm_B$), this forces that $M=A/I$ where $I$ is an 
ideal of $A$ containing $\fm^d$. Finally, 
$$B^{op}=\End_{B}(B)
\cong \End_{\mathcal A}(A/I,A/I)
=\End_{A}(A/I,A/I)\cong A/I$$
which is commutative. Hence $B$ is commutative.
\end{proof}

\section{Examples}
\label{xxsec5}
In this section we give three examples.

\subsection{Frobenius-Perron theory of projective line ${\mathbb P}^1
:=\Proj \Bbbk [t_0,t_1]$}
\label{xxsec5.1}

\begin{example}
\label{xxex5.1} 
Let $coh({\mathbb P}^1)=:{\mathfrak A}$ denote the category of 
coherent sheaves on ${\mathbb P}^1$. It is well-known (and 
follows from \cite[Example 3.18]{BB}) that 

\medskip

\noindent
{\bf Claim 5.1.1:} Every brick object $X$ in ${\mathfrak A}$ 
(namely, satisfying $\Hom_{{\mathbb P}^1}(X,X)=\Bbbk$)  is either
${\mathcal O}(m)$ for some $m\in {\mathbb Z}$ or 
${\mathcal O}_p$ for some $p\in {\mathbb P}^1$.

\medskip

Let $\phi$ be in $\Phi_{n,b}(coh({\mathbb P}^1))$. If $n=1$ 
or $\phi$ is a singleton, then there are two cases: either 
$\phi=\{{\mathcal O}(m)\}$ or $\phi=\{{\mathcal O}_p\}$. Let 
$E^1$ be the functor $\Ext^1_{{\mathbb P}^1}(-,-)$. In the 
first case, $\rho(A(\phi, E^1))=0$ because 
$\Ext^1_{{\mathbb P}^1}({\mathcal O}(m), {\mathcal O}(m))=0$, 
and in the second case, $\rho(A(\phi, E^1))=1$ because 
$\Ext^1_{{\mathbb P}^1}({\mathcal O}_p, {\mathcal O}_p)=1$.

If $|\phi|>1$, then ${\mathcal O}(m)$ can not appear in $\phi$ 
as $\Hom_{{\mathbb P}^1}({\mathcal O}(m),{\mathcal O}(m'))\neq 0$ 
and $\Hom_{{\mathbb P}^1}({\mathcal O}(m), {\mathcal O}_p)\neq 0$
for all $m\leq m'$ and $p\in {\mathbb P}^1$. Hence, 
$\phi$ is a collection of ${\mathcal O}_p$ for finitely many
distinct points $p$'s.
In this case, the adjacency matrix is the identity $n\times n$-matrix 
and $\rho(A(\phi, E^1))=1$. Therefore
\begin{equation}
\label{E5.1.1}\tag{E5.1.1}
\fpdim^n(coh({\mathbb P}^1))=\fpdim(coh({\mathbb P}^1))=1
\end{equation}
for all $n\geq 1$. Since $coh({\mathbb P}^1)$ is hereditary, 
by Theorem \ref{xxthm3.5}(3,4), we obtain that
\begin{equation}
\label{E5.1.2}\tag{E5.1.2}
\fpdim^n(D^b(coh({\mathbb P}^1)))=\fpdim(D^b(coh({\mathbb P}^1)))=1
\end{equation}
for all $n\geq 1$. 

Let $K_2$ be the Kronecker quiver 

\begin{equation}
\label{E5.1.3}\tag{E5.1.3}
\xymatrix{ \bullet\ar@/^1pc/[rr]\ar@/_1pc/[rr]&& \bullet\\}
\end{equation}

\bigskip

By a result of Beilinson \cite{Bei}, the derived category 
$D^b(\Mod_{f.d.}-\Bbbk K_2)$ is triangulated equivalent to 
$D^b(coh({\mathbb P}^1))$. As a consequence, 
\begin{equation}
\label{E5.1.4}\tag{E5.1.4}
\fpdim(D^b(\Mod_{f.d.}-\Bbbk K_2))=\fpdim(D^b(coh({\mathbb P}^1)))=1.
\end{equation}
It is easy to see, or by Theorem \ref{xxthm1.8}(1),
$$\fpdim K_2=0$$
where $\fpdim $ of a quiver is defined in Definition \ref{xxdef1.6}.

This implies that 
\begin{equation}
\label{E5.1.5}\tag{E5.1.5}
\fpdim(D^b(\Mod_{f.d.}-\Bbbk K_2))>\fpdim K_2.
\end{equation}

\medskip

Next we consider some general auto-equivalences of $D^b(coh({\mathbb P}^1))$.
Let 
$$(m): coh({\mathbb P}^1)\to coh({\mathbb P}^1)$$
be the auto-equivalence induced by the shift of degree $m$ of the graded modules 
over $\Bbbk [t_0,t_1]$ and let $\Sigma$ be 
the suspension functor of $D^b(coh({\mathbb P}^1))$. Then the Serre functor 
$S$ of $D^b(coh({\mathbb P}^1))$ is $\Sigma \circ (-2)$. Let $\sigma$ be the 
functor $\Sigma^a \circ (b)$ for 
some $a,b\in {\mathbb Z}$. By Theorem \ref{xxthm3.5}(1),
$$\fpdim^n(\Sigma^a \circ (b))=0, \;\; \forall \; a\neq 0,1.$$
For the rest we consider $a=0$ or $1$. By Theorem \ref{xxthm3.5}(2,3),
we only need to consider $\fpdim$ on $coh({\mathbb P}^1)$.

If $\phi$  is a singleton $\{{\mathcal O}(n)\}$, then the adjacency 
matrix is 
$$A(\phi,\sigma)=\dim ({\mathcal O}, \Sigma^a{\mathcal O}(b))
=\begin{cases} 
0 &a=0, b<0,\\
b+1 & a=0, b\geq 0,\\
0& a=1, b\geq -1,\\
-b-1& a=1, b<-1.
\end{cases}
$$
This follows from the well-known computation of 
$H^i_{{\mathbb P}^1}({\mathcal O}(m))$ for $i=0,1$ and $m\in {\mathbb Z}$.
(It also follows from a more general computation 
\cite[Theorem 8.1]{AZ}.)
If $\phi=\{{\mathcal O}_p\}$ for some $p\in {\mathbb P}^1$, then the adjacency 
matrix is 
$$A(\phi,\sigma)=\dim ({\mathcal O}_p, \Sigma^a({\mathcal O}_p))
=1, \;\; {\rm{for}}\;\; a=0,1.
$$
It is easy to see from the above computation that
\begin{equation}
\label{E5.1.6}\tag{E5.1.6}
\fpdim^1(\Sigma^a \circ (b))
=\begin{cases} 1 &a=0, b<0,\\
b+1 & a=0, b\geq 0,\\
1& a=1, b\geq -1,\\
-b-1& a=1, b<-1.
\end{cases}
\end{equation}

Now we consider the case when $n>1$.  If $\phi\in 
\Phi_{n,b}(coh({\mathbb P}^1))$, $\phi$ 
is a collection of ${\mathcal O}_p$ for finitely many 
distinct $p$'s.
In this case, the adjacency matrix $A(\phi, \Sigma^a \circ (b))$ 
is the identity $n\times n$-matrix for $a=0,1$, and 
$\rho(A(\phi, \sigma))=1$. Therefore
\begin{equation}
\label{E5.1.7}\tag{E5.1.7}
\fpdim^n(\Sigma^a \circ (b))=1
\end{equation}
for all $n> 1$, when restricted to the category $coh({\mathbb P}^1)$. 

It follows from Theorem \ref{xxthm3.5}(2,3) that\\
\noindent
{\bf Claim 5.1.2:} Consider $\Sigma^a \circ (b)$ 
as an endofunctor of $D^b(coh({\mathbb P}^1))$.
For $a,b\in {\mathbb Z}$ and 
$n\geq 1$, we have 
\begin{equation}
\label{E5.1.8}\tag{E5.1.8}
\fpdim^n(\Sigma^a \circ (b))
=\begin{cases} 0 & a\neq 0, 1,\\
1 &a=0, b<0,\\
b+1 & a=0, b\geq 0,\\
1& a=1, b\geq -1,\\
-b-1& a=1, b<-1.
\end{cases}
\end{equation}

Since $S=\Sigma \circ (-2)$, we have the following (also see Figure 1, next page)
\begin{equation}
\label{E5.1.9}\tag{E5.1.9}
\fpdim^n(\Sigma^a \circ S^b )=\fpdim^n(\Sigma^{a+b} \circ (-2b) )
=\begin{cases} 0 & a+b\neq 0, 1,\\
1 &a+b=0, b>0,\\
-(2b-1) & a+b=0, b\leq 0,\\
1& a+b=1, b\leq 0,\\
2b-1& a+b=1, b>0.
\end{cases}
\end{equation}

\noindent
{\bf Claim 5.1.3:}
Since $D^b(coh({\mathbb P}^1))$ and $D^b(\Mod_{f.d.}-\Bbbk K_2)$ 
are equivalent, the 
fp-theory of $D^b(\Mod_{f.d.}-\Bbbk K_2)$ agrees with \eqref{E5.1.9} 
and Figure 1 (next page).
\end{example}

\begin{figure}
    \centering
    \includegraphics{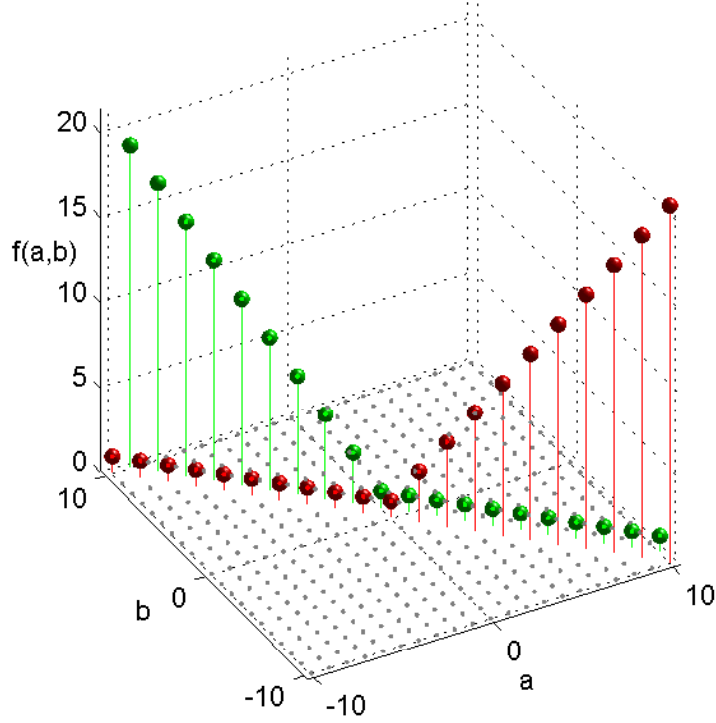}
    \caption{fp-$S$-theory for ${\mathbb P}^1$, where $f(a,b)=\fpdim(\Sigma^a\circ S^b)$.}
\end{figure}


\subsection{Frobenius-Perron theory of the quiver $A_2$}
\label{xxsec5.2}

We start with the following example. 

\begin{example}
\label{xxex5.2} 
Let $A$ be the ${\mathbb Z}$-graded algebra $\Bbbk[x]/(x^2)$
with $\deg x=1$. Let ${\mathcal C}:=\GKgr-A$ be the category of finitely generated 
graded left $A$-modules. Let $\sigma:=(-)$ be the degree shift functor 
of ${\mathcal C}$. It is clear that $\sigma$ is an autoequivalence of ${\mathcal C}$.
Let ${\mathfrak A}$ be the additive subcategory of ${\mathcal C}$ generated by 
$\sigma^n (A) =A(n)$ for all $n\in {\mathbb Z}$. Note that ${\mathfrak A}$ is not
abelian and that every object in ${\mathfrak A}$ is of the form $\bigoplus_{n\in {\mathbb Z}}
A(n)^{\oplus p_n}$ for some integers $p_n\geq 0$. Since the Hom-set in the 
graded module category consists of homomorphisms of degree zero, we have
$$\Hom_{\mathfrak A}(A,A(n))=\begin{cases} \Bbbk & n=0,1\\ 0& {\rm{otherwise}}
\end{cases}.$$
In the following diagram each arrow represents $1$-dimensional Hom for all 
possible $\Hom$-set for different objects $A(n)$
\begin{equation}
\label{E5.2.1}\tag{E5.2.1}
\cdots \longrightarrow A(-2) \longrightarrow A(-1) 
\longrightarrow A(0) \longrightarrow A(1) 
\longrightarrow A(2) \longrightarrow \cdots
\end{equation} 
(where the number of arrows from $A(m)$ to $A(n)$ agrees with $\dim \Hom(A(m),A(n))$).
It is easy to see that the set of indecomposable objects is
$\{ A(n)\}_{n\in {\mathbb Z}}$, which is also the set of bricks in 
${\mathfrak A}$. 
\end{example}

\begin{lemma}
\label{xxlem5.3}
Retain the notation as in Example \ref{xxex5.2}. 
When restricting $\sigma$ onto the category 
${\mathfrak A}$, we have, for every $m\geq 1$,
\begin{equation}
\label{E5.3.1}\tag{E5.3.1}
\fpdim^m(\sigma^n)=\begin{cases} 1 & n=0, 1, \\
0&{\rm{otherwise}} \end{cases}.
\end{equation}
\end{lemma}

\begin{proof}
When $n=0$, \eqref{E5.3.1} is trivial. Let $n=1$. For each set
$\phi\in \Phi_{m,b}$, we can assume that 
$\phi=\{A(d_1), A(d_2,),\cdots, A(d_m)\}$
for a strictly increasing sequence $\{d_i\mid i=1,2,\cdots,m\}$. 
For any $i< j$, the $(i,j)$-entry of the adjacency matrix is 
$$a_{ij}=\dim(A(d_i), A(d_j+1))=0.$$ 
Thus $A(\phi, \sigma)$ is a 
lower triangular matrix with 
$$a_{ii}=\dim(A(d_i), A(d_i+1))=1.$$ 
Hence $\rho(A(\phi,\sigma))=1$. So $\fpdim^m(\sigma)=1$.

Similarly, $\fpdim^m(\sigma^n)=0$ when $n>1$ as 
$\dim(A(d_i),A(d_i+2))=0$ for all $i$.

Let $n<0$. Let $\phi=\{A(d_1), A(d_2,),\cdots, A(d_m)\}\in \Phi_{m,b}$ where $d_i$ are 
strictly decreasing. Then $a_{ij}=\dim(A(d_i), A(d_j+n))=0$ for all $i\leq j$. Thus 
$\rho(A(\phi, \sigma^{n}))=0$ and \eqref{E5.3.1} follows in this case.
\end{proof}

\begin{example}
\label{xxex5.4} 
Consider the quiver $A_2$
\begin{equation}
\label{E5.4.1}\tag{E5.4.1}
\bullet_2 \longleftarrow \bullet_1 .
\end{equation}
Let $P_i$ (respectively, $I_i$) be the projective 
(respectively, injective) left $\Bbbk A_2$-modules 
corresponding to vertices $i$, for $i=1,2$,
It is well-known that there are only three indecomposable 
left modules over $A_2$, with Auslander-Reiten quiver (or 
AR-quiver, for short)
\begin{equation}
\label{E5.4.2}\tag{E5.4.2}
P_2 \longrightarrow P_1(=I_2)\longrightarrow I_1
\end{equation}
where each arrow represents a nonzero homomorphism (up to a scalar)
\cite[Ex.1.13, pp.24-25]{Sc1}.
The AR-translation (or translation, for short) $\tau$ is determined by 
$\tau (I_1)=P_2$. 
Let ${\mathcal T}$ be $D^b(\Mod_{f.d.}-\Bbbk A_2)$. 
The  Auslander-Reiten theory can be extended
from the module category to the derived category. It is direct that,
in ${\mathcal T}$, we have the AR-quiver of all indecomposable objects


\begin{equation}
\label{E5.4.3}\tag{E5.4.3}
\quad 
\end{equation}

\unitlength=1mm
\begin{picture}(100,16)
\put(14,10){\vector(1,-1){7}}
\put(28,4){\vector(1,1){6}}
\put(40,10){\vector(1,-1){7}}
\put(52,4){\vector(1,1){6}}
\put(64,10){\vector(1,-1){7}}
\put(76,4){\vector(1,1){6}}
\put(88,10){\vector(1,-1){7}}
\put(100,4){\vector(1,1){6}}
\put(8,12){$P_{2}[-1]$}
\put(19,0){$P_{1}[-1]$}
\put(31,12){$I_{1}[-1]$}
\put(48,0){$P_{2}$}
\put(56,12){$P_{1}=I_{2}$}
\put(72,0){$I_{1}$}
\put(82,12){$P_{2}[1]$}
\put(94,0){$P_{1}[1]$}
\put(104,12){$I_{1}[1]$}
\put(4,6){$\ldots$}
\put(112,6){$\ldots$}
\end{picture}

\vspace{5mm}

The above represents all possible nonzero morphisms (up to a scalar)
between non-isomorphic indecomposable objects in ${\mathcal T}$.
Note that ${\mathcal T}$ has a Serre functor $S$ and that the AR-translation
$\tau$ can be extended to a functor of ${\mathcal T}$ 
such that $S=\Sigma \circ \tau$ \cite[Proposition I.2.3]{RVdB} or 
\cite[Remarks (2), p. 23]{Cr}. After
we identifying 
$$P_2[i]\leftrightarrow A(3i), \quad P_1[i]\leftrightarrow A(3i+1), 
\quad I_1[i]\leftrightarrow A(3i+2),$$
\eqref{E5.4.3} agrees with \eqref{E5.2.1}. Using the above identification, 
at least when restricted to objects, we have 
\begin{align}
\label{E5.4.4}\tag{E5.4.4}
\Sigma ( A(i))&\cong A(i+3),\\
\label{E5.4.5}\tag{E5.4.5}
\tau ( A(i)) &\cong A(i-2), \\
\label{E5.4.6}\tag{E5.4.6}
S ( A(i)) &\cong A(i+1).
\end{align}
It follows from the definition of the AR-quiver \cite[VII]{ARS} that the 
degree of $\tau$ is $-2$, see also \cite[Picture on p. 131]{AS}. Equation
\eqref{E5.4.5} just means that the degree of $\tau$ is $-2$.

By equation \eqref{E5.4.6}, the Serre functor $S$ satisfies the property 
of $\sigma$ defined in Example \ref{xxex5.2}. By Lemma \ref{xxlem5.3} 
or \eqref{E5.3.1}, we have
$$\fpdim^n(\Sigma^a \circ S^b)=
\fpdim^n(\sigma^{3a+b})=\begin{cases} 1 & 3a+b=0,1, \\
0 &{\rm{otherwise}}.\end{cases}$$
Therefore the fp-$S$-theory of ${\mathcal T}$ is given as above, and given 
as in Figure 2 (next page). 

\begin{figure}
    \centering
    \includegraphics{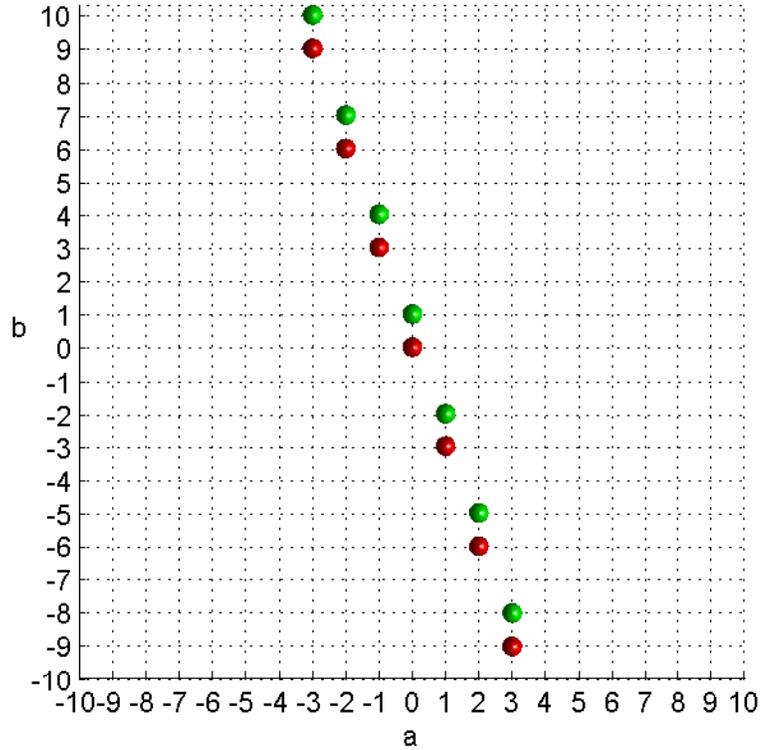}
    \caption{fp-$S$-theory for $A_2$}
\end{figure}

In particular, we have proven 
$$\fpgldim (D^b(\Mod_{f.d.}-\Bbbk A_2)) =
\fpdim (D^b(\Mod_{f.d.}-\Bbbk A_2))=\fpdim (\Sigma)=0,$$
which is less than $\gldim \Bbbk A_2=1$.
\end{example}

\subsection{An example of non-integral
Frobenius-Perron dimension}
\label{xxsec5.3}

In the next example, we ``glue'' $K_2$ in \eqref{E5.1.3} 
and $A_2$ in \eqref{E5.4.1} together.

\begin{example}
\label{xxex5.5}
Let $G_2$ be the quiver 
\begin{equation}
\label{E5.5.1}\tag{E5.5.1}
\xymatrix{ {\stackrel{\bullet}{1}}\ar@/^1pc/[rr]_{\gamma}\ar@/^1.8pc/[rr]^{\beta}&& 
{\stackrel{\bullet}{2}}\ar@/^1.3pc/[ll]^{\alpha}\\}
\end{equation}
consisting of two vertices
$1$ and $2$, with arrow $\alpha: 2\to 1$
and $\beta, \gamma: 1\to 2$ satisfying relations
\begin{equation}
\label{E5.5.2}\tag{E5.5.2}
R: \quad \beta \alpha=\gamma\alpha =0, 
\alpha \beta=\alpha \gamma =0.
\end{equation}
Note that $(G_2,R)$ is a quiver with relations. 
The corresponding quiver algebra with relations
is a 5-dimensional algebra
$$A=\Bbbk e_1+ \Bbbk e_2+ \Bbbk \alpha +\Bbbk 
\beta +\Bbbk \gamma$$
satisfying, in addition to \eqref{E5.5.2}, 
\begin{equation}
\label{E5.5.3}\tag{E5.5.3}
e_1^2=e_1, \quad e_2^2=e_2, \quad 1=e_1+e_2,
\end{equation}
\begin{equation}
\label{E5.5.4}\tag{E5.5.4}
e_2\alpha=0, \quad \alpha e_1=0, 
\quad e_1\beta=e_1\gamma=0, \quad 
\beta e_2=\gamma e_2=0,
\end{equation}
and
\begin{equation}
\label{E5.5.5}\tag{E5.5.5}
\alpha e_2=\alpha=e_1 \alpha, \quad
\beta e_1=\beta=e_2\beta, \quad 
\gamma e_1=\gamma=e_2\gamma.
\end{equation}
We can use the following matrix form 
to represent the algebra $A$
$$A=\begin{pmatrix}
\Bbbk e_1 & \Bbbk \alpha\\
\Bbbk \beta +\Bbbk \gamma & \Bbbk e_2
\end{pmatrix}.
$$

For each $i=1,2$, let $S_i$ be the left simple $A$-module
corresponding to the vertex $i$ and $P_i$ be the 
projective cover of $S_i$. Then $P_1\cong Ae_1$ is isomorphic
to the first column of $A$, namely, 
$\begin{pmatrix} \Bbbk e_1\\
\Bbbk \beta +\Bbbk \gamma\end{pmatrix}$; 
and $P_2\cong Ae_2$ is isomorphic to the second 
column of $A$, namely, $\begin{pmatrix}\Bbbk \alpha \\ 
\Bbbk e_2\end{pmatrix}$. 

We will show that the Frobenius-Perron dimension of the category 
of finite dimensional representations of $(G_2,R)$ is $\sqrt{2}$, by 
using several lemmas below that contain some detailed computations.
\end{example}

\begin{lemma}
\label{xxlem5.6}
Let $V=(V_1,V_2)$ be a representation of $(G_2,R)$. Let 
$\overline{W}={\rm{im}} \;\alpha$
and $K=\ker \alpha$. Take a $\Bbbk$-space decomposition $V_2=W\oplus K$ 
where $W\cong \overline{W}$. 
Then there is a decomposition of $(G_2,R)$-representations 
$V\cong (\overline{W}\oplus T, W\oplus K)\cong (\overline{W}, W)\oplus (T,K)$
where $\alpha$ is the identity when restricted to $W$ 
{\rm{(}}and identifying $W$ with
$\overline{W}${\rm{)}}  and is zero
when restricted to $K$, where $\beta$ and $\gamma$ are zero 
when restricted to $\overline{W}$. 
\end{lemma}

\begin{proof} Since $\overline{W}={\rm{im}} \; \alpha$, $V_2\cong W\oplus K$
where $K=\ker \alpha$ and $W\cong \overline{W}$. Write $V_1=\overline{W}\oplus T$ for some
$\Bbbk$-subspace $T\subseteq V_1$. The assertion follows by using the
relations in \eqref{E5.5.2}.
\end{proof}

Recall that $A_2$ is the quiver given in \eqref{E5.4.1}
and $K_2$ is the Kronecker quiver given in \eqref{E5.1.3}.
By the above lemma, the subrepresentation $(W,W)$ (where 
we identify $\overline{W}$ with $W$) is in fact 
a representation of $\begin{pmatrix}
\Bbbk e_1 & \Bbbk \alpha\\
 0 & \Bbbk e_2
\end{pmatrix}(\cong \Bbbk A_2)$ and the 
subrepresentation $(T,K)$ is a representation of
$\begin{pmatrix}
\Bbbk e_1 & 0\\
\Bbbk \beta +\Bbbk \gamma& \Bbbk e_2
\end{pmatrix}(\cong \Bbbk K_2)$.

Let $I_n$ be the $n\times n$-identity matrix.
Let $Bl(\lambda)$ denote the block matrix
$$\begin{pmatrix} 
\lambda & 1     & 0& \cdots & 0 & 0\\
0&      \lambda & 1 & \cdots & 0& 0\\
\ldots & \ldots & \ldots &\ldots & \ldots & \ldots\\
0& 0& 0& \cdots & \lambda& 1\\
0& 0& 0& \cdots & 0& \lambda
\end{pmatrix}.
$$

\begin{lemma}
\label{xxlem5.7} Suppose $\Bbbk$ is 
of characteristic zero.
The following is a complete list of indecomposable
representations of $(G_2,R)$.
\begin{enumerate}
\item[(1)]
$P_2\cong (\Bbbk, \Bbbk)$, where $\alpha=I_1$ and 
$\beta=\gamma=0$.
\item[(2)]
$X_n(\lambda)=(K, K)$ with $\dim K=n$, where $\alpha=0$,
$\beta=I_n$ and $\gamma=Bl(\lambda)$ for some $\lambda\in \Bbbk$.
\item[(3)]
$Y_n=(K, K)$ with $\dim K=n$, where $\alpha=0$,
$\beta=Bl(0)$ and $\gamma=I_n$.
\item[(4)]
$S_{2,n}=(T,K)$ with $\dim T=n$ and $\dim K=n+1$, where
$\alpha=0$, $\beta=(I_n, 0)$ and $\gamma=(0,I_n)$.
\item[(5)]
$S_{1,n}=(T,K)$ with $\dim T=n+1$ and $\dim K=n$, where
$\alpha=0$, $\beta=(I_n, 0)^\tau$ and $\gamma=(0,I_n)^\tau$.
\end{enumerate}
As a consequence, $\Bbbk G_2/(R)$ is of tame representation type
{\rm{[}}Definition \ref{xxdef7.1}{\rm{]}}.
\end{lemma}

\begin{proof} (1) By Lemma \ref{xxlem5.6}, 
this is the only case that could happen 
when $\alpha\neq 0$. Now we assume $\alpha=0$. 

(2,3,4,5) If $\alpha=0$, then we are working with 
representations of Kronecker quiver $K_2$ \eqref{E5.1.3}.
The classification follows from a classical result of Kronecker 
\cite[Theorem 4.3.2]{Ben}.

By (1-5), for each integer $n$, there are only finitely many
1-parameter families of indecomposable representations of 
dimension $n$. Therefore $A$ is of tame representation type. 
\end{proof}

The following is a consequence of Lemma \ref{xxlem5.7}
and a direct computation. 

\begin{lemma}
\label{xxlem5.8} Retain the hypotheses of Lemma \ref{xxlem5.7}.
The following is a complete list of brick 
representations of $(G_2,R)$.
\begin{enumerate}
\item[(1)]
$P_2\cong (\Bbbk, \Bbbk)$, where $\alpha=I_1$ and 
$\beta=\gamma=0$.
\item[(2)]
$X_1(\lambda)=(\Bbbk, \Bbbk)$, where $\alpha=0$,
$\beta=I_1$ and $\gamma=\lambda I_1$ for some $\lambda\in \Bbbk$.
\item[(3)]
$Y_1=(\Bbbk, \Bbbk)$, where $\alpha=0$,
$\beta=0$ and $\gamma=I_1$.
\item[(4)]
$S_{2,n}$ for $n\geq 0$.
\item[(5)]
$S_{1,n}$ for $n\geq 0$.
\end{enumerate}
The set $\Phi_{1,b}$ consists of the above objects.
\end{lemma}

Let $X_1(\infty)$ denote $Y_1$. We have the following
short exact sequences of 
$(G_2,R)$-representations
$$0\to S_1\to P_2\to S_2\to 0,$$
$$0\to S_2\to X_1(\lambda)\to S_1\to 0,$$
$$0\to S_{2}^{n+1}\to S_{2,n}\to S_1^{n}\to 0,$$
$$0\to S_2^{n}\to S_{1,n}\to S_{1}^{n+1}\to 0,$$
$$0\to S_2^{2}\to S_{2,n}\to S_{1,n-1}\to 0,$$
where $n\geq 1$ for the last exact sequence,
and have the following nonzero homs, where $A=\Bbbk G_2/(R)$,
$$
\begin{aligned}
\Hom_{A}(X_1(\lambda), S_{1,n})&\neq 0, \;\; \forall \; n\geq 1\\
\Hom_{A}(S_{2,n}, X_1(\lambda))&\neq 0, \;\; \forall \; n\geq 1\\
\Hom_{A}(S_{2,m}, S_{2,n})&\neq 0, \;\; \forall \; m\leq n\\
\Hom_{A}(S_{1,n}, S_{1,m})&\neq 0, \;\; \forall \; m\leq n\\
\Hom_{A}(S_{2,n}, S_{1,m})&\neq 0, \;\; \forall \; m+n\geq 1
\end{aligned}
$$

\begin{lemma}
\label{xxlem5.9} 
Retain the hypotheses of Lemma \ref{xxlem5.7}.
The following is  the complete list of zero hom-sets between 
brick representations of $G_2$ in both directions.
\begin{enumerate}
\item[(1)]
$\Hom_{A}(X_1(\lambda), X_1(\lambda'))=
\Hom_{A}(X_1(\lambda'), X_1(\lambda))=0$ if $\lambda\neq \lambda'$
in $\Bbbk \cup\{\infty\}$.
\item[(2)]
$\Hom_{A}(S_1, S_2)=\Hom_{A}(S_2, S_1)=0$.
\end{enumerate}
As a consequence, if $\phi\in \Phi_{n,b}$ for some $n\geq 2$,
then $\phi=\{S_1, S_2\}$ or $\phi=\{X_1(\lambda_i)\}_{i=1}^n$
for different parameters $\{\lambda_1,\cdots,\lambda_n\}$.
\end{lemma}

We also need to compute the $\Ext^1_{A}$-groups.

\begin{lemma}
\label{xxlem5.10} 
Retain the hypotheses of Lemma \ref{xxlem5.7}.
Let $\lambda\neq \lambda'$ be
in $\Bbbk \cup\{\infty\}$.
\begin{enumerate}
\item[(1)]
$\Ext^1_{A}(X_1(\lambda), X_1(\lambda))=\Hom_{A}(X_1(\lambda), X_1(\lambda))=\Bbbk$.
\item[(2)]
$\Ext^1_{A}(X_1(\lambda), X_1(\lambda'))=\Hom_{A}(X_1(\lambda), X_1(\lambda'))=0$.
\item[(3)]
$\begin{pmatrix}
\Ext^1_{A}(S_1, S_1)&\Ext^1_{A}(S_1, S_2)\\
\Ext^1_{A}(S_2, S_1)&\Ext^1_{A}(S_2, S_2)
\end{pmatrix}=\begin{pmatrix}
0&\Bbbk^{\oplus 2}\\
\Bbbk&0
\end{pmatrix}.$
\item[(4)]
$\Ext^1_{A}(P_2, P_2)=0$.
\item[(5)]
$\dim \Ext^1_{A}(S_{2,n}, S_{2,n})\leq 1$ for all $n$.
\item[(6)]
$\dim \Ext^1_{A}(S_{1,n}, S_{1,n})\leq 1$ for all $n$.
\end{enumerate}
\end{lemma}

\begin{remark}
\label{xxrem5.11}
In fact, one can show the following stronger version of
Lemma \ref{xxlem5.10}(5,6).
\begin{enumerate}
\item[(5')]
$\Ext^1_{A}(S_{2,n}, S_{2,n})=0$ for all $n$.
\item[(6')]
$\Ext^1_{A}(S_{1,n}, S_{1,n})=0$ for all $n$.
\end{enumerate}
\end{remark}

\begin{proof}[Proof of Lemma \ref{xxlem5.10}]
(1,2) Consider a minimal projective resolution of $X_1(\lambda)$
$$P_1\to P_2 \xrightarrow{f_{\lambda}} P_1\to X_1(\lambda)\to 0$$
where $f_{\lambda}$ sends $e_2\in P_2$ to $ \gamma-\lambda \beta\in P_1$.
More precisely, we have
$$\begin{pmatrix} \Bbbk e_1\\
\Bbbk \beta+\Bbbk \gamma\end{pmatrix}
\xrightarrow{e_1\to \alpha}
\begin{pmatrix} \Bbbk \alpha\\
\Bbbk e_2\end{pmatrix}
\xrightarrow{e_2\to \gamma-\lambda \beta}
\begin{pmatrix} \Bbbk e_1\\
\Bbbk \beta+\Bbbk \gamma\end{pmatrix}
\longrightarrow
P_1/(\Bbbk (\gamma-\lambda \beta))
\longrightarrow
0.$$
Applying $\Hom_{A}(-,X_1(\lambda'))$ to the truncated projective resolution
of the above, we obtain the following complex
$$ \Bbbk \xleftarrow{ \quad 0\quad } \Bbbk \xleftarrow{ \quad g\quad} \Bbbk
\longleftarrow 0.$$
If $g$ is zero, this is case (1). If $g\neq 0$, this is case (2).

(3) The proof is similar to the above by considering minimal projective 
resolutions of $S_1$ and $S_2$.

(4) This is clear since $P_2$ is a projective module. 

(5,6) Let $S$ be either $S_{2,n}$ or $S_{1,n}$. By Example \ref{xxex5.1},
$\fpdim (\Mod_{f.d}-\Bbbk K_2)=1$. This implies that 
$$\dim \Ext^1_{\Bbbk K_2}(S,S)\leq 1$$
where $S$ is considered as an indecomposable $K_2$-module.

Let us make a comment before we continue the proof. Following 
a more careful analysis, one can actually show that 
$$\Ext^1_{\Bbbk K_2}(S,S)=0.$$
Using this fact, the rest of the proof would show the 
assertions (5',6') in Remark \ref{xxrem5.11}.

Now we continue the proof.
There is a projective cover $P_1^b \xrightarrow{f} S$
so that $\ker f$ is a direct sum of finitely many copies of $S_2$. Since 
$P_2$ is the projective cover of $S_2$,
we have a minimal projective resolution
$$\longrightarrow P_2^{a} \longrightarrow P_1^{b} 
\longrightarrow S\longrightarrow 0$$
for some $a,b$. In the category $\Mod_{f.d}-\Bbbk K_2$, 
we have a minimal projective resolution of $S$
$$0\longrightarrow S_2^{a} \longrightarrow P_1^{b} 
\longrightarrow S\longrightarrow 0$$
where $S_2$ is a projective $\Bbbk K_2$-module. 
Hence we have a morphism of complexes
$$\begin{CD}
 @>>> P_2^{a} @>>> P_1^{b} @>>> S @>>> 0\\
@. @VVV @ VV =V @VV=V\\
0@>>> S_2^{a} @>>> P_1^{b} @>>> S @>>> 0
\end{CD}$$
Applying $\Hom_{A}(-,S)$ to above, we obtain that
{\small $$\begin{CD}
\cdots  @<<< \Hom_{A}(P_2^{a}, S) @<h<< \Hom_{A}(P_1^{b}, S) 
@<<< \Hom_{A}(S,S) @<<< 0\\
@. @A g AA @ AA =A @AA=A\\
\cdots @<<< \Hom_{A}(S_2^{a}, S) @<f<< \Hom_{A}(P_1^{b}, S) 
@<<< \Hom_{A}(S,S) @<<< 0\\
\end{CD}$$}
Note that $g$ is an isomorphism. Since $\dim \Ext^1_{\Bbbk K_2}(S,S)\leq 1$,
the cokernel of $f$ has dimension at most 1.
Since $g$ is an isomorphism, the cokernel of 
$h$ has dimension at most 1. This implies that
$\Ext^1_{A}(S,S)$ has dimension at most 1.
\end{proof}

\begin{proposition}
\label{xxpro5.12}
Let ${\mathfrak A}$ be the category $\Mod_{f.d.}-A$
where $A$ is in Example \ref{xxex5.5}. 
\begin{enumerate}
\item[(1)]
$\fpdim^n {\mathfrak A}=\begin{cases} \sqrt{2} & n=2,\\ 1 & n\neq 2.\end{cases}$ 
As a consequence, $\fpdim {\mathfrak A}=\sqrt{2}$.
\item[(2)]
$SI({\mathfrak A})=2$.
\item[(3)]
$\fpdim {\mathcal A}\geq \sqrt{2}$.
\end{enumerate}
\end{proposition}

\begin{proof} (1) This is a consequence of Lemmas \ref{xxlem5.9} and
\ref{xxlem5.10}. Parts (2,3) follow from part (1).
\end{proof}

\begin{remark}
\label{xxrem5.13}
Let $A$ be the algebra given in Example \ref{xxex5.5}.
We list some facts, comments and questions.
\begin{enumerate}
\item[(1)]
The algebra $A$ is non-connected ${\mathbb N}$-graded Koszul.
\item[(2)]
The minimal projective resolutions of $S_1$ and $S_2$ are
$$ \cdots \longrightarrow P_1^{\oplus 4}
\longrightarrow P_2^{\oplus 4}
\longrightarrow P_1^{\oplus 2}
\longrightarrow P_2^{\oplus 2}
\longrightarrow P_1
\longrightarrow S_1
\longrightarrow 0,
$$
and
$$ \cdots \longrightarrow P_2^{\oplus 4}
\longrightarrow P_1^{\oplus 2}
\longrightarrow P_2^{\oplus 2}
\longrightarrow P_1
\longrightarrow P_2
\longrightarrow S_2
\longrightarrow 0.
$$
\item[(3)] 
For $i\geq 0$, we have
$$\begin{aligned}
\Ext^i_{A}(S_1,S_1)&=\begin{cases} \Bbbk^{\oplus 2^{i/2}} & 
i \; {\text{is}} \; {\text{even}}\\
0 & \; {\text{is}} \; {\text{odd}}
\end{cases}\\
\Ext^i_{A}(S_1,S_2)&=\begin{cases} 0  & 
i \; {\text{is}} \; {\text{even}}\\
\Bbbk^{\oplus 2^{(i+1)/2}} & \; {\text{is}} \; {\text{odd}}
\end{cases}\\
\Ext^i_{A}(S_2,S_2)&=\begin{cases} \Bbbk^{\oplus 2^{i/2}} & 
i \; {\text{is}} \; {\text{even}}\\
0 & \; {\text{is}} \; {\text{odd}}
\end{cases}\\
\Ext^i_{A}(S_2,S_1)&=\begin{cases} 0 & 
i \; {\text{is}} \; {\text{even}}\\
\Bbbk^{\oplus 2^{(i-1)/2}} & \; {\text{is}} \; {\text{odd.}}
\end{cases}
\end{aligned}
$$
\item[(4)]
One can check that every algebra of dimension 4 or less has 
either infinite or integral $\fpdim$. Hence,
$A$ is an algebra of smallest $\Bbbk$-dimension that 
has finite non-integral (or irrational) $\fpdim$. It is unknown 
if there is a finite dimensional algebra $A$ such that 
$\fpdim (\Mod_{f.d.}-A)$ is transcendental.
\item[(5)]
Several authors have studied the connection between tame-wildness and 
complexity \cite{BS, ES, Fa, FW, Kul, Ri}. The algebra $A$ is probably the first 
explicit example of a tame algebra with infinite complexity. 
\item[(6)]
It follows from part (3) that
the fp-curvature of ${\mathcal A}:=D^b(\Mod_{f.d.}-A)$ is 
$\sqrt{2}$ (some details are omitted). 
As a consequence, $\fpg({\mathcal A})=\infty$.
By Theorem \ref{xxthm8.3}, the complexity of $A$ is $\infty$.
We don't know what $\fpdim {\mathcal A}$ is.
\end{enumerate}
\end{remark}

\section{$\sigma$-decompositions}
\label{xxsec6}

We fix a category ${\mathcal C}$ and an endofunctor 
$\sigma$. For a set of bricks $B$ in ${\mathcal C}$ 
(or a set of atomic objects when ${\mathcal C}$ is 
triangulated), we define 
$$\fpdim^n\mid_{B} (\sigma) =\sup\{ \rho(A(\phi,\sigma)) \;
\mid \; \phi:=\{X_1,\cdots,X_n\}\in \Phi_{n,b}, 
\;\; {\rm{and}} \;\; X_i\in B\;\; \forall i\}.$$

Let $\Lambda:=\{\lambda\}$ be a totally ordered set. We say 
a set of bricks $B$ in ${\mathcal C}$ has a 
{\it $\sigma$-decomposition} $\{B^{\lambda}\}_{\lambda\in \Lambda}$
(based on $\Lambda$) if the following holds.
\begin{enumerate}
\item[(1)]
$B$ is a disjoint union $\bigcup_{\lambda\in \Lambda} B^{\lambda}$.
\item[(2)]
If $X\in B^{\lambda}$ and $Y\in B^{\delta}$
with $\lambda<\delta$, $\Hom_{\mathcal C}(X,\sigma(Y))=0$.
\end{enumerate}

The following lemma is easy.

\begin{lemma}
\label{xxlem6.1} Let $n$ be a positive integer.
Suppose that $B$ has a 
{\it $\sigma$-decomposition} $\{B^{\lambda}\}_{\lambda\in \Lambda}$.
Then 
$$\fpdim^n|_{B}(\sigma)\leq \sup_{\lambda \in \Lambda, m\leq n} 
\{\fpdim^m|_{B^{\lambda}}(\sigma)\}.$$
\end{lemma}

\begin{proof} Let $\phi$ be a brick set that is used in the 
computation of $\fpdim^n|_{B}(\sigma)$. Write 
\begin{equation}
\label{E6.1.1}\tag{E6.1.1}
\phi=\phi_{\lambda_{1}} \cup \cdots \cup \phi_{\lambda_{s}}
\end{equation}
where $\lambda_i$ is strictly increasing and 
$\phi_{\lambda_i}=\phi\cap B^{\lambda_i}$. 
For any objects $X\in \phi^{\lambda_i}$ and $Y\in \phi^{\lambda_j}$,
where $\lambda_i<\lambda_j$, by definition, $\Hom_{\mathcal C}(X,\sigma(Y))=0$.
Listing the objects in $\phi$ in the order that suggested by \eqref{E6.1.1},
then the adjacency matrix of $(\phi, \sigma)$ is of the form
\begin{equation}
\notag
A(\phi, \sigma)=\begin{pmatrix} A_{11} & 0      &0 & \cdots &0\\
                                 \ast   & A_{22} &0 &\cdots & 0\\
			       \ast &\ast &A_{33} &\cdots &0\\
			     \ldots &\ldots &\ldots &\ldots &0\\
			   \ast& \ast&\ast &\cdots & A_{ss}
		       \end{pmatrix}
\end{equation}
where each $A_{ii}$ is the adjacency matrix $A(\phi_{\lambda_i},\sigma)$. 
By definition,
$$\rho(A_{ii})\leq \fpdim^{s_i}\mid_{B^{\lambda_{i}}} (\sigma)$$
where $s_i$ is the size of $A_{ii}$, which is no more than 
$n$. Therefore
$$\rho(A(\phi,\sigma))= \max_{i} \{\rho(A_{ii})\} 
\leq \sup_{\lambda \in \Lambda, m\leq n} 
\{\fpdim^m|_{B^{\lambda}}(\sigma)\}.$$
The assertion follows. 
\end{proof}

We give some examples of $\sigma$-decompositions.

\begin{example}
\label{xxex6.2}
Let ${\mathfrak A}$ be an abelian category and 
${\mathcal A}$ be the derived category $D^b({\mathfrak A})$.
Let $[n]$ be the $n$-fold suspension $\Sigma^n$. 
\begin{enumerate}
\item[(1)]
Suppose that $\alpha$ is an endofunctor of ${\mathfrak A}$ 
and $\overline{\alpha}$ is the induced endofunctor 
of ${\mathcal A}$. For each $n\in {\mathbb Z}$, 
let $B^n:=\{ M[-n] \mid M\;\; {\rm{is}} \; {\rm{a}} \; 
{\rm{brick}} \; {\rm{in}} \; {\mathfrak A}\}$ and 
$B:=\bigcup_{n\in {\mathbb Z}} B^n$. If $M_i[-n_i] \in B^{n_i}$, 
for $i=1,2$, such that $n_1<n_2$, then 
$$\Hom_{\mathcal A}(M_1[-n_1], \overline{\alpha}(M_2[-n_2]))=
\Ext_{\mathfrak A}^{n_1-n_2}(M_1,\alpha(M_2))=0.$$
Then $B$ has a $\overline{\alpha}$-decomposition 
$\{B^n\}_{n\in {\mathbb Z}}$ based on ${\mathbb Z}$. 
\item[(2)]
Suppose $g:=\gldim {\mathfrak A}<\infty$. Let $\sigma$ be
the functor $\Sigma^g \circ \overline{\alpha}$.
For each $n\in {\mathbb Z}$, 
let $B^n:=\{ M[n] \mid M\;\; {\rm{is}} \; {\rm{a}} \; 
{\rm{brick}} \; {\rm{in}} \; {\mathfrak A}\}$ and 
$B:=\bigcup_{n\in {\mathbb Z}} B^n$. If $M_i[n_i] \in B^{n_i}$, 
for $i=1,2$, such that $n_1<n_2$, then 
$$\Hom_{\mathcal A}(M_1[n_1], \sigma(M_2[n_2]))=
\Ext_{\mathfrak A}^{n_2-n_1+g}(M_1,\alpha(M_2))=0.$$
Then $B$ has a $\sigma$-decomposition 
$\{B^n\}_{n\in {\mathbb Z}}$ based on ${\mathbb Z}$. 
\end{enumerate}
\end{example}

\begin{example}
\label{xxex6.3}
Let $C$ be a smooth projective curve and let ${\mathfrak A}$
be the the category of coherent sheaves over $C$. Every
coherent sheaf over $C$ is a direct sum of a torsion subsheaf 
and a locally free subsheaf. Define
$$\begin{aligned}
B^0&=\{T\;\; {\rm{is}}\;\; {\rm{a}}\;\; {\rm{torsion}}\;\;
{\rm{brick}}\;\; {\rm{object}}\;\; {\rm{in}}\;\; {\mathfrak A}\},\\
B^{-1}&=\{F\;\; {\rm{is}}\;\; {\rm{a}}\;\; {\rm{locally}}\;{\rm{free}}\;\;
{\rm{brick}}\;\; {\rm{object}}\;\; {\rm{in}}\;\; {\mathfrak A}\}
\end{aligned}
$$
and
$$B=B^{-1}\cup B^{0}.$$
Let $\sigma$ be the functor $E^1:=\Ext^1_{\mathfrak A}(-,-)$.
If $F\in B^{-1}$ and $T\in B^0$, then 
$$\Ext^1_{\mathfrak A}(F,T)=0.$$
Hence, $B$ has an $E^1$-decomposition based on the totally ordered set 
$\Lambda:=\{-1,0\}$.
\end{example}

The next example is given in \cite{BB}.

\begin{example}
\label{xxex6.4}
Let $C$ be an elliptic curve. Let ${\mathfrak A}$
be the the category of coherent sheaves over $C$ and 
${\mathcal A}$ be the derived category $D^b({\mathfrak A})$.

First we consider coherent sheaves.
Let $\Lambda$ be the totally ordered set ${\mathbb Q}\cup\{+\infty\}$.
The slope of a coherent sheaf $X\neq 0$ \cite[Definition 4.6]{BB}
is defined to be
$$\mu(X):=\frac{\chi(X)}{\rk (X)}\in \Lambda$$
where $\chi(X)$ is the Euler characteristic of $X$ and $\rk(X)$ 
is the rank of $X$. If $X$ and $Y$ are bricks such that 
$\mu(X)<\mu(Y)$, by \cite[Corollary 4.11]{BB}, $X$ and
$Y$ are semistable, and thus by \cite[Proposition 4.9(1)]{BB},
$\Hom_{\mathfrak A}(Y,X)=0$. By Serre duality (namely, Calabi-Yau property),
\begin{equation}
\label{E6.4.1}\tag{E6.4.1}
\Hom_{\mathcal A}(X, Y[1])=
\Ext^1_{\mathfrak A}(X,Y)=\Hom_{\mathfrak A}(Y,X)^*=0.
\end{equation}
Write $B=\Phi_{1,b}({\mathfrak A})$ and $B^{\lambda}$ be the
set of (semistable) bricks with slope $\lambda$. Then 
$B=\bigcup_{\lambda\in \Lambda} B^{\lambda}$. By \eqref{E6.4.1},
$\Ext^1_{\mathfrak A}(X,Y)=0$ when $X\in B^{\lambda}$ and 
$Y\in B^{\nu}$ with $\lambda<\nu$. Hence $B$ has 
an $E^1$-decomposition. By Lemma \ref{xxlem6.1}, for every
$n\geq 1$,
$$\fpdim^n(E^1)=\fpdim^n|_{B}(E^1)\leq \sup_{\lambda \in \Lambda, m\leq n} 
\{\fpdim^m|_{B^{\lambda}}(E^1)\}.$$

Next we compute $\fpdim^n|_{B^{\lambda}}(E^1)$. 
Let $SS^{\lambda}$ be the full subcategory of ${\mathfrak A}$
consisting of semistable coherent sheaves of slope $\lambda$. 
By \cite[Summary]{BB}, $SS^{\lambda}$ is an abelian
category that is equivalent to $SS^{\infty}$. Therefore 
one only needs to compute $\fpdim^n|_{B^{\infty}}(E^1)$ in the 
category $SS^{\infty}$. Note that $SS^{\infty}$ is the abelian
category of torsion sheaves and every brick object in 
$SS^{\infty}$ is of the form ${\mathcal O}_{p}$ for some 
$p\in C$. In this case, $A(\phi,E^1)$
is the identity matrix. Consequently, $\rho(A(\phi,E^1))=1$.
This shows that $\fpdim^n|_{B^{\lambda}}(E^1)
=\fpdim^n|_{B^{\infty}}(E^1)=1$ for all $n\geq 1$. It is clear
that $\fpdim^n(E^1)\geq \fpdim^n|_{B^{\infty}}(E^1)=1$. 
Combining with Lemma \ref{xxlem6.1}, we obtain that
$\fpdim^n(E^1)=1$ for all $n$. (The above approach works 
for functors other than $E^1$.)

Finally we consider the fp-dimension for the derived category
${\mathcal A}$. It follows from Theorem \ref{xxthm3.5}(3) that
$$\fpdim^n(\Sigma)=\fpdim^n(E^1)=1$$
for all $n\geq 1$. By definition,
$$\fpdim({\mathcal A})=\fpdim({\mathfrak A})=1.$$
\end{example}

As we explained before $\fpdim$ is an indicator of the representation
types of categories. 

Drozd-Greuel studied a tame-wild dichotomy for vector bundles 
on projective curves \cite{DrG} and introduced the notion of 
VB-finite, VB-tame and VB-wild similar to the corresponding notion 
in the representation theory of finite dimensional algebras. 
In \cite{DrG} Drozd-Greuel showed the following:

Let $C$ be a connected smooth projective curve. Then 
\begin{enumerate}
\item[(a)]
$C$ is VB-finite if and only if $C$ is ${\mathbb P}^1$.
\item[(b)]
$C$ is VB-tame if and only if $C$ is elliptic (that is, of genus 1).
\item[(c)]
$C$ is VB-wild if and only if $C$ has genus $g\geq 2$.
\end{enumerate}

We now prove a fp-version of Drozd-Greuel's result 
\cite[Theorem 1.6]{DrG}. We thank Max Lieblich for 
providing ideas in the proof of Proposition \ref{xxpro6.5}(3)
next.

\begin{proposition}
\label{xxpro6.5} 
Suppose $\Bbbk={\mathbb C}$.
Let ${\mathbb X}$ be a connected smooth projective curve
and let $g$ be the genus of ${\mathbb X}$.
\begin{enumerate}
\item[(1)]
If $g=0$ or ${\mathbb X}={\mathbb P}^1$, then $\fpdim\; D^b(coh(\mathbb X))=1.$
\item[(2)]
If $g=1$ or ${\mathbb X}$ is an elliptic curve, then $\fpdim\; D^b(coh(\mathbb X))=1.$
\item[(3)]
If $g\geq 2$, then $\fpdim \; D^b(coh(\mathbb X))=\infty$.
\end{enumerate}
\end{proposition}

\begin{proof} (1)
The assertion follows from \eqref{E5.1.4}.

(2) 
The assertion follows from Example \ref{xxex6.4}.

(3) By Theorem \ref{xxthm3.5}(4), $\fpdim(D^b(coh({\mathbb X})))
=\fpdim(coh({\mathbb X}))$. Hence it suffices to show that 
$\fpdim(coh({\mathbb X}))=\infty$.

For each $n$, let $\{x_i\}_{i=1}^n$ be a set of $n$ distinct points 
on ${\mathbb X}$. By \cite[Lemma 1.7]{DrG}, we might further assume that
$2x_i\not\sim x_j+x_k$ for all $i\neq j$, as divisors on ${\mathbb X}$.
Write $\mathcal{E}_i:={\mathcal O}(x_i)$ for all $i$. 
By \cite[p.11]{DrG}, $\Hom_{{\mathcal O}_{\mathbb X}}
(\mathcal{E}_i, \mathcal{E}_j)=0$ for all
$i\neq j$, which is also a consequence of a more general result 
\cite[Proposition 1.2.7]{HL}. It is clear that 
$\Hom_{{\mathcal O}_{\mathbb X}}(\mathcal{E}_i, 
\mathcal{E}_i)=\Bbbk$ for all $i$. Let $\phi_n$ be the set 
$\{\mathcal{E}_1, \ldots, \mathcal{E}_n\}$. Then it is 
a brick set of non-isomorphic vector bundles on 
${\mathbb X}$ (which are stable with 
${\text{rank}}(\mathcal{E}_i)=\deg(\mathcal{E}_i)=1$ 
for all $i$).

Define the sheaf $\mathcal{H}_{ij}=
\mathcal{H}om(\mathcal{E}_i, \mathcal{E}_j)$ for all $i,j$. 
Then $\deg(\mathcal{H}_{ij})=0.$ By the Riemann-Roch Theorem, we have
\begin{align*}
0 &= \deg(\mathcal{H}_{ij}) \\
&= \chi(\mathcal{H}_{ij}) - {\text{rank}}(\mathcal{H}_{ij})
\chi(\mathcal{O}_{\mathbb X}) \\
&= \dim \Hom_{\mathcal{O}_{\mathbb X}}(\mathcal{E}_i, \mathcal{E}_j) 
- \dim \Ext^1_{\mathcal{O}_{\mathbb X}}(\mathcal{E}_i, \mathcal{E}_j) 
- (1-g) \\
&= \delta_{ij} - \dim \Ext^1_{\mathcal{O}_{\mathbb X}}(\mathcal{E}_i, \mathcal{E}_j) 
+ (g-1),
\end{align*}
which implies that $\dim \Ext^1_{\mathcal{O}_{\mathbb X}}(\mathcal{E}_i, \mathcal{E}_j) 
=g-1+\delta_{ij}$. This formula was also  given in 
\cite[p.11 before Lemma 1.7]{DrG} when $i\neq j$.

Define the matrix $A_n$ with entries $a_{ij}:=\dim 
\Ext^1_{\mathcal{O}_{\mathbb X}}(\mathcal{E}_i, \mathcal{E}_j)
=g-1+\delta_{ij}$, which is the adjacency matrix of $(\phi_n, E^1)$. 
This matrix has entries $g$ along the diagonal and 
entries $g-1$ everywhere else. Therefore the vector $(1, \ldots, 1)$ is an 
eigenvector for this matrix with eigenvalue $n(g-1)+1$. So $\rho(A_n) \geq n(g-1)+1 
\geq n+1$. Since we can define $\phi_n$ for arbitrarily large $n$, we must have 
$\fpdim(coh({\mathbb X}))=\infty.$ 
\end{proof}

\begin{question}
\label{xxque6.6}
Let ${\mathbb X}$ be a smooth irreducible projective curve of genus $g\geq 2$.
Is $\fpdim^n({\mathbb X})$ finite for each $n$? If yes, 
do these invariants recover $g$?
\end{question}

\begin{proposition}
\label{xxpro6.7}
Suppose $\Bbbk={\mathbb C}$.
Let ${\mathbb Y}$ be a smooth projective scheme of dimension at least 2.
Then 
$$\fpdim^1(coh({\mathbb Y}))=\fpdim(coh({\mathbb Y}))
=\fpdim^1(D^b(coh({\mathbb Y})))=
\fpdim(D^b(coh({\mathbb Y})))=\infty.$$
\end{proposition}

\begin{proof} It is clear that $\fpdim^1(coh({\mathbb Y}))$
is smallest among these four invariants. It suffices to show
that $\fpdim^1(coh({\mathbb Y}))=\infty$.

It is well-known that ${\mathbb Y}$ contains an irreducible
projective curve ${\mathbb X}$ of arbitrarily large 
(either geometric or arithmetic) genus, see, for example, 
\cite[Theorem 0.1]{CFZ} or \cite[Theorems 1 and 2]{Ch}.
Let ${\mathcal O}_{\mathbb X}$ be the coherent
sheaf corresponding to the curve ${\mathbb X}$ and
let $g$ be the arithmetic genus of ${\mathbb X}$.
In the abelian category
$coh({\mathbb X})$, we have
$$\dim \Ext^1_{\mathcal{O}_{\mathbb X}}
(\mathcal{O}_{\mathbb X}, \mathcal{O}_{\mathbb X})
=\dim H^1({\mathbb X}, \mathcal{O}_{\mathbb X})=g.$$
Since $coh({\mathbb X})$ is a full subcategory of 
$coh({\mathbb Y})$, we have
$$\dim \Ext^1_{\mathcal{O}_{\mathbb Y}}
(\mathcal{O}_{\mathbb X}, \mathcal{O}_{\mathbb X})
\geq 
\dim \Ext^1_{\mathcal{O}_{\mathbb X}}
(\mathcal{O}_{\mathbb X}, \mathcal{O}_{\mathbb X})
=g.$$
By taking $\phi=\{\mathcal{O}_{\mathbb X}\}$, one sees 
that $\fpdim^1(coh({\mathbb Y}))\geq \fpdim^1(coh({\mathbb X}))
\geq g$ for all such ${\mathbb X}$. Since $g$ can be
arbitrarily large, the assertion follows.
\end{proof}

\section{Representation types}
\label{xxsec7}
\subsection{Representation types}
\label{xxsec7.1}

We first recall some known definitions and results.

\begin{definition}
\label{xxdef7.1}
Let $A$ be a finite dimensional algebra.
\begin{enumerate}
\item[(1)]
We say $A$ is of {\it finite representation type} if there are 
only finitely many isomorphism classes of finite dimensional 
indecomposable left $A$-modules. 
\item[(2)]
We say $A$ is {\it tame} or {\it of tame representation type} 
if it is not of finite representation type, and 
for every $n\in {\mathbb N}$, all but finitely many isomorphism 
classes of $n$-dimensional indecomposables occur in a finite 
number of one-parameter families.
\item[(3)]
We say $A$ is {\it wild} or {\it of wild representation type} 
if, for every finite dimensional $\Bbbk$-algebra 
$B$, the representation theory of $B$ can be embedded into 
that of $A$.
\end{enumerate}
\end{definition}

The following is the famous trichotomy result due to Drozd \cite{Dr1}.

\begin{theorem}[Drozd's trichotomy theorem] 
\label{xxthm7.2}
Every finite dimensional algebra is either of 
finite, tame, or wild representation type.
\end{theorem}

\begin{remark}
\label{xxrem7.3}
\begin{enumerate}
\item[(1)]
An equivalent and more precise definition of a wild algebra is the following. 
An algebra $A$ is called {\it wild} if there is a faithful exact
embedding of abelian categories
\begin{equation}
\label{E7.3.1}\tag{E7.3.1}
Emb: \quad \Mod_{f.d.}-\Bbbk \langle x,y\rangle \to \Mod_{f.d.}-A
\end{equation}
that respects indecomposability and isomorphism classes.
\item[(2)]
A stronger notion of wildness is the following. An algebra $A$ 
is called {\it strictly wild}, also called {\it fully wild}, 
if $Emb$ in \eqref{E7.3.1} is a fully faithful embedding.
\item[(3)]
It is clear that strictly wild is wild. The converse is not true.
\end{enumerate}
\end{remark}

We collect some celebrated results in terms of representation types
of path algebras \cite{Ga1} and \cite{Na, DoF}. 

\begin{theorem}
\label{xxthm7.4}
Let $Q$ be a finite connected quiver.
\begin{enumerate}
\item[(1)] \cite{Ga1}
The path algebra $\Bbbk Q$ is of finite representation
type if and only if the underlying graph of $Q$ is a 
Dynkin diagram of type $ADE$.
\item[(2)] \cite{Na, DoF}
The path algebra $\Bbbk Q$ is of tame
representation type if and only if the underlying graph of 
$Q$ is an extended Dynkin
diagram of type $\widetilde{A}\widetilde{D}\widetilde{E}$.
\end{enumerate}
\end{theorem}

Our main goal in this section is to prove Theorem 
\ref{xxthm0.3}. 
We thank Klaus Bongartz for suggesting the following lemma.

\begin{lemma}
\label{xxlem7.5} \cite{B1}
Let $A$ be a finite dimensional algebra that is strictly 
wild. Then, for each integer $a>0$, there is a finite dimensional 
brick left $A$-module $N$ such that $\dim \Ext^1_A(N,N)\geq a$. 
\end{lemma}

\begin{proof} Let $V$ be the vector space $\bigoplus_{i=1}^a \Bbbk x_i$
and let $B$ be the finite dimensional algebra 
$\Bbbk\langle V \rangle/(V^{\otimes 2})$. By \cite[Theorem 2(i)]{B2}, 
there is a fully faithful exact embedding
$$\Mod_{f.d.}-B \longrightarrow \Mod_{f.d.}-\Bbbk \langle x,y\rangle.
$$ 
Since $A$ is strictly wild, there is a fully faithful exact
embedding  
$$
\Mod_{f.d.}-\Bbbk \langle x,y\rangle 
\longrightarrow \Mod_{f.d.}-A.
$$
Hence we have a fully faithful exact embedding  
\begin{equation}
\label{E7.5.1}\tag{E7.5.1}
F: \quad \Mod_{f.d.}-B \longrightarrow \Mod_{f.d.}-A.
\end{equation} 
Let $S$ be the trivial $B$-module $B/B_{\geq 1}$. It follows from an easy 
calculation that 
$\dim \Ext^1_B(S,S)=\dim (V)^*=a$. Since $F$ is fully faithful exact, $F$ induces an 
injection 
$$F: \Ext^1_B(S,S)\longrightarrow \Ext^1_{A}(F(S), F(S)).$$
Thus $\dim \Ext^1_A(F(S),F(S))\geq a$. Since $S$ is simple,
it is a brick. Hence, $F(S)$ is a brick. The assertion follows by
taking $N=F(S)$.
\end{proof}

\begin{proposition}
\label{xxpro7.6}
\begin{enumerate}
\item[(1)]
Let $A$ be a finite dimensional algebra that is strictly wild, then
$$\fpdim^1(E^1)=\fpdim({\mathfrak A})=\fpdim ({\mathcal A})=\infty.$$
\item[(2)]
If $A:=\Bbbk Q$  is wild, then 
$$\fpdim^1(E^1)=\fpdim({\mathfrak A})=\fpdim ({\mathcal A})=\infty.$$
\end{enumerate}
\end{proposition}

\begin{proof} 
(1) For each integer $a$, by Lemma \ref{xxlem7.5},
there is a brick $N$ in ${\mathfrak A}$ such that 
$\Ext^1_{\mathfrak A}(N,N)\geq a$. Hence $\fpdim^1 (E^1)\geq a$. 
Since $a$ is arbitrary, $\fpdim^1(E^1)=\infty$. Consequently,
$\fpdim({\mathfrak A})=\infty$. By Lemma \ref{xxlem2.9},
$\fpdim({\mathcal A})=\infty$.

(2) It is well-known that a wild path algebra is strictly
wild, see a comment of Gabriel \cite[p.149]{Ga2} or 
\cite[Proposition 7]{Ar}. The assertion follows from part (1).
\end{proof}

The following lemma is based on a well-understood AR-quiver theory 
for acyclic quivers of finite representation type and the  
hammock theory introduced by Brenner \cite{Bre}. We refer to 
\cite{RV} if the reader is interested in a more abstract version 
of the hammock theory. 

For a class of quivers including all ADE quivers, there is a 
convenient (though not essential) way of positioning the vertices 
as in \cite[Example IV.2.6]{AS}. A quiver $Q$ is called 
{\it well-positioned} if the vertices of $Q$ are located so that 
all arrows are strictly from the right to the left of the same 
horizontal distance. 
For example, the following quiver $D_n$ is well-positioned: 
\[
\begin{tikzcd}
 & & & & n-1 \arrow[ld]  \\
1 & \arrow[l] 2 & \arrow[l] \cdots & \arrow[l] n-2 \\
 & & & & n \arrow[lu]  \\
\end{tikzcd}
\]

\begin{lemma}
\label{xxlem7.7}
Let $Q$ be a quiver such that 
\begin{enumerate}
\item[(i)]
the underlying graph of $Q$ is a Dynkin diagram of 
type $A$, or $D$, or $E$, and that
\item[(ii)]
$Q$ is well-positioned. 
\end{enumerate}
Let $A=\Bbbk Q$ and let $M, N$ be two indecomposable 
left $A$-modules in the AR-quiver of $A$. 
Then the following hold.
\begin{enumerate}
\item[(1)]
There is a standard way of defining the order or degree 
for indecomposable left $A$-modules $M$,
denoted by $\deg M$, such that all arrows in the 
the AR-quiver have degree 1, or equivalently, 
all arrows are from the left to the right of the 
same horizontal distance. As in 
\eqref{E5.4.2}, when $Q=A_2$, $\deg P_2=0$, $\deg P_1=1$
and $\deg I_1=2$.
\item[(2)]
If $\Hom_A(M,N)\neq 0$, then $\deg M\leq \deg N$. 
\item[(3)]
The degree of the AR-translation $\tau$ is $-2$. 
\item[(4)]
If $\Ext^1_{A}(M,N)\neq 0$, then $\deg M\geq \deg N+2$.
\item[(5)]
There is no oriented cycle in the $E^1$-quiver of 
${\mathfrak A}:=\Mod_{f.d.}-\Bbbk Q$, 
denoted by $Q^{E^1}_{\mathfrak A}$, defined before 
Lemma {\rm{\ref{xxlem2.10}}}.  
\item[(6)]
$\fpdim({\mathfrak A})=0$. 
\end{enumerate}
\end{lemma}

\begin{proof} (1) This is a well-known fact in AR-quiver 
theory. For each given quiver $Q$ as described in (i,ii), 
one can build the AR-quiver by using the Auslander-Reiten 
translation $\tau$ and {\it the Knitting Algorithm}, see 
\cite[Ch. 3]{Sc1}. Some explicit examples are given in 
\cite[Ch. 6]{Ga3} and \cite[Ch. 3]{Sc1}.

(2) This follows from (1). Note that the precise dimension of
$\Hom_A(M,N)$ can be computed by using hammock theory 
\cite{Bre, RV}. Some examples are given in 
\cite[Ch. 3]{Sc1}.

(3) This follows from the definition of the translation $\tau$ 
in the AR-quiver theory \cite[VII]{ARS}. See also, \cite[Remarks (2),
p. 23]{Cr}.

(4) By Serre duality, $\Ext^1_R(M,N)=\Hom_A(N, \tau M)^\ast$
\cite[Proposition I.2.3]{RVdB} or \cite[Lemma 1, p. 22]{Cr}.
If $\Ext^1_R(M,N)\neq 0$, then, by Serre duality and part (2),
$\deg N\leq \deg \tau M=\deg M-2$. Hence
$\deg M\geq \deg N+2$.

(5) In this case, every indecomposable module is a brick.
Hence the $E^1$-quiver $Q^{E^1}_{\mathfrak A}$ has the 
same vertices as the AR-quiver. By part (4), if there is an 
arrow from $M$ to $N$ in the quiver  $Q^{E^1}_{\mathfrak A}$,
then $\deg M\geq \deg N+2$. This 
means that all arrows in $Q_{\mathfrak A}^{E^1}$ 
are from the right to the left. 
Therefore there is no oriented cycle in $Q^{E^1}_{\mathfrak A}$.

(6) This follows from part (5), Theorem \ref{xxthm1.8}(1)
and Lemma \ref{xxlem2.10}.
\end{proof}

\begin{theorem}
\label{xxthm7.8} 
Let $Q$ be a finite quiver whose underlying graph is 
a Dynkin diagram of type $ADE$ and let $A=\Bbbk Q$.
Then $\fpdim ({\mathfrak A})=\fpdim({\mathcal A})=\fpgldim 
({\mathcal A})=0$.
\end{theorem}

\begin{proof} Since the path algebra $A$ is hereditary, 
${\mathfrak A}$ is a-hereditary of global dimension 1. 
By Theorem \ref{xxthm3.5}(3), $\fpdim({\mathfrak A})=
\fpdim({\mathcal A})$. If $Q_1$ and $Q_2$ are two 
quivers whose underlying graphs are the same, 
then, by Bernstein-Gelfand-Ponomarev (BGP) reflection
functors \cite{BGP},  $D^b(\Mod_{f.d.}-\Bbbk Q_1)$ and 
$D^b(\Mod_{f.d.}-\Bbbk Q_2)$ are triangulated 
equivalent. Hence we only need prove the statement 
for one representative. Now we can assume that 
$Q$ satisfies the hypotheses (i,ii) of Lemma \ref{xxlem7.7}.
By Lemma \ref{xxlem7.7}(6), $\fpdim({\mathfrak A})=0$.
Therefore $\fpdim({\mathcal A})=0$, or equivalently,
$\fpdim(\Sigma)=0$. By Theorem 
\ref{xxthm3.5}(1), $\fpdim(\Sigma^{i})=0$ for all $i\neq 0,1$.
Therefore $\fpgldim ({\mathcal A})=0$.
\end{proof}

\subsection{Weighted projective lines}
\label{xxsec7.2}
To prove Theorem \ref{xxthm0.3}, it remains to show part (2)
of the theorem. Our proof uses a result of \cite{CG} about 
weighted projective lines, which we now review. Details can 
be found in \cite[Section 1]{GL}.

For $t\geq 1$, let ${\bf p}:=(p_0,p_1,\cdots,p_t)$ be a 
$(t+1)$-tuple of positive integers, called the 
{\it weight sequence}. Let
${\bf D}:=(\lambda_0, \lambda_1,\cdots, \lambda_t)$ be a 
sequence of distinct points of the projective line 
${\mathbb P}^1$ over $\Bbbk$. We normalize ${\bf D}$ 
so that $\lambda_0=\infty$, $\lambda_1=0$ and 
$\lambda_2=1$ (if $t\geq 2$). Let
$$S:=\Bbbk[X_0,X_1,\cdots,X_t]/(X_i^{p_i}-X_1^{p_1}+\lambda_i X_0^{p_0},
i=2,\cdots,t).$$
The image of $X_i$ in $S$ is denoted by $x_i$ for all $i$.
Let ${\mathbb L}$ be the abelian group of rank 1 
generated by $\overrightarrow{x_i}$ for $i=0,1,\cdots,t$
and subject to the relations
$$p_0 \overrightarrow{x_0}= \cdots =p_i \overrightarrow{x_i}=\cdots
=p_t \overrightarrow{x_t}=: \overrightarrow{c}.$$
The algebra $S$ is ${\mathbb L}$-graded by setting $\deg x_i=
\overrightarrow{x_i}$. The corresponding 
{\it weighted projective line}, denoted by ${\mathbb X}({\bf p},{\bf D})$
or simply ${\mathbb X}$,
is a noncommutative space whose category of coherent sheaves is given by
the quotient category 
$$coh({\mathbb X}):=\frac{\gr^{\mathbb L}-S}{\gr_{f.d.}^{\mathbb L}-S}$$
where $\gr^{\mathbb L}-S$ is the category of noetherian 
${\mathbb L}$-graded left $S$-modules and $\gr_{f.d.}^{\mathbb L}-S$
is the full subcategory of $\gr^{\mathbb L}-S$ consisting of finite
dimensional modules. 

The weighted projective lines are classified into the following
three classes:
$${\mathbb X} \;\; {\rm{is}}\;\; 
\begin{cases} domestic \;\; & {\rm{if}} \;\; {\bf p} 
\;\; {\rm{is}}\; (p, q), (2,2,n), (2,3,3), (2,3,4), (2,3,5);\\
tubular \;\; & {\rm{if}} \;\; {\bf p} 
\;\; {\rm{is}}\; (2,3,6), (3,3,3), (2,4,4), (2,2,2,2);\\
wild \;\; & {\rm{otherwise}}.
\end{cases}
$$

Let ${\mathbb X}$ be a weighted projective curve. Let
$Vect({\mathbb X})$ be the full subcategory of $coh({\mathbb X})$
consisting of all vector bundles.
Similar to the elliptic curve case [Example \ref{xxex6.4}], one can define 
the concepts of {\it degree}, {\it rank} and {\it slope} of a vector 
bundle on a weighted projective curve ${\mathbb X}$, see \cite[Section 2]{LM}
for details. For each $\mu\in {\mathbb Q}\cup \{\infty\}$,
let $Vect_{\mu}({\mathbb X})$ be the full subcategory of $Vect({\mathbb X})$
consisting of all vector bundles of slope $\mu$.

\begin{lemma}
\label{xxlem7.9}
Let ${\mathbb X}={\mathbb X}({\bf p}, {\bf D})$ be a weighted 
projective line.
\begin{enumerate}
\item[(1)]
$coh({\mathbb X})$ is noetherian and hereditary.
\item[(2)]
$$D^b(coh({\mathbb X})) \cong 
\begin{cases} 
D^b(\Mod_{f.d.}-\Bbbk \widetilde{A}_{p, q}) & {\rm{if}}\;\; {\bf p}=(p,q),\\
D^b(\Mod_{f.d.}-\Bbbk \widetilde{D}_n) & {\rm{if}}\;\; {\bf p}=(2,2,n),\\
D^b(\Mod_{f.d.}-\Bbbk \widetilde{E}_6) & {\rm{if}}\;\; {\bf p}=(2,3,3),\\
D^b(\Mod_{f.d.}-\Bbbk \widetilde{E}_7) & {\rm{if}}\;\; {\bf p}=(2,3,4),\\
D^b(\Mod_{f.d.}-\Bbbk \widetilde{E}_8) & {\rm{if}}\;\; {\bf p}=(2,3,5).
\end{cases}
$$
\item[(3)]
Let $M$ be a generic simple object in $coh({\mathbb X})$.
Then $\Ext^1_{\mathbb X}(M,M)=1$.
\item[(4)]
$\fpdim^1 (coh({\mathbb X}))\geq 1$.
\item[(5)]
If ${\mathbb X}$ is tubular or domestic, then $\Ext^1_{\mathbb X}(X,Y)=0$ 
for all $X\in Vect_{\mu'}({\mathbb X})$ and $Y\in Vect_{\mu}({\mathbb X})$
with $\mu'< \mu$.
\item[(6)]
If ${\mathbb X}$ is domestic, then $\Ext^1_{\mathbb X}(X,Y)=0$ 
for all $X\in Vect_{\mu'}({\mathbb X})$ and $Y\in Vect_{\mu}({\mathbb X})$
with $\mu'\leq \mu$. As a consequence, $\fpdim(\Sigma\mid_{Vect_{\mu'}({\mathbb X})})=0$
for all $\mu<\infty$.
\item[(7)]
Suppose ${\mathbb X}$ is tubular. Then 
every indecomposable vector bundle ${\mathbb X}$ 
is semi-stable.
\item[(8)]
Suppose ${\mathbb X}$ is tubular and let $\mu\in {\mathbb Q}$. 
Then each $Vect_{\mu}({\mathbb X})$ is a uniserial category. 
Accordingly indecomposables in $Vect_{\mu}({\mathbb X})$ 
decomposes into Auslander-Reiten components, which all are 
tubes of finite rank. 
\end{enumerate}
\end{lemma}

\begin{proof} (1) This is well-known.

(2) \cite[5.4.1]{GL}.

(3) Let $M$ be a generic simple object. 
Then $M$ is a brick and $\Ext^1(M,M)=1$.

(4) Follows from (3) by taking $\phi:=\{M\}$.

(5) This is \cite[Corollary 4.34(i)]{Sc2} since tubular is also called
elliptic in \cite{Sc2}.

(6) This is \cite[Comments after Corollary 4.34]{Sc2} since domestic 
is also called parabolic in \cite{Sc2}. The consequence is clear.

(7) \cite[Theorem 5.6(i)]{GL}.

(8) \cite[Theorem 5.6(iii)]{GL}.
\end{proof}

We will use the following result which is proved in \cite{CG}.

\begin{theorem}\cite{CG}
\label{xxthm7.10} Let ${\mathbb X}$ be a weighted 
projective line.
\begin{enumerate}
\item[(1)]
If ${\mathbb X}$ is domestic, then $\fpdim\; D^b(coh({\mathbb X}))=1$.
\item[(2)]
If ${\mathbb X}$ is tubular, then $\fpdim\; D^b(coh({\mathbb X}))=1$.
\item[(3)]
If ${\mathbb X}$ is wild, then $\fpdim\; D^b(coh({\mathbb X}))
\geq \dim 
\Hom_{\mathbb X}({\mathcal O}_{\mathbb X},
{\mathcal O}_{\mathbb X}(\overrightarrow{\omega}))$ where $\overrightarrow{\omega}$
is the dualizing element \cite[Sec. 1.2]{GL}. 
\end{enumerate}
\end{theorem}

There is a similar statement for smooth complex projective curves
[Proposition \ref{xxpro6.5}].
The authors are interested in answering the following question.

\begin{question}
\label{xxque7.11}
Let ${\mathbb X}$  be a wild weighted projective line. What is 
the exact value of $\fpdim^n \; D^b(coh({\mathbb X}))$?
\end{question}

\subsection{Tubes}
The following example is studied in \cite{CG}, which is 
dependent on direct linear algebra calculations.

\begin{example}\cite{CG}
\label{xxex7.12}
Let $\xi$ be a primitive $n$th root of unity.
Let $T_n$ be the algebra 
$$T_n:=\frac{\Bbbk \langle g,x\rangle}{(g^n-1, xg-\xi gx)}.$$
This algebra can be expressed by using a group action.
Let $G$ be the group 
$$\{g\mid g^n=1\}\cong {\mathbb Z}/(n)$$ 
acting on the polynomial ring $\Bbbk[x]$ by $g\cdot x=\xi x$.
Then $T_n$ is naturally isomorphic to the skew group ring 
$\Bbbk[x]\ast G$. Let $\overrightarrow{A_{n-1}}$ denote the 
cycle quiver with $n$ vertices, namely, the quiver with one 
oriented cycle connecting $n$ vertices. It is also known that 
$T_n$ is isomorphic to the path algebra of the quiver 
$\overrightarrow{A_{n-1}}$. Then $\fpdim (\Mod_{f.d.}-T_n)=1$
by \cite{CG}.
\end{example}

\subsection{Proof of Theorem \ref{xxthm0.3}}
\label{xxsec7.4}

\begin{proof}[Proof of Theorem \ref{xxthm0.3}]
Part (1) follows from Theorems \ref{xxthm7.4}(1) and \ref{xxthm7.8}.
Part (3) follows from Proposition \ref{xxpro7.6}(2). It remains to deal with
part (2).

(2) By Theorem  \ref{xxthm7.4}(2), $Q$ must be of type either $\overrightarrow{A_{n-1}}$,
or $\widetilde{A}_{p, q}$, or $\widetilde{D}_n$, or $\widetilde{E}_{6,7,8}$. 
If $Q$ is of type $\overrightarrow{A_{n-1}}$, the assertion follows from Example 
\ref{xxex7.12}. If $Q$ is of type $\widetilde{A}_{p, q}$, 
$\widetilde{D}_n$, or $\widetilde{E}_{6,7,8}$, the assertion follows 
from Lemma \ref{xxlem7.9}(2) and Theorem \ref{xxthm7.10}(1).
\end{proof}

\section{Complexity}
\label{xxsec8}

The concept of complexity was first introduced by Alperin-Evens 
in the study of group cohomology in 1981 \cite{AE}. Since then
the study of complexity has been extended to finite dimensional 
algebras, Frobenius algebras, Hopf algebras and commutative algebras. 
First we recall the classical 
definition of the complexity for finite dimensional algebras and then 
give a definition of the complexity for triangulated categories. We 
give the following modified (but equivalent) version, which can 
be generalized.

\begin{definition}
\label{xxdef8.1}
Let $A$ be a finite dimensional algebra and $T=A/J(A)$
where $J(A)$ is the Jacobson radical of $A$. Let $M$ be a 
finite dimensional left $A$-module. 
\begin{enumerate}
\item[(1)]
The {\it complexity} of $M$ is defined to be
$$cx(M):=\limsup_{n\to \infty} \log_{n} (\dim \Ext^n_A(M,T))+1.
$$
\item[(2)]
The {\it complexity} of the algebra $A$ is defined to be
$$cx(A):=cx(T).$$
\end{enumerate}
\end{definition}

In the original definition of {\it complexity} by Alperin-Evens 
\cite{AE} and in most of other papers,  the dimension of $n$-syzygies
is used instead of the dimension of the $\Ext^n$-groups, but it is easy 
to see that the asymptotic behavior of these two series are the same, 
therefore these give rise to the same complexity. It is well-known that 
$cx(M)\leq cx(A)$ for all finite dimensional left $A$-modules $M$. Next 
we introduce the notion of a complexity for a triangulated category 
which is partially motivated by the work in \cite[Section 4]{BHZ}.

\begin{definition}
\label{xxdef8.2}
Let ${\mathcal T}$ be a pre-triangulated category. Let $d$ be a real 
number.
\begin{enumerate}
\item[(1)]
The left subcategory of complexity less than $d$ is defined to
be
$$_d{\mathcal T}:=\left\{ X\in {\mathcal T} \mid 
\lim_{n\to \infty}
\frac{\dim \Hom_{\mathcal T}(X, \Sigma^n (Y))}{n^{d-1}} =0, 
\forall\; Y\in {\mathcal T}\right\}.
$$
\item[(2)]
The right subcategory of complexity less than $d$ is defined to
be
$${\mathcal T}_d:=\left\{ X\in {\mathcal T} \mid 
\lim_{n\to \infty}
\frac{\dim \Hom_{\mathcal T}(Y, \Sigma^n (X))}{n^{d-1}} =0, 
\forall\; Y\in {\mathcal T}\right\}.
$$
\item[(3)]
The {\it complexity} of ${\mathcal T}$ is defined to 
be
$$cx({\mathcal T}):=\inf \left\{ d \mid {_d{\mathcal T}}={\mathcal T}\right\}
$$
\item[(4)]
The {\it Frobenius-Perron complexity} of ${\mathcal T}$ is defined to 
be
$$\fpc({\mathcal T}):=\fpg (\Sigma)+1.$$
\end{enumerate}
\end{definition}

Note that it is not hard to show that
$cx({\mathcal T})=\inf \left\{ d \mid {{\mathcal T}_d}={\mathcal T}\right\}.$

\begin{theorem}
\label{xxthm8.3} Let ${\mathcal T}$ be a pre-triangulated category. 
Then $\fpc({\mathcal T})\leq cx({\mathcal T})$.
\end{theorem}

\begin{proof}
Let $d$ be any number strictly larger than $cx({\mathcal T})$.
We need to show that $\fpc({\mathcal T})\leq d$.

Let $\phi\in \Phi_{m,a}$ be an atomic set and let $X:=
\bigoplus_{X_i\in \phi} X_i$. Then, by definition,
$$\lim_{n\to \infty}
\frac{\dim \Hom_{\mathcal T}(X, \Sigma^n (X))}{n^{d-1}} =0.$$
Then there is a constant $C$ such that
$\dim \Hom_{\mathcal T}(X, \Sigma^n (X))< C n^{d-1}$ for all
$n>0$. Since each $X_i$ is a direct summand of $X$, we have 
$$a_{ij}(n):=\dim \Hom_{\mathcal T}(X_i, \Sigma^n (X_j))< C n^{d-1}$$
for all $i,j$. This means that each entry $a_{ij}(n)$ 
in the adjacency matrix
of $A(\phi, \Sigma^n)$ is less than $Cn^{d-1}$. Therefore
$\rho(A(\phi, \Sigma^n))< m C n^{d-1}$. By Definition 
\ref{xxdef2.3}(3), $\fpg(\Sigma)\leq d-1$. Thus $\fpc({\mathcal T})
\leq d$ as desired.
\end{proof}

We will prove that the equality 
$\fpc({\mathcal T})=cx({\mathcal T})$ holds 
under some extra hypotheses.
Let $A$ be a finite dimensional algebra with a complete list of simple
left $A$-modules $\{S_1,\cdots, S_w\}$. We use $n$ for any integer and $i,j$
for integers between 1 and $w$. Define, for $i\leq j$,
$$p_{ij}(n):=\min\{ \dim \Ext^n_A(S_i,S_j), \dim \Ext^n_A(S_j,S_i)\}$$
and 
$$P_n:=\max\{p_{ij}(n) \mid i\leq j\}.$$
We say $A$ satisfies {\it averaging growth condition} (or {\it AGC} for short) 
if there are positive integers $C$ and $d$,
independent of the choices of $n$ and $(i,j)$,
such that
\begin{equation}
\label{E8.3.1}\tag{E8.3.1}
\dim \Ext^n_{A}(S_i,S_j)\leq C \max\{P_{n-d}, P_{n-d+1},\cdots, P_{n+d}\}
\end{equation}
for all $n$ and all $1\leq i,j\leq w$.

\begin{theorem}
\label{xxthm8.4} 
Let $A$ be a finite dimensional algebra and ${\mathcal A}=D^b(\Mod_{f.d.}-A)$.
\begin{enumerate}
\item[(1)]
$cx(A)=cx({\mathcal A})$. As a consequence, $cx(A)$ is a derived invariant.
\item[(2)]
If $A$ satisfies AGC, then $\fpc({\mathcal A})= cx({\mathcal A})=cx(A)$.
As a consequence, if $A$ is local or commutative, then 
$\fpc({\mathcal A})= cx({\mathcal A})=cx(A)$.
\end{enumerate}
\end{theorem}

We will prove Theorem \ref{xxthm8.4} after the next lemma.

Let ${\mathcal T}$ be a pre-triangulated category with suspension
$\Sigma$. We use  $X, Y, Z$ for objects in ${\mathcal T}$.
Fix a family $\phi$ of objects in ${\mathcal T}$ and a 
positive number $d$. Define

\begin{align}
\label{E8.4.1}\tag{E8.4.1}
{_d(\phi)}&=
\left\{X\in {\mathcal T}\mid \lim_{n\to \infty} \frac{ \dim 
\Hom_{\mathcal T}(X, \Sigma^n(Y))}{n^{d-1}}=0, \forall \; Y\in \phi\right\},\\
\label{E8.4.2}\tag{E8.4.2}
(\phi)_d&=
\left\{ X\in {\mathcal T}\mid \lim_{n\to \infty} \frac{ \dim 
\Hom_{\mathcal T}(Y, \Sigma^n(X))}{n^{d-1}}=0, \forall \; Y\in \phi\right\},\\
\label{E8.4.3}\tag{E8.4.3}
{^d(\phi)}&=
\left\{ X\in {\mathcal T}\mid \lim_{n\to \infty} \frac{ \sum_{i\leq n}\dim 
\Hom_{\mathcal T}(X, \Sigma^i(Y))}{n^d}=0, \forall \; Y\in \phi\right\},\\
\label{E8.4.4}\tag{E8.4.4}
(\phi)^d&=
\left\{ X\in {\mathcal T}\mid \lim_{n\to \infty} \frac{ \sum_{i\leq n}\dim 
\Hom_{\mathcal T}(Y, \Sigma^i(X))}{n^d}=0, \forall \; Y\in \phi\right\}.
\end{align}

\begin{lemma}
\label{xxlem8.5} 
The following are
full thick pre-triangulated subcategories of ${\mathcal T}$
closed under direct summands.
\begin{enumerate}
\item[(1)]
${_d(\phi)}$.
\item[(2)]
$(\phi)_d$.
\item[(3)]
${^d(\phi)}$.
\item[(4)]
$(\phi)^d$.
\end{enumerate}
\end{lemma}

\begin{proof} We only prove (1). The proofs of other parts are similar.
Suppose $X\in {_d(\phi)}$. Using the fact $\lim_{n\to\infty} \frac{n^{d-1}}{(n+1)^{d-1}}=1$,
we see that $X[1]=\Sigma(X)$ is in ${_d(\phi)}$. Similarly,
$X[-1]$ is in ${_d(\phi)}$. If $f: X_1\to X_2$ be a morphism of objects in
${_d(\phi)}$, and let
$X_3$ be the mapping cone of $f$, then, for each $Y\in \phi$,  we have an exact sequence
$$\to \Hom_{\mathcal T}(X_1, \Sigma^{n-1}(Y))\to \Hom_{\mathcal T}(X_3, \Sigma^{n}(Y))\to 
\Hom_{\mathcal T}(X_2, \Sigma^{n}(Y))\to$$
which implies that $X_3\in {_d(\phi)}$. Therefore ${_d(\phi)}$ is
a thick pre-triangulated subcategory of ${\mathcal T}$. If $X\in {_d(\phi)}$
and $X=Y\oplus Z$, it is clear that $Y,Z\in {_d(\phi)}$. 
Therefore ${_d(\phi)}$ is closed under 
taking direct summands.
\end{proof}

\begin{proof}[Proof of Theorem \ref{xxthm8.4}]
(1) Let $c=cx(A)$. For every $d<c$, we have that
$$\limsup_{n\to \infty} \frac{\dim \Ext^n_{A}(T,T)}{n^{d-1}}=\infty$$
which implies that $T\not\in {_d {\mathcal A}}$. Therefore 
$d\leq cx({\mathcal A})$. 

Conversely, let $d>c$. It follows from the definition that 
$$\limsup_{n\to \infty} \frac{\dim \Ext^n_{A}(T,T)}{n^{d-1}}=0.$$
This means that $T\in (\{T\})_d$. Since $T$ generates ${\mathcal A}$,
we have ${\mathcal A}=(\{T\})_d$. Again, since $T$ generates ${\mathcal A}$, 
we have  ${\mathcal A}={\mathcal A}_d={_d{\mathcal A}}$. By definition,
$d\geq cx({\mathcal A})$ as desired.

(2) Assume that $A$ satisfies AGC. Let
$$\begin{aligned}
c_1&=\fpc({\mathcal A}),\\
c_2&=\limsup_{n\to \infty}
\log_{n}(C \max\{P_{n-d}, P_{n-d+1},\cdots, P_{n+d}\}) +1,\\
c_3&=\limsup_{n\to \infty}
\log_{n}(P_{n}) +1,\\
c_4&=cx(A)=cx({\mathcal A}).
\end{aligned}
$$

By calculus, we have $c_2=c_3$. 
Let $\phi$ be the atomic set of simple objects $\{S_i\}_{i=1}^{w}$.
Then $\rho(\phi, \Sigma^n)\geq p_{ij}(n)$, for all $i,j$, by 
Lemma \ref{xxlem1.7}(2). So $\rho(\phi, \Sigma^n)
\geq P_n$. As a consequence, $c_1\geq c_3$. 
Let $T=A/J=\bigoplus_{i=1}^w S_i^{d_i}$ for some 
finite numbers $\{d_i\}_{i=1}^w$. Let $D$ be $\max_i\{d_i\}$.
By AGC, namely, \eqref{E8.3.1},
$$\begin{aligned}
\dim \Ext^n_A(T, T)&=\sum_{i,j} d_i d_j \dim \Ext^n_A(S_i,S_j)\\
&\leq w^2 D C \max\{P_{n-d}, P_{n-d+1},\cdots, P_{n+d}\}
\end{aligned}
$$
which implies that $c_4=cx(A)=cx(T)\leq c_2$. Combining
with Theorem \ref{xxthm8.3}, we have $c_1=c_2=c_3=c_4$
as desired.

If $A$ is local, then there is only one simple module
$S_1$. Then \eqref{E8.3.1} is automatic. If $A$ is 
commutative, then $\Ext^i_A(S_i,S_j)=0$ for all $n$
and all $i\neq j$. Again, in this case, \eqref{E8.3.1}
is obvious. The consequence follows from the main 
assertion.
\end{proof}

For all well-studied finite dimensional algebras $A$, 
\eqref{E8.3.1} holds. For example, the algebra $A$ in 
Example \ref{xxex5.5} satisfies AGC. This can be shown by using 
the computation given in Remark \ref{xxrem5.13}(3).
It is natural to ask if every finite dimensional 
algebra  satisfies AGC.

\begin{lemma}
\label{xxlem8.6}
\begin{enumerate}
\item[(1)]
Let ${\mathfrak A}$ be an abelian category and ${\mathcal A}=D^b({\mathfrak A})$.
If $\gldim {\mathfrak A}<\infty$, then $\fpc({\mathcal A})=0$.
\item[(2)]
Let ${\mathcal T}$ be a pre-triangulated category.
If $\fpgldim {\mathcal T}<\infty$, then $\fpc({\mathcal T})=0$.
\end{enumerate}
\end{lemma}

\begin{proof} Both are easy and proofs are omitted. 
\end{proof}

We conclude with examples of non-integral $\fpg$ of a 
triangulated category.

\begin{example}
\label{xxex8.7}
(1) Let $\alpha$ be any real number in $\{0\}\cup 
\{1\}\cup [2,\infty)$. By \cite[Theorem 1.8, or p. 14]{KL},
there is a finitely generated algebra $R$ with $\GKdim R=\alpha$. More precisely,
\cite[Theorem 1.8]{KL} implies that 
there is a 2-dimensional vector space $V\subset R$ that generates $R$ such that,
there are positive integers $a<b$, for every $n>0$,
\begin{equation}
\notag
a n^{\alpha} < \dim (\Bbbk 1+ V)^n < b n^{\alpha}.
\end{equation}
Define a filtration ${\mathcal F}$ on $R$ by
$$F_i R=(\Bbbk 1+V)^i, \quad \forall \; i.$$
Let $A$ be the associated graded algebra $\gr R$
with respect to this grading. Then $A$ is connected graded and 
generated by two elements in degree 1 and satisfying, for every $n>0$,
\begin{equation}
\label{E8.7.1}\tag{E8.7.1}
a n^{\alpha} < \sum_{i=0}^n \dim A_i < b n^{\alpha}.
\end{equation}
To match up with the definition of complexity, we further assume that
there are $c<d$ such that, for every $n>0$,
\begin{equation}
\label{E8.7.2}\tag{E8.7.2}
c n^{\alpha-1} < \dim A_n < d n^{\alpha-1}.
\end{equation}
This can be achieved, for example, by replacing $A$ by its polynomial 
extension $A[t]$ (with $\deg t=1$) and replacing $\alpha$ by $\alpha+1$.

Next we make $A$ a differential graded (dg) algebra by setting elements
in $A_i$ to have cohomological degree $i$ and $d_{A}=0$. For this 
dg algebra, we denote the derived category of left dg $A$-modules
by ${\mathcal A}$. Let ${\mathcal O}$ be the object $_AA$ in 
${\mathcal A}$. By the definition of the cohomological degree of $A$, we have
\begin{equation}
\label{E8.7.3}\tag{E8.7.3}\Hom_{\mathcal A}({\mathcal O}, \Sigma^{i} {\mathcal O})=A_i, 
\quad \forall \; i.
\end{equation}

Let ${\mathcal T}$ be the full triangulated subcategory of ${\mathcal A}$
generated by ${\mathcal O}$. \eqref{E8.7.3} implies that ${\mathcal O}$ 
is an atomic object. Now using \eqref{E8.7.3} together with \eqref{E8.7.2}, 
we obtain that 
\begin{equation}
\label{E8.7.4}\tag{E8.7.4}
\fpc({\mathcal T}) \geq \alpha.
\end{equation}

By \eqref{E8.7.2}-\eqref{E8.7.3}, we have
that, for every $d>\alpha$,
${\mathcal O}\in {_{d} (\{\mathcal O\})}$. Since ${\mathcal O}$
generates ${\mathcal T}$, we have ${_{d} (\{\mathcal O\})}={\mathcal T}$.
The last equation means that ${\mathcal O}\in ({\mathcal T})_{d}$.
Since ${\mathcal O}$
generates ${\mathcal T}$, we have $({\mathcal T})_{d}={\mathcal T}$.
By definition, $d>cx({\mathcal T})$. Combining these with Theorem
\ref{xxthm8.3} and \eqref{E8.7.4}, we have, for every $d>\alpha$,
$$\alpha \leq \fpc({\mathcal T})\leq cx({\mathcal T})<d$$
which implies that
$\fpc({\mathcal T})=cx({\mathcal T})=\alpha$. This construction 
implies that 
\begin{equation}
\label{E8.7.5}\tag{E8.7.5}
\GKdim \left( \bigoplus_{i=0}^{\infty} 
\Hom_{\mathcal T}({\mathcal O}, \Sigma^i({\mathcal O}))\right)
=\GKdim A=\alpha.
\end{equation}

(2) We now consider an extreme case. Let $a:=\{a_i\}_{i=0}^{\infty}$
be any sequence of non-negative integers with $a_0=1$. 
Define $B$ to be the dg algebra $\bigoplus B_i$ 
such that
\begin{enumerate}
\item[(i)]
$\dim B_i=a_i$ for all $i$. In particular, $B_0=\Bbbk$.
Elements in $B_i$ have cohomological degree $i$.
\item[(ii)]
$(\bigoplus_{i>0} B_i)^2=0$.
\item[(iii)]
Differential $d_B=0$.
\end{enumerate}
In this case, $\GKdim B=0$.
Similar to part (1), the derived category of left dg $B$-modules
is denoted by ${\mathcal B}$. Let ${\mathcal O}$ be the object $_BB$ in 
${\mathcal B}$. Then 
$$\Hom_{\mathcal B}
({\mathcal O}, \Sigma^{i} {\mathcal O})=B_i, \quad \forall \; i,$$
and ${\mathcal O}$ is an atomic object. 
Let ${\mathcal T}$ be the full triangulated subcategory of ${\mathcal B}$
generated by ${\mathcal O}$. The argument in part (1) shows that
$$\fpc({\mathcal T})=\limsup_{n\to\infty} \log_{n}( a_n) +1.$$
Now let $r$ be any real number $\geq 1$ and 
let $a_i=\begin{cases} 1 &i=0\\
\lfloor i^{r-1} \rfloor & i\geq 1\end{cases}$. Then we have $\fpc({\mathcal T})=r$.
Let $r$ be any real number $\geq 1$ and $a_i=\lfloor r^i \rfloor$ for all $i\geq 0$. 
Then $\fpc({\mathcal T})=\begin{cases} 1& r=1\\ \infty & r>1\end{cases}$. 
Using a similar method (with details omitted),
$\fpv({\mathcal T})=r$.
\end{example}

\bigskip

\subsection*{Acknowledgments}
The authors would like to thank Klaus Bongartz, Christof Geiss, 
Ken Goodearl, Claus Michael Ringel, and Birge Huisgen-Zimmermann 
for many useful conversations on the subject, and thank Max 
Lieblich for the proof of Proposition \ref{xxpro6.5}(3). J. Chen 
was partially supported by the National Natural Science Foundation 
of China (Grant No. 11571286) and the Natural Science Foundation 
of Fujian Province of China (Grant No. 2016J01031). Z. Gao was 
partially supported by the National Natural Science Foundation of 
China (Grant No. 61401381). E. Wicks and J.J. Zhang were partially 
supported by the US National Science Foundation (Grant Nos. 
DMS-1402863 and DMS-1700825). X.-H. Zhang was partially supported 
by the National Natural Science Foundation of China (Grant No. 11401328).
H. Zhu was partially supported by a grant from Jiangsu overseas 
Research and Training Program for university prominent young and 
middle-aged Teachers and Presidents, China.



\begin{thebibliography}{99}


\bibitem[AE]{AE} 
J. Alperin and L. Evens, 
\emph{Representations, resolutions, and Quillen's 
dimension theorem}, 
J. Pure Appl. Algebra {\bf 22} (1981), 1-9.

\bibitem[Ar]{Ar}
S. Ariki, 
\emph{Hecke algebras of classical type and their 
representation type}, 
Proc. London Math. Soc. (3) {\bf 91} (2005), no. 2, 355--413.


\bibitem[AZ]{AZ}
M. Artin, J.J. Zhang,
\emph{Noncommutative projective schemes},
Adv. Math. {\bf 109} (1994), 228--287.

\bibitem[AS]{AS}
I. Assem, D. Simson and A. Skowroński, 
\emph{Elements of the representation theory of associative algebras}, 
Vol. 1. Techniques of representation theory. London Mathematical 
Society Student Texts, 65. Cambridge University Press, Cambridge, 2006. 

\bibitem[ARS]{ARS}
M. Auslander, I. Reiten and S.O.  Smal{\/o}, 
\emph{Representation theory of Artin algebras}, 
Corrected reprint of the 1995 original. 
Cambridge Studies in Advanced Mathematics, 
{\bf 36} Cambridge University Press, Cambridge, 1997.

\bibitem[Av]{Av}
L.L. Avramov, 
\emph{Infinite free resolutions}, 
in: Six Lectures on Commutative Algebra, in: Prog. Math., 
vol. {\bf 166}, Birkh{\"a}user Verlag, Basel, Boston, Berlin, 1998, 
1-118.


\bibitem[BHZ]{BHZ}
Y.-H. Bao, J.-W. He and J.J. Zhang,
\emph{Pertinency of Hopf actions and quotient 
categories of Cohen-Macaulay algebras},
J. Noncommut. Geom. (accepted for publication).



\bibitem[BZ]{BZ}
J. Bell and J.J. Zhang, 
\emph{An isomorphism lemma for graded rings}, 
Proc. Amer. Math. Soc., {\bf 145} (2017), no. 3, 989--994




\bibitem[Bei]{Bei}
A. A. Be{\v i}linson, 
\emph{Coherent sheaves on ${\mathbb P}^n$
and problems in linear algebra},
Funktsional. Anal. i Prilozhen. {\bf 12}(3) (1978), 68--69.


\bibitem[Ben]{Ben}
D.J. Benson, 
\emph{Representations and cohomology I}. 
Basic representation theory of finite groups and associative algebras. 
Cambridge Studies in Advanced Mathematics, 30. Cambridge University 
Press, Cambridge, 1991. 


\bibitem[BS]{BS}
P.A. Bergh and {\O}. Solberg,
\emph{Relative support varieties}, 
Q. J. Math. {\bf 61} (2010), no. 2, 171--182.


\bibitem[BGP]{BGP}
I.N. Bernstein, I.M. Gel'fand and V.A. Ponomarev, 
\emph{Coxeter functors and Gabriel's theorem}, 
Russ. Math. Surv. {\bf 28} (1973), no. 2, 17--32 (English).


\bibitem[BO1]{BO1}
A. Bondal and D. Orlov, 
\emph{Reconstruction of a variety from the derived category and 
groups of autoequivalences},
Compositio Math. {\bf 125} (2001), no. 3, 327--344.

\bibitem[BO2]{BO2}
A. Bondal and D. Orlov, 
\emph{Derived categories of coherent sheaves}, 
In Proceedings of the International Congress of Mathematicians, Vol. II 
(Beijing, 2002), pages 47--56. Higher Ed. Press, Beijing, 2002. 


\bibitem[B1]{B1}
K. Bongartz, 
priviate communications.

\bibitem[B2]{B2}
K. Bongartz, 
Representation embeddings and the second
Brauer-Thrall conjecture, preprint, (2017),
arXiv:1611.02017.


\bibitem[Bre]{Bre}
S. Brenner,
\emph{A combinatorial characterization of finite Auslander-Reiten quivers}, 
Proceedings ICRA 4, Ottawa 1984, Lecture Notes in Mathematics {\bf 1177} 
(Springer, Berlin, 1986), pp. 13--49.

\bibitem[Bri]{Bri}
T. Bridgeland, 
\emph{Flops and derived categories}, 
Invent. Math., {\bf 147}(3), (2002), 613--632.

\bibitem[BKR]{BKR}
T. Bridgeland, A. King, and M. Reid, 
\emph{The MacKay correspondence as an equivalence of derived
categories}, 
J. Amer. Math. Soc., {\bf 14}(3) (2001), 535--554.

\bibitem[BB]{BB}
K. Br{\"u}ning and I. Burban, 
\emph{Coherent sheaves on an elliptic curve}, (English summary) 
Interactions between homotopy theory and algebra, 297--315, 
Contemp. Math., 436, Amer. Math. Soc., Providence, RI, 2007. 


\bibitem[Ca]{Ca}
J.F. Carlson,
\emph{The decomposition of the
trivial module in the complexity quotient category},
J. Pure Appl. Algebra {\bf 106} (1996), 23--44.

\bibitem[CDW]{CDW}
J.F. Carlson, P. W. Donovan and W. W. Wheeler,
\emph{Complexity and quotient categories for group algebras},
J. Pure Appl. Algebra {\bf 93} (1994), 147--167.

\bibitem[CC]{CC}
A.T. Carroll and C. Chindris, 
\emph{Moduli spaces of modules of Schur-tame algebras}, (English summary) 
Algebr. Represent. Theory {\bf 18} (2015), no. 4, 961--976. 

\bibitem[Ch]{Ch}
J.A. Chen, 
\emph{On genera of smooth curves in higher-dimensional varieties}, 
Proc. Amer. Math. Soc. {\bf 125} (1997), no. 8, 2221--2225.

\bibitem[CG]{CG}
J.M. Chen,  Z.B. Gao, E. Wicks, J. J. Zhang, X-.H. Zhang and H. Zhu, 
\emph{Frobenius-Perron theory for projective schemes}, ArXiv e-prints, (2019). ArXiv:1907.02221.


\bibitem[CKW]{CKW}
C. Chindris, R. Kinser and J. Weyman, 
\emph{Module varieties and representation type of finite-dimensional algebras}, 
Int. Math. Res. Not. IMRN (2015), no. {\bf 3}, 631--650.

\bibitem[CFZ]{CFZ}
C. Ciliberto, F. Flamini and M. Zaidenberg, 
\emph{Gaps for geometric genera}, 
Arch. Math. (Basel) {\bf 106} (2016), no. 6, 531--541.


\bibitem[Cr]{Cr}
W. Crawley-Boevey, 
\emph{Lectures on representations of quivers}, Lectures in Oxford in 1992, 
available at
www.amsta.leeds.ac.uk/~pmtwc/quivlecs.pdf.


\bibitem[DoG]{DoG}
M.A. Dokuchaev, N.M. Gubareni, V.M. Futorny,  M.A. Khibina and V.V. Kirichenko,
Dynkin diagrams and spectra of graphs,
S{\~a}o Paulo J. Math. Sci., {\bf 7} (2013), no. 1, 83--104. 

\bibitem[DoF]{DoF}
P. Donovan and M.R. Freislich, 
\emph{The representation theory of finite graphs and associated algebras}, 
Number 5 in Carleton Mathematical Lecture Notes. Carleton University, 
Ottawa, Ont., 1973.


\bibitem[Dr1]{Dr1}
J.A. Drozd, 
{\it Tame and wild matrix problems}, 
Representation theory, II (Proc. Second Internat. Conf., Carleton Univ., 
Ottawa, Ont., 1979), Lecture Notes in Math., vol. 832, Springer, Berlin, 
1980, pp. 242--258.

\bibitem[Dr2]{Dr2}
Y.A. Drozd, 
\emph{Derived tame and derived wild algebras}, 
Algebra Discrete Math. (2004), no. 1, 57--74. 

\bibitem[DrG]{DrG}
Y.A. Drozd and G.-M. Greuel,
\emph{Tame and wild projective curves and classification of vector bundles,} 
J. Algebra {\bf 246} (2001), no. 1, 1--54.

\bibitem[ES]{ES}
K. Erdmann and {\O}. Solberg, 
\emph{Radical cube zero selfinjective algebras of finite complexity}, 
J. Pure Appl. Algebra {\bf 215} (2011), no. 7, 1747--1768.



\bibitem[EG]{EG}
P. Etingof, S. Gelaki,  D. Nikshych and V. Ostrik, 
Tensor categories, 
Mathematical Surveys and Monographs, {\bf 205}. 
American Mathematical Society, Providence, RI, 2015.

\bibitem[EGO]{EGO}
P. Etingof, S.  Gelaki and V. Ostrik, 
\emph{Classification of fusion categories of dimension $pq$}, 
Int. Math. Res. Not., 2004, no. {\bf 57}, 3041--3056. 

\bibitem[ENO]{ENO}
P. Etingof, D. Nikshych and V. Ostrik,
\emph{On fusion categories},
Ann. of Math., (2) {\bf 162} (2005), no. 2, 581--642. 

\bibitem[Fa]{Fa}
R. Farnsteiner, 
\emph{Tameness and complexity of finite group schemes}, 
Bulletin of the London Mathematical Society,
{\bf 39} (2007) no. 1, 63--70.

\bibitem[FW]{FW}
J. Feldvoss and S. Witherspoon,
\emph{Support varieties and representation type of self-injective algebras}, 
Homology Homotopy Appl. {\bf 13} (2011), no. 2, 197--215. 

\bibitem[Ga1]{Ga1}
P. Gabriel, 
\emph{Unzerlegbare Darstellungen. I}, (German. English summary) 
Manuscripta Math. {\bf 6} (1972), 71--103; correction, ibid. 
{\bf 6} (1972), 309.

\bibitem[Ga2]{Ga2}
P. Gabriel, 
\emph{Repr{\'e}sentations ind{\'e}composables}, (French) 
S{\'e}minaire Bourbaki, 26e ann{\'e}e (1973/1974), 
Exp. No. 444, pp. 143--169. Lecture Notes in Math., Vol. 431, 
Springer, Berlin, 1975. 

\bibitem[Ga3]{Ga3}
P. Gabriel,  
\emph{Auslander-Reiten sequences and representation-finite algebras}, 
Representation theory, I (Proc. Workshop, Carleton Univ., Ottawa, Ont., 1979), pp. 1--71, 
Lecture Notes in Math., 831, Springer, Berlin, 1980. 

\bibitem[GL]{GL}
W. Geigle and H. Lenzing, 
\emph{A class of weighted projective curves arising in 
representation theory of finite-dimensional algebras}, 
Singularities, representation of algebras, and vector bundles 
(Lambrecht, 1985), 265--297, Lecture Notes in Math., 1273, 
Springer, Berlin, 1987.


\bibitem[GKO]{GKO}
C. Geiss, B. Keller and S. Oppermann, 
\emph{$n$–angulated categories}, J. Reine Angew. Math.
{\bf 675} (2013) 101--120. 

\bibitem[GKr]{GKr}
C. Geiss and H. Krause, 
\emph{On the notion of derived tameness}, 
(English summary) 
J. Algebra Appl. {\bf 1} (2002), no. 2, 13--157. 


\bibitem[GLW]{GLW}
J.-Y. Guo, A. Li, Q. Wu,
\emph{Selfinjective Koszul algebras of finite complexity},
Acta Math. Sinica, English Series {\bf 25} (2009), 2179--2198.

\bibitem[HPR]{HPR}
D. Happel, U. Preiser, and C.M. Ringel,
Binary polyhedral groups and Euclidean diagrams.
{\it Manuscripta Math.}, {\bf 31} (1-3), (1980), 317--329.


\bibitem[HL]{HL}
D. Huybrechts and M. Lehn, 
\emph{The geometry of moduli spaces of sheaves}, 
Second edition. Cambridge Mathematical Library. 
Cambridge University Press, Cambridge, 2010. 

\bibitem[Ke1]{Ke1}
B. Keller, 
\emph{Derived categories and tilting}, 
Handbook of tilting theory, 49--104, 
London Math. Soc. Lecture Note Ser., {\bf 332}, 
Cambridge Univ. Press, Cambridge, 2007. 

\bibitem[Ke2]{Ke2}
B. Keller, 
\emph{Calabi-Yau triangulated categories.} 
Trends in representation theory of algebras and related topics, 
467--489, EMS Ser. Congr. Rep., Eur. Math. Soc., Zürich, 2008.

\bibitem[KV]{KV}
B. Keller and D. Vossieck
\emph{Sous les cat{\'e}gories dériv{\'e}es} 
(French) [Beneath the derived categories] 
C. R. Acad. Sci. Paris S{\'e}r. I Math. 
{\bf 305} (1987), no. 6, 225--228. 


\bibitem[KL]{KL}
G.R. Krause and T.H. Lenagan,
\emph{Growth of algebras and Gelfand-Kirillov dimension},
Research Notes in Mathematics, Pitman Adv. Publ. Program,
{\bf 116} (1985).

\bibitem[Kul]{Kul}
J. K{\"u}lshammer, 
\emph{Representation type of Frobenius-Lusztig kernels}, 
Q. J. Math. {\bf 64} (2013), no. 2, 471--488. 
\emph{Corrigendum "Representation type of Frobenius-Lusztig kernels"}
Q. J. Math. {\bf 66} (2015), no. 4, 1139.

\bibitem[LM]{LM}
H. Lenzing, H. Meltzer, 
\emph{Sheaves on a weighted projective line of genus one, and representations
of a tubular algebra}, Representations of algebras (Ottawa, Canada, 1992), CMS Conf.
Proc. {\bf 14}, 313--337, (1994).


\bibitem[MR]{MR}
J.C. McConnell and J.C . Robson,
``Noncommutative Noetherian Rings,'' Wiley, Chichester, 1987.


\bibitem[Ni]{Ni}
D. Nikshych,
\emph{Semisimple weak Hopf algebras},
J. Algebra, {\bf 275} (2004), no. 2, 639--667. 

\bibitem[Na]{Na}
L.A. Nazarova, 
\emph{Representations of quivers of infinite type}, 
Izvestiya Akademii Nauk SSSR. Seriya Matematicheskaya, 
{\bf 37} (1973), 752--791.

\bibitem[RVdB]{RVdB}
I. Reiten and M. Van den Bergh, 
\emph{Noetherian hereditary abelian categories satisfying Serre duality}, 
J. Amer. Math. Soc. {\bf 15} (2002), no. 2, 295--366. 

\bibitem[Ri]{Ri}
J. Rickard,
\emph{The representation type of self-injective algebras}, 
Bull. London Math. Soc. {\bf 22} (1990), 540--546.


\bibitem[RV]{RV}
C.M. Ringel and D. Vossieck, 
\emph{Hammocks}, 
Proc. London Math. Soc. (3) {\bf 54} (1987), no. 2, 216--246. 


\bibitem[Sc1]{Sc1}
R. Schiffler, 
\emph{Quiver representations},
CMS Books in Mathematics/Ouvrages de Math{\'e}matiques de la SMC. 
Springer, Cham, 2014. 


\bibitem[Sc2]{Sc2}
O. Schiffmann,
\emph{Lectures on Hall algebras},
Geometric methods in representation theory. II, 1--141, 
S{\'e}min. Congr., 24-II, Soc. Math. France, Paris, 2012. 

\bibitem[Sm]{Sm}
J.H. Smith, 
Some properties of the spectrum of a graph, in: Combinatorial Structures
and their Applications (Eds. R. Guy, H. Hanani, N. Sauer, J. Sch{\"o}nheim),
Gordon and Breach (New York), 1970, 403--406.


\bibitem[VdB]{VdB}
M. Van den Bergh,
Three-dimensional flops and noncommutative rings,
\emph{Duke Math. J.} {\bf 122}(3), (2004), 423--455.

\bibitem[vR1]{vR1}
A.-C. van Roosmalen, 
\emph{Abelian 1-Calabi-Yau categories}, 
Int. Math. Res. Not. IMRN 2008, no. 6, Art. ID rnn003, 20pp.

\bibitem[vR2]{vR2}
A.-C. van Roosmalen, 
\emph{Numerically finite hereditary categories with Serre duality}, 
Trans. Amer. Math. Soc. {\bf 368} (2016), no. 10, 7189--7238. 

\bibitem[Ye]{Ye}
A. Yekutieli, 
\emph{Dualizing complexes, Morita equivalence and the derived Picard group of a ring}, 
J. London Math. Soc. (2) {\bf 60} (1999), no. 3, 723--746.

\end{thebibliography}

\end{document}